\documentclass[11pt]{article}

\usepackage{geometry}
\usepackage{hyperref} 

\usepackage{amsmath}
\usepackage{amssymb}
\usepackage{amscd}
\usepackage{textcomp}
\usepackage{dsfont}
\usepackage{mathrsfs}

\long\def\forget#1{}

\usepackage{color}
\newcounter{commentcounter}

\def\?{\ 
{\bf\color{red}???}\ 
\immediate\write16{}
\immediate\write16{Warning: There was still a question mark . . . }
\immediate\write16{}}

\usepackage{amsthm}
\theoremstyle{plain}
\newtheorem{theorem}{Theorem}[section]

\newtheorem{lemma}[theorem]{Lemma}

\newtheorem{corollary}[theorem]{Corollary}
\newtheorem{proposition}[theorem]{Proposition}

\theoremstyle{definition}
\newtheorem{definition}[theorem]{Definition}
\newtheorem{definition-theorem}[theorem]{Definition-Theorem}
\newtheorem{definition-remark}[theorem]{Definition-Remark}

\newtheorem{example}[theorem]{Example}
\newtheorem{remark}[theorem]{Remark}

\theoremstyle{remark}

\usepackage{xy}
\xyoption{all}

\newdir^{ (}{{}*!/-3pt/\dir^{(}}    
\newdir^{  }{{}*!/-3pt/\dir^{}}    
\newdir_{ (}{{}*!/-3pt/\dir_{(}}    
\newdir_{  }{{}*!/-3pt/\dir_{}}

\newcounter{zahl}

\def\theenumi{(\alph{enumi})}

\def\p@enumii{\theenumi}

\newcommand{\DS}{\displaystyle}
\newcommand{\TS}{\textstyle}
\newcommand{\SC}{\scriptstyle}
\newcommand{\SSC}{\scriptscriptstyle}

\newcommand{\cG}{\mathcal{G}}

\newcommand{\cK}{\mathcal{K}}

\newcommand{\cO}{\mathcal{O}}

\DeclareMathOperator{\Aut}{Aut}

\DeclareMathOperator{\End}{End}

\DeclareMathOperator{\Frob}{Frob}
\DeclareMathOperator{\Gal}{Gal}
\DeclareMathOperator{\GL}{GL}

\DeclareMathOperator{\Koh}{H}
\DeclareMathOperator{\CKoh}{\check H}
\DeclareMathOperator{\Hom}{Hom}

\DeclareMathOperator{\Ind}{Ind}
\DeclareMathOperator{\Int}{Int}
\DeclareMathOperator{\Isom}{Isom}

\DeclareMathOperator{\PGL}{PGL}

\DeclareMathOperator{\Quot}{Frac}

\DeclareMathOperator{\Rep}{Rep}
\DeclareMathOperator{\Res}{Res}

\DeclareMathOperator{\SL}{SL}
\DeclareMathOperator{\Spm}{Sp}
\DeclareMathOperator{\Spec}{Spec}
\DeclareMathOperator{\Spf}{Spf}

\DeclareMathOperator{\Var}{V}

\newcommand{\alg}{{\rm alg}}

\DeclareMathOperator{\charakt}{char}

\newcommand{\cont}{{\rm cont}}

\newcommand{\dom}{{\rm dom}}

\DeclareMathOperator{\equi}{equi}
\newcommand{\et}{{\rm\acute{e}t\/}}
\newcommand{\fppf}{{\it fppf\/}}
\newcommand{\fpqc}{{\it fpqc\/}}

\DeclareMathOperator{\id}{\,id}

\renewcommand{\mod}{{\rm\,mod\,}}

\newcommand{\red}{{\rm red}}

\newcommand{\sep}{{\rm sep}}

\DeclareMathOperator{\whtimes}{\mathchoice
            {\wh{\raisebox{0ex}[0ex]{$\DS\times$}}}
            {\wh{\raisebox{0ex}[0ex]{$\TS\times$}}}
            {\wh{\raisebox{0ex}[0ex]{$\SC\times$}}}
            {\wh{\raisebox{0ex}[0ex]{$\SSC\times$}}}}

\renewcommand{\phi}{\varphi}
\renewcommand{\epsilon}{\varepsilon}
\usepackage{amsfonts}
\newcommand{\BOne} {{\mathchoice{\hbox{\rm1\kern-2.7pt l\kern.9pt}}
                              {\hbox{\rm1\kern-2.7pt l\kern.9pt}}
                              {\hbox{\scriptsize\rm1\kern-2.3pt l\kern.4pt}}
                              {\hbox{\scriptsize\rm1\kern-2.4pt l\kern.5pt}}}}

\newcommand{\BD}{{\mathbb{D}}}

\newcommand{\BF}{{\mathbb{F}}}
\newcommand{\BG}{{\mathbb{G}}}
\newcommand{\BH}{{\mathbb{H}}}

\newcommand{\BL}{{\mathbb{L}}}

\newcommand{\BN}{{\mathbb{N}}}

\newcommand{\BP}{{\mathbb{P}}}
\newcommand{\BQ}{{\mathbb{Q}}}

\newcommand{\BZ}{{\mathbb{Z}}}

\newcommand{\CA}{{\mathcal{A}}}

\newcommand{\CC}{{\mathcal{C}}}

\newcommand{\CE}{{\mathcal{E}}}
\newcommand{\CF}{{\mathcal{F}}}
\newcommand{\CG}{{\mathcal{G}}}

\newcommand{\CI}{{\mathcal{I}}}
\newcommand{\CJ}{{\mathcal{J}}}
\newcommand{\CK}{{\mathcal{K}}}
\newcommand{\CL}{{\mathcal{L}}}
\newcommand{\CM}{{\mathcal{M}}}
\newcommand{\CN}{{\mathcal{N}}}
\newcommand{\CO}{{\mathcal{O}}}
\newcommand{\CP}{{\mathcal{P}}}

\newcommand{\CS}{{\mathcal{S}}}
\newcommand{\CT}{{\mathcal{T}}}
\newcommand{\CU}{{\mathcal{U}}}
\newcommand{\CV}{{\mathcal{V}}}

\newcommand{\CX}{{\mathcal{X}}}

\newcommand{\CZ}{{\mathcal{Z}}}

\newcommand{\FF}{{\mathfrak{F}}}
\newcommand{\FG}{{\mathfrak{G}}}

\newcommand{\FM}{{\mathfrak{M}}}

\newcommand{\Fa}{{\mathfrak{a}}}

\newcommand{\scrH}{{\mathscr{H}}}
\newcommand{\scrB}{{\mathscr{B}}}

\let\setminus\smallsetminus

\newcommand{\es}{\enspace}

\newcommand{\dual}{^{\SSC\lor}}

\newcommand{\mal}{^{\SSC\times}}

\newcommand{\ul}[1]{{\underline{#1}}}
\newcommand{\ol}[1]{{\overline{#1}}}
\newcommand{\wh}[1]{{\widehat{#1}}}
\newcommand{\wt}[1]{{\widetilde{#1}}}

\usepackage{ifthen}

\newcommand{\invlim}[1][]{\ifthenelse{\equal{#1}{}}% falls Argument leer
{\DS \lim_{\longleftarrow}}%                         verwende niedrige Version
{\DS \lim_{\underset{#1}{\longleftarrow}}}%  sonst:  verwende Argument
}

\newcommand{\dirlim}[1][]{\ifthenelse{\equal{#1}{}}% falls Argument leer
{\DS \lim_{\longrightarrow}}%                        verwende niedrige Version
{\DS \lim_{\underset{#1}{\longrightarrow}}}% sonst:  verwende Argument
}

\newcommand{\dbl}{{\mathchoice{\mbox{\rm [\hspace{-0.15em}[}}
                              {\mbox{\rm [\hspace{-0.15em}[}}
                              {\mbox{\scriptsize\rm [\hspace{-0.15em}[}}
                              {\mbox{\tiny\rm [\hspace{-0.15em}[}}}}
\newcommand{\dbr}{{\mathchoice{\mbox{\rm ]\hspace{-0.15em}]}}
                              {\mbox{\rm ]\hspace{-0.15em}]}}
                              {\mbox{\scriptsize\rm ]\hspace{-0.15em}]}}
                              {\mbox{\tiny\rm ]\hspace{-0.15em}]}}}}
\newcommand{\dpl}{{\mathchoice{\mbox{\rm (\hspace{-0.15em}(}}
                              {\mbox{\rm (\hspace{-0.15em}(}}
                              {\mbox{\scriptsize\rm (\hspace{-0.15em}(}}
                              {\mbox{\tiny\rm (\hspace{-0.15em}(}}}}
\newcommand{\dpr}{{\mathchoice{\mbox{\rm )\hspace{-0.15em})}}
                              {\mbox{\rm )\hspace{-0.15em})}}
                              {\mbox{\scriptsize\rm )\hspace{-0.15em})}}
                              {\mbox{\tiny\rm )\hspace{-0.15em})}}}}

\newcommand{\dotBD}{\vbox{\hbox{\kern2pt\bf.}\vskip-4.5pt\hbox{$\BD$}}}

\DeclareMathOperator{\QIsog}{QIsog}
\DeclareMathOperator{\Nilp}{\CN \!{\it ilp}}

\DeclareMathOperator{\Sets}{\CS \!{\it ets}}

\def\ulM{{\underline{M\!}\,}{}}
\def\olB{{\,\overline{\!B}}}
\def\olS{{\,\overline{\!S}}}

\def\s{\sigma^\ast}

\def\longto{\longrightarrow}
\def\into{\hookrightarrow}

\def\isoto{\stackrel{}{\mbox{\hspace{1mm}\raisebox{+1.4mm}{$\SC\sim$}\hspace{-3.5mm}$\longrightarrow$}}}

\newbox\mybox
\def\arrover#1{\mathrel{
       \setbox\mybox=\hbox spread 1.4em{\hfil$\scriptstyle#1$\hfil}
       \vbox{\offinterlineskip\copy\mybox
             \hbox to\wd\mybox{\rightarrowfill}}}}

\newcommand{\ppsi}{\delta}
\newcommand{\RZ}{\ul{\CM}_{\ul{\BL}_0}^{\hat{Z}}}
\newcommand{\Test}{{Y}}

\newcommand{\BaseOfD}{\BF}
\newcommand{\BaseFldOfLocSht}{k}

\newcommand{\BaseFldInSectUnif}{k}
\newcommand{\AlgClFld}{k}

\newcommand{\genericG}{P}
\newcommand{\eeeta}{\xi}
\newcommand{\Sht}{Sht}
\newcommand{\EtSht}{\acute{E}tSht}
\newcommand{\Vect}{V\!ect}
\DeclareMathOperator{\SpaceFl}{\CF\ell}
\newcommand{\tauGlob}{\tau}
\newcommand{\tauLoc}{\hat\tau}
\newcommand{\charsect}{s}

\DeclareMathOperator{\AbSh}{\CA {\it b}-\CS {\it h}}

\newcommand{\DefIndAlgStack}{Definition~3.14}
\newcommand{\ThmRepNablaH}{Theorem~3.15}
\newcommand{\RemDrinfeldVarshavsky}{Remark~3.19}
\newcommand{\EqBounded}{\cite[Equations~(3.2) and (3.3)]{AH_Unif}}
\newcommand{\LevelStructure}{Chapter~6}

\begin{document}

%%%%%%%%%%%%%%%%%%%%%%%%%%%%%%%%%%%%%%%%%%%%%%%%%%%%%%%%%%%%%%%%%%%%%%
\author{Esmail Arasteh Rad and Urs Hartl\footnote{Both authors acknowledge support by the Deutsche Forschungsgemeinschaft (DFG) in form of the research grant HA3002/2-1 and the SFB's 478 and 878.}}

\date{\today}

\title{Local $\BP$-shtukas and their relation to global $\FG$-shtukas}

\maketitle

\begin{abstract}
\noindent
This is the first in a sequence of two articles investigating moduli stacks of global $\FG$-shtukas, which are function field analogs for Shimura varieties. Here $\FG$ is a flat affine group scheme of finite type over a smooth projective curve, and global $\FG$-shtukas are generalizations of Drinfeld shtukas and analogs of abelian varieties with additional structure. Our moduli stacks generalize various moduli spaces used by different authors to prove instances of the Langlands program over function fields. 

In the present article we explain the relation between global $\FG$-shtukas and local $\BP$-shtukas, which are the function field analogs of $p$-divisible groups with additional structure. We prove the analog of a theorem of Serre and Tate stating the equivalence between the deformations of a global $\FG$-shtuka and its associated local $\BP$-shtukas. We also investigate local $\BP$-shtukas alone and explain their relation with Galois representations through their Tate modules. And if $\BP$ is a smooth affine group scheme with connected reductive generic fiber we prove the existence of Rapoport--Zink spaces for bounded local $\BP$-shtukas as formal schemes locally formally of finite type. In the sequel to this article we use these Rapoport--Zink spaces to uniformize the moduli stacks of global $\FG$-shtukas.

\noindent
{\it Mathematics Subject Classification (2000)\/}: 
%11F80,  % Galois representations (Discontinuous groups and automorphic forms)
11G09,  % Drinfeld Modules, higher dimensional motives
%11S20,  % Galois theory of local and $p$-adic fields
%11S25,  % Galois cohomology of local and $p$-adic fields
%13A35,  % Characteristic $p$ methods (Frobenius endomorphism) ...
%14F30,  % $p$-adic cohomology, crystalline cohomology
%14G20,  % Local ground fields
%14G22,  % Rigid analytic geometry
%14G35,  % Modular and Shimura varieties
(11G18,  % Arithmetic aspects of modular and Shimura varieties
14L05,  % Formal groups, $p$-divisible groups
14M15)  % Grassmannians, Schubert varieties, flag manifolds
%20G25,  % Linear algebraic groups over local fields and their integers
\end{abstract}

\tableofcontents

%%%%%%%%%%%%%%%%%%%%%%%%%%%%%%%%%%%%%%%%%%%%%%%%%%%%%%%%%%%%%%%%%%%%%%
%
%    Introduction
%
%%%%%%%%%%%%%%%%%%%%%%%%%%%%%%%%%%%%%%%%%%%%%%%%%%%%%%%%%%%%%%%%%%%%%%

\thispagestyle{empty}

\bigskip
\section{Introduction}
\setcounter{equation}{0}

Let $\BF_q$ be a finite field with $q$ elements, let $C$ be a smooth projective geometrically irreducible curve over $\BF_q$, and let $\FG$ be a flat affine group scheme of finite type over $C$. A \emph{global $\FG$-shtuka} $\ul\CG$ over an $\BF_q$-scheme $S$ is a tuple $(\CG,\charsect_1,\ldots,\charsect_n,\tauGlob)$ consisting of a $\FG$-torsor $\CG$ over $C_S:=C\times_{\BF_q}S$, an $n$-tuple of (characteristic) sections $(\charsect_1,\ldots,\charsect_n)\in C^n(S)$ and a Frobenius connection $\tauGlob$ defined outside the graphs of the sections $\charsect_i$, that is, an isomorphism $\tauGlob\colon\s \CG|_{C_S\setminus \cup_i \Gamma_{\charsect_i}}\isoto \CG|_{C_S\setminus \cup_i \Gamma_{\charsect_i}}$ where $\s=(\id_C \times \Frob_{q,S})^\ast$. 

In \cite{AH_Unif} we will show that the moduli stack $\nabla_n\scrH^1(C,\FG)$ of global $\FG$-shtukas, after imposing suitable boundedness conditions and level structures, is an algebraic Deligne--Mumford stack over $C^n$. 
One can hope that $\nabla_n\scrH^1(C,\FG)$ may play the same role that Shimura varieties play for number fields. More specifically one can hope that the Langlands correspondence for function fields is realized on its cohomology. Note that in particular our moduli stack generalizes the space $FSh_{D,r}$ of $F$-sheaves (also called ``Drinfeld-shtukas'') which was considered by Drinfeld~\cite{Drinfeld1} and Laurent Lafforgue~\cite{Laff} in their proof of the Langlands correspondence for $\FG=\GL_2$ (resp.\ $\FG=\GL_r$), and which in turn was generalized by Varshavsky's~\cite{Var} moduli stacks $FBun$ to the case where $\FG$ is a constant split reductive group. Varshavsky's moduli stack and our generalization are used by Vincent Lafforgue~\cite{Lafforgue12} to prove Langlands parameterization over function fields. Strictly speaking Drinfeld and L.~Lafforgue did not use the language of $\GL_r$-torsors but rather the equivalent one of locally free sheaves. Our space $\nabla_n\scrH^1(C,\FG)$ likewise generalizes the moduli stacks $Cht_{\ul\lambda}$ of Ng\^o and Ng\^o Dac~\cite{NgoNgo} who explain a simple method to count $\FG$-shtukas over finite fields, the stacks $\CE\ell\ell_{C,\mathscr{D},I}$ of Laumon, Rapoport and Stuhler~\cite{LRS} who used them to prove the local Langlands correspondence for $\GL_r$, and the stacks $\AbSh^{r,d}_H$ of the second author~\cite{Har1}; see \cite[\RemDrinfeldVarshavsky]{AH_Unif} for a detailed comparison between these moduli stacks.

In \cite{AH_Unif} we also prove that $\nabla_n\scrH^1(C,\FG)$ has a Rapoport--Zink uniformization by Rapoport--Zink spaces for local $\BP$-shtukas. More precisely, let $A_\nu\cong\BF_\nu\dbl\zeta\dbr$ be the completion of the local ring $\CO_{C,\nu}$ at a closed point $\nu\in C$, let $Q_\nu$ be its fraction field, and consider the group schemes $\BP=\BP_\nu:=\FG\times_C \Spec A_\nu$ and $\genericG_\nu=\FG\times_C\Spec Q_\nu$. Let $\Nilp_{A_\nu}$ denote the category of $A_\nu$-schemes on which the uniformizer $\zeta$ of $A_\nu$ is locally nilpotent. A \emph{local $\BP_\nu$-shtuka} over a scheme $S\in\Nilp_{A_\nu}$ is a pair $\ul \CL = (\CL_+,\tauLoc)$ consisting of an $L^+\BP_\nu$-torsor $\CL_+$ on $S$ and an isomorphism of the $L\genericG_\nu$-torsors $\tauLoc\colon  \hat{\sigma}^\ast \CL \isoto\CL$. Here $L\genericG_\nu$ (resp.\ $L^+\BP_\nu$) denotes the group of loops (resp.\ positive loops) of $\BP_\nu$ (see Section~\ref{LoopsAndSht}), $\CL$ denotes the $L\genericG_\nu$-torsor associated with $\CL_+$ and $\hat{\sigma}^\ast \CL$ the pullback of $\CL$ under the absolute $\BF_\nu$-Frobenius endomorphism $\Frob_{(\#\BF_\nu),S} \colon  S \to S$. Building on earlier work of Anderson~\cite{Anderson2}, Drinfeld~\cite{Drinfeld76}, Genestier~\cite{Genestier}, Laumon~\cite{Laumon}, Rosen~\cite{Rosen} and Taguchi~\cite{Taguchi93}, local $\GL_r$-shtukas were studied by the second author in \cite{HartlPSp} as function field analogs of $p$-divisible groups and $F$-crystals. Local $\BP_\nu$-shtukas, which can be viewed as function field analogs of $p$-divisible groups with extra structure by the group scheme $\BP_\nu$, were introduced by Viehmann and the second author in \cite{H-V,HV2} in the case where $\BP_\nu$ is a constant split reductive group. Our definition is a generalization to flat affine group schemes $\BP_\nu$ of finite type.

\newcommand{\SSS}{S}
\newcommand{\TTT}{T}
As a preparation to \cite{AH_Unif} we show in this article that for $\BP_\nu$ smooth over $A_\nu$ and for a fixed local $\BP_\nu$-shtuka $\ul\BL_0$ over a field $k$, the \emph{unbounded Rapoport--Zink functor} 
\begin{eqnarray*}
\ul{\CM}_{\ul{\BL}_0}\colon (\Nilp_{k\dbl\zeta\dbr})^o & \longto &  \Sets\\
\SSS &\longmapsto & \big\{\text{Isomorphism classes of }(\ul{\CL},\bar{\delta})\colon\;\text{where }\ul{\CL}~\text{is a local $\BP_\nu$-shtuka}\\ 
&&\text{\quad over $\SSS$ and }\bar{\delta}\colon  \ul{\CL}_{\bar{\SSS}}\to \ul{\BL}_{0,\bar{\SSS}}~\text{is a quasi-isogeny  over $\bar{\SSS}$}\big\},
\end{eqnarray*}
where $\bar\SSS=\Var(\zeta)\subset S$, is representable by an ind-scheme, ind-quasi-projective over $\Spf k\dbl\zeta\dbr$; see Theorem~\ref{ThmModuliSpX}. More precisely, if the $L^+\BP_\nu$-torsor underlying $\ul\BL_0$ is trivial then $\ul\CM_{\ul\BL_0}\cong\SpaceFl_{\BP_\nu}\whtimes_{\BF_\nu}\Spf\BF_\nu\dbl\zeta\dbr$, where $\SpaceFl_{\BP_\nu}$ is the affine flag variety of $\BP_\nu$; see Remark~\ref{Flagisquasiproj}. To obtain a formal scheme locally formally of finite type, as in the analog for $p$-divisible groups, one has to assume that $\BP_\nu$ has connected reductive generic fiber, and one has to \emph{bound the Hodge polygon}, that is the relative position of $\hat\sigma^*\CL_+$ and $\CL_+$ under $\tauLoc$. We give an axiomatic treatment of bounds in Section~\ref{BC} and prove the representability of the bounded Rapoport--Zink functor by a formal scheme locally formally of finite type over $\Spf k\dbl\zeta\dbr$ in Theorem~\ref{ThmRRZSp}. Our proof, which generalizes \cite[Theorem~6.3]{H-V}, is inspired by Rapoport's and Zink's original result \cite[Theorem~2.16]{RZ} for $p$-divisible groups.

In addition, in Chapter~\ref{GalRepSht} we discuss the relation between local $\BP$-shtukas and Galois representations which is given by the associated Tate module. This chapter is largely independent of the rest of this article and is only used in Remark~\ref{RemTateF}. In Chapter~\ref{UnboundedUnif} we consider the formal stack $\nabla_n\scrH^1(C,\FG)^{\ul \nu}$, which is obtained by taking the formal completion of the stack $\nabla_n\scrH^1(C,\FG)$ at a fixed $n$-tuple of pairwise different characteristic places $\ul \nu=(\nu_1,\ldots,\nu_n)$. This means we let $A_{\ul\nu}$ be the completion of the local ring $\CO_{C^n,\ul\nu}$, and we consider global $\FG$-shtukas only over schemes $S$ whose characteristic morphism $S\to C^n$ factors through $\Nilp_{A_\ul\nu}$.
Recall that with an abelian variety over a scheme in $\Nilp_{\BZ_p}$ one can associate its $p$-divisible group. In the analogous situation for global $\FG$-shtukas one can associate a tuple $(\wh\Gamma_{\nu_i}(\ul\CG))_i$ of local $\BP_{\nu_i}$-shtukas $\wh\Gamma_{\nu_i}(\ul\CG)$ with a global $\FG$-shtuka $\ul\CG$ in $\nabla_n\scrH^1(C,\FG)^{\ul \nu}(S)$. We construct this \emph{global-local functor} in Section~\ref{SectGlobalLocalFunctor} by first generalizing the glueing lemma of Beauville and Laszlo \cite{B-L} in Lemma~\ref{LemmaBL}. In analogy with a theorem of Serre and Tate relating the deformation theory of abelian varieties over schemes in $\Nilp_{\BZ_p}$ and their associated $p$-divisible groups, we prove in Theorem~\ref{Serre-Tate} the equivalence between the infinitesimal deformations of a global $\FG$-shtuka and the infinitesimal deformations of its associated $n$-tuple of local $\BP_{\nu_i}$-shtukas. Note that unlike abelian varieties, $\FG$-shtukas posses more than one characteristic and we must keep track of the deformations of the local $\BP_{\nu_i}$-shtukas at each of these characteristic places $\nu_i$. This theorem for abelian $\tau$-sheaves (corresponding to the case $\FG=\GL_r$) and their associated $z$-divisible groups was first stated and proved by the second author in \cite{Har1}.

\medskip

{\it Acknowledgements.} 
We would like to thank E.~Viehmann and L.~Kramer for helpful discussions and the anonymous referee for his careful reading and valuable remarks.

\subsection{Notation and Conventions}\label{Notation and Conventions}
Throughout this article we denote by
\begin{tabbing}
$\genericG_\nu:=\FG\times_C\Spec Q_\nu,$\; \=\kill
$\BF_q$\> a finite field with $q$ elements and characteristic $p$,\\[1mm]
$C$\> a smooth projective geometrically irreducible curve over $\BF_q$,\\[1mm]
$Q:=\BF_q(C)$\> the function field of $C$,\\[1mm]
$\nu$\> a closed point of $C$, also called a \emph{place} of $C$,\\[1mm]
$\BF_\nu$\> the residue field at the place $\nu$ on $C$,\\[1mm]
$A_\nu$\> the completion of the stalk $\CO_{C,\nu}$ at  $\nu$,\\[1mm]
$Q_\nu:=\Quot(A_\nu)$\> its fraction field,\\[1mm]
$\BaseOfD$ \> a finite field containing $\BF_q$,\\[1mm]

$\BD_R:=\Spec R\dbl z \dbr$ \> \parbox[t]{0.775\textwidth}{the spectrum of the ring of formal power series in $z$ with coefficients in an $\BaseOfD$-algebra $R$,}\\[1mm]
$\hat{\BD}_R:=\Spf R\dbl z \dbr$ \> the formal spectrum of $R\dbl z\dbr$ with respect to the $z$-adic topology.
\end{tabbing}
\noindent
When $R= \BaseOfD$ we drop the subscript $R$ from the notation of $\BD_R$ and $\hat{\BD}_R$.\\[-3mm]

\noindent
For a formal scheme $\wh S$ we denote by $\Nilp_{\wh S}$ the category of schemes over $\wh S$ on which an ideal of definition of $\wh S$ is locally nilpotent. We  equip $\Nilp_{\wh S}$ with the \fppf-topology. We also denote by
\begin{tabbing}
$\genericG_\nu:=\FG\times_C\Spec Q_\nu,$\; \=\kill
$n\in\BN_{>0}$\> a positive integer,\\[1mm]
$\ul \nu:=(\nu_i)_{i=1\ldots n}$\> an $n$-tuple of closed points of $C$,\\[1mm]
$A_{\ul\nu}$\> the completion of the local ring $\CO_{C^n,\ul\nu}$ of $C^n$ at the closed point $\ul\nu=(\nu_i)$,\\[1mm]
$\Nilp_{A_{\ul\nu}}:=\Nilp_{\Spf A_{\ul\nu}}$\> \parbox[t]{0.775\textwidth}{the category of schemes over $C^n$ on which the ideal defining the closed point $\ul\nu\in C^n$ is locally nilpotent,}\\[2mm]
$\Nilp_{\BaseOfD\dbl\zeta\dbr}:=\Nilp_{\hat\BD}$\> \parbox[t]{0.775\textwidth}{the category of $\BD$-schemes $S$ for which the image of $z$ in $\CO_S$ is locally nilpotent. We denote the image of $z$ by $\zeta$ since we need to distinguish it from $z\in\CO_\BD$.}\\[2mm]
$\FG$\> a flat affine group scheme of finite type over $C$, \\[1mm]
$\BP_\nu:=\FG\times_C\Spec A_\nu$ \> the base change of $\FG$ to $\Spec A_\nu$,\\[1mm]
$\genericG_\nu:=\FG\times_C\Spec Q_\nu$ \> the generic fiber of $\BP_\nu$ over $\Spec Q_\nu$,\\[1mm]
$\BP$\> a flat affine group scheme of finite type over $\BD=\Spec\BaseOfD\dbl z\dbr$,\\[1mm] 
$\genericG:=\BP\times_{\BD}\Spec\BaseOfD\dpl z\dpr$\> the generic fiber of $\BP$ over $\Spec\BaseOfD\dpl z\dpr$.
\end{tabbing}

\noindent
Let $S$ be an $\BF_q$-scheme. We denote by $\sigma_S \colon  S \to S$ its $\BF_q$-Frobenius endomorphism which acts as the identity on the points of $S$ and as the $q$-power map on the structure sheaf. Likewise we let $\hat{\sigma}_S\colon S\to S$ be the $\BaseOfD$-Frobenius endomorphism of an $\BaseOfD$-scheme $S$. We set
\begin{tabbing}
$\genericG_\nu:=\FG\times_C\Spec Q_\nu,$\; \=\kill
$C_S := C \times_{\Spec\BF_q} S$, \> and \\[1mm]
$\sigma := \id_C \times \sigma_S$.
\end{tabbing}

Let $H$ be a sheaf of groups (for the \fppf-topology) on a scheme $X$. In this article a (\emph{right}) \emph{$H$-torsor} (also called an \emph{$H$-bundle}) on $X$ is a sheaf $\CG$ for the \fppf-topology on $X$ together with a (right) action of the sheaf $H$ such that $\CG$ is isomorphic to $H$ on an \fppf-covering of $X$. Here $H$ is viewed as an $H$-torsor by right multiplication.

%%%%%%%%%%%%%%%%%%%%%%%%%%%%%%%%%%%%%%%%%%%%%%%%%%%%%%%%%%%%%%%%%%%%%%
%
%  Global and Local Shtukas
%
%%%%%%%%%%%%%%%%%%%%%%%%%%%%%%%%%%%%%%%%%%%%%%%%%%%%%%%%%%%%%%%%%%%%%%

\section{Local \texorpdfstring{$\BP$}{P}-Shtukas and Global \texorpdfstring{$\FG$}{G}-Shtukas}\label{Shtukas}
\setcounter{equation}{0}

Global $\FG$-shtukas are function field analogs of abelian varieties. They were introduced by Drinfeld~\cite{Drinfeld1} in the case where $\FG=\GL_r$ and used by him and by L.~Lafforgue~\cite{Laff} to establish the Langlands correspondence for $\GL_r$ over global function fields. As we mentioned in the introduction, they were generalized by Laumon, Rapoport and Stuhler~\cite{LRS}, Varshavsky~\cite{Var}, Ng\^o and Ng\^o Dac~\cite{NgoNgo}, and in \cite{Har1}. We further generalize all these variants in Definition~\ref{Global Sht}. Varshavsky's and our generalization are used by Vincent Lafforgue~\cite{Lafforgue12} to prove Langlands parameterization over function fields. The local $p$-adic properties of abelian varieties are largely captured by their associated $p$-divisible groups. In the theory of global $\FG$-shtukas the latter correspond to local shtukas; see \cite[Chapter~3]{HartlDict}.

\subsection{Loop Groups and Local \texorpdfstring{$\BP$}{P}-Shtukas}\label{LoopsAndSht}

Since we want to develop the theory of local $\BP$-shtukas partly independently of global $\FG$-shtukas we let $\BaseOfD$ be a finite field and $\BaseOfD\dbl z\dbr$ be the power series ring over $\BaseOfD$ in the variable $z$. We let $\BP$ be a flat affine group scheme of finite type over $\BD:=\Spec\BaseOfD\dbl z\dbr$, and we let $\genericG:=\BP\times_\BD\dot{\BD}$ be the generic fiber of $\BP$ over $\dot\BD:=\Spec\BaseOfD\dpl z\dpr$. We are mainly interested in the situation where we have an isomorphism $\BD\cong\Spec A_\nu$ for a place $\nu$ of $C$ and where $\BP=\BP_\nu:=\FG\times_C\Spec A_\nu$. We recall the following

\begin{definition}\label{DefLoopGps}
The \emph{group of positive loops associated with $\BP$} is the infinite dimensional affine group scheme $L^+\BP$ over $\BaseOfD$ whose $R$-valued points for an $\BaseOfD$-algebra $R$ are 
\[
L^+\BP(R):=\BP(R\dbl z\dbr):=\BP(\BD_R):=\Hom_\BD(\BD_R,\BP)\,.
\]
The \emph{group of loops associated with $\genericG$} is the $\fpqc$-sheaf of groups $L\genericG$ over $\BaseOfD$ whose $R$-valued points for an $\BaseOfD$-algebra $R$ are 
\[
L\genericG(R):=\genericG(R\dpl z\dpr):=\genericG(\dot{\BD}_R):=\Hom_{\dot\BD}(\dot\BD_R,\genericG)\,,
\]
where we write $R\dpl z\dpr:=R\dbl z \dbr[\frac{1}{z}]$ and $\dot{\BD}_R:=\Spec R\dpl z\dpr$. It is representable by an ind-scheme of ind-finite type over $\BaseOfD$; see \cite[\S\,1.a]{PR2}, or \cite[\S4.5]{B-D}, \cite{Ngo-Polo}, \cite{Faltings03} when $\BP$ is constant. 
Let $\scrH^1(\Spec \BaseOfD,L^+\BP)\,:=\,[\Spec \BaseOfD/L^+\BP]$ (respectively $\scrH^1(\Spec \BaseOfD,L\genericG)\,:=\,[\Spec \BaseOfD/L\genericG]$) denote the classifying space of $L^+\BP$-torsors (respectively $L\genericG$-torsors). It is a stack fibered in groupoids over the category of $\BaseOfD$-schemes $S$ whose category $\scrH^1(\Spec \BaseOfD,L^+\BP)(S)$ consists of all $L^+\BP$-torsors (resp.\ $L\genericG$-torsors) on $S$. The inclusion of sheaves $L^+\BP\subset L\genericG$ gives rise to the natural 1-morphism 
\begin{equation}\label{EqLoopTorsor}
L\colon\scrH^1(\Spec \BaseOfD,L^+\BP)\longto \scrH^1(\Spec \BaseOfD,L\genericG),~\CL_+\mapsto \CL\,.
\end{equation}
\end{definition}

\begin{definition}\label{formal torsor def}
Let $\hat{\BP}$ be the formal group scheme over $\hat{\BD}:=\Spf\BaseOfD\dbl z\dbr$, obtained by the formal completion of $\BP$ along $V(z)$. A \emph{formal $\hat{\BP}$-torsor} over an $\BaseOfD$-scheme $S$ is a $z$-adic formal scheme $\hat{\CP}$ over $\hat{\BD}_{S}:=\hat{\BD}\whtimes_{\BaseOfD} S$ together with an action
$\hat{\BP}\whtimes_{\hat{\BD}}\hat{\CP}\rightarrow\hat{\CP}$ of $\hat{\BP}$ on $\hat{\CP}$ such that there is a covering $\hat{\BD}_{S'} \rightarrow \hat{\BD}_{S}$ where ${S'\rightarrow S}$ is an \fpqc-covering
and a $\hat{\BP}$-equivariant isomorphism 
$$
\hat{\CP} \whtimes_{\hat{\BD}_{S}} \hat{\BD}_{S'}\isoto\hat{\BP} \whtimes_{\hat{\BD}}\hat{\BD}_{S'}.
$$ 
Here $\hat{\BP}$ acts on itself by right multiplication.
Let $\scrH^1(\hat{\BD},\hat{\BP})$ be the category fibered in groupoids that assigns to each $\BaseOfD$-scheme $S$ the groupoid consisting of all formal $\hat{\BP}$-torsors over $\hat{\BD}_S$.
\end{definition}

\begin{remark}\label{B-L remark 1}
If $\BP$ is smooth over $\BD$ then for any $\wh \CP$ in $\scrH^1(\hat{\BD},\hat{\BP})(\Spec R)$ one can find an \'etale covering $R\rightarrow R'$ such that $\wh \CP\whtimes_{\hat\BD_R}\hat\BD_{R'}$ is isomorphic to $\hat\BP_{R'}$ in $\scrH^1(\hat{\BD},\hat{\BP})(R')$. Indeed, since $\wh\CP\to\hat\BD_R$ is smooth, the restriction $\wh\CP_0$ of $\wh \CP$ to $V(z) \subseteq \hat{\BD}_R$ is likewise smooth over $R$. Therefore $\wh\CP_0$ has a section over an \'etale covering $R\to R'$. Then by smoothness this section extends over $\hat{\BD}_R$.
\end{remark}

In \cite[Proposition~2.2.(a)]{H-V} Viehmann and the second author proved that for a split reductive group $G$, there is a bijection of (pointed) sets between the \v{C}ech cohomology $\CKoh^1(S_\fpqc,L^+G)$ and the set of isomorphism classes of $z$-adic formal schemes over $\hat{\BD}_S$. By the same arguments one can even see that there is a canonical equivalence between the corresponding categories.
\begin{proposition}\label{formal torsor prop}
There is a natural isomorphism 
$$
\scrH^1(\hat{\BD},\hat{\BP}) \isoto \scrH^1(\Spec \BaseOfD,L^+\BP)
$$
of groupoids. In particular, if $\BP$ is smooth over $\BD$ then all $L^+\BP$-torsors for the $\fpqc$-topology on $S$ are already trivial \'etale locally on $S$. 
\end{proposition}

\begin{proof}
With a given element $\hat{\CP}$ of $\scrH^1(\hat{\BD},\hat{\BP})(S)$ one can associate the following sheaf 
\begin{eqnarray*}
\ul{\cK}\colon\es  S_\fpqc & \longto & \Sets\\
T & \longmapsto & \Hom_{\hat{\BD}_S}(\hat{\BD}_T,\hat{\CP}),
\end{eqnarray*}
where $S_\fpqc$ denotes the big $\fpqc$-site on $S$. This sheaf is a torsor under the action of $L^+\BP(T)=\Hom_{\hat{\BD}}(\hat{\BD}_T,\hat{\BP})$.

Conversely let $\CK$ be an $L^+\BP$-torsor. Let $S'\to S$ be an $\fpqc$-covering that trivializes $\CK$ and fix a trivialization $\CK_{S'}\isoto (L^+\BP)_{S'}$. This gives a 1-cocycle $g\in L^+\BP(S'')$, where $S''=S'\times_S S'$. Now $\bar{g}=g (mod~z^n)$ can be viewed as a descent data on $\hat{\BP} \whtimes_{\BD} \BD_{n,S'}=\BP\times_{\BD}{\BD}_{n,S'}$ where $\BD_{n,S'}:=\Spec\BaseOfD\dbl z\dbr/(z^n)\times_\BaseOfD S'$. Since $\BD_{n,S'} \to \BD_{n,S}$ is an $\fpqc$-covering and $\BP$ is affine, the descent data is effective by \cite[\S\,6.1, Theorem~6]{BLR} and gives an affine finitely presented scheme $\hat{\CG}_n$ over $\BD_{n,S}$ by \cite[IV$_2$, Proposition~2.7.1]{EGA}, which is moreover smooth by \cite[IV$_4$, Corollaire 17.7.3]{EGA} if $\BP$ is smooth over $\BD$. These schemes form an inductive system $\{\hat{\CG}_n\}_{n\in \BN}$. Now set $\hat{\cG}:=\dirlim \hat{\CG}_n$, the existence of this limit (in the category of $z$-adic formal schemes over $\hat{\BD}_S$) follows from \cite[I$_{\rm new}$, Corollary~10.6.4]{EGA}. This shows that the functor is essentially surjective. By the above construction we see that the functor is also fully faithful.

The last statement now follows from Remark~\ref{B-L remark 1}.
\end{proof}

\begin{definition}\label{ker}
Assume that we have two morphisms $f,g\colon X\to Y$ of schemes or stacks. We denote by $\equi(f,g\colon X\rightrightarrows Y)$ the pull back of the diagonal under the morphism $(f,g)\colon X\to Y\times_\BZ Y$, that is, we let 
\[
\equi(f,g\colon X\rightrightarrows Y)\,:=\,X\times_{(f,g),Y\times Y,\Delta}Y
\]
where $\Delta=\Delta_{Y/\BZ}\colon Y\to Y\times_\BZ Y$ is the diagonal morphism.
\end{definition}

Generalizing \cite[Definition~3.1]{H-V} we define the space of local $\BP$-shtukas as follows.

\begin{definition}\label{localSht}
Let $\CX$ be the fiber product 
$$
\scrH^1(\Spec \BaseOfD,L^+\BP)\times_{\scrH^1(\Spec \BaseOfD,L\genericG)} \scrH^1(\Spec \BaseOfD,L^+\BP)
$$ 
of groupoids. Let $pr_i$ denote the projection onto the $i$-th factor. We define the groupoid of \emph{local $\BP$-shtukas} $\Sht_{\BP}^{\BD}$ to be 
$$
\Sht_{\BP}^{\BD}\;:=\;\equi\left(\hat{\sigma}\circ pr_1,pr_2\colon \CX \rightrightarrows \scrH^1(\Spec \BaseOfD,L^+\BP)\right)\whtimes_{\Spec\BaseOfD}\Spf\BaseOfD\dbl\zeta\dbr.
$$
(see Definition~\ref{ker}) where $\hat{\sigma}:=\hat{\sigma}_{\scrH^1(\Spec \BaseOfD,L^+\BP)}$ is the absolute $\BaseOfD$-Frobenius of $\scrH^1(\Spec \BaseOfD,L^+\BP)$. The category $\Sht_{\BP}^{\BD}$ is fibered in groupoids over the category $\Nilp_{\BaseOfD\dbl\zeta\dbr}$ of $\BaseOfD\dbl\zeta\dbr$-schemes on which $\zeta$ is locally nilpotent. We call an object of the category $\Sht_{\BP}^{\BD}(S)$ a \emph{local $\BP$-shtuka over $S$}. 

More explicitly a local $\BP$-shtuka over $S\in \Nilp_{\BaseOfD\dbl\zeta\dbr}$ is a pair $\ul \CL = (\CL_+,\tauLoc)$ consisting of an $L^+\BP$-torsor $\CL_+$ on $S$ and an isomorphism of the associated loop group torsors $\tauLoc\colon  \hat{\sigma}^\ast \CL \to\CL$ from \eqref{EqLoopTorsor}. 
\end{definition}

\begin{definition}\label{etallocalSht}
A local $\BP$-shtuka $(\CL_+,\tauLoc)$ is called \emph{\'etale} if $\tauLoc$ comes from an isomorphism of $L^+\BP$-torsors $\hat{\sigma}^\ast\CL_+\isoto \CL_+$. We denote by $\EtSht_\BP^\BD(S)$ the category of \'etale local $\BP$-shtukas over $S$.
\end{definition}

\begin{lemma}\label{LSisnonemptyA}
Let $\AlgClFld$ be a separably closed field extension of $\BaseOfD$. If $\BP$ is smooth over $\BD$ with connected special fiber, then for any $b\in L^+\BP(\AlgClFld)$ there exists some $c\in L^+\BP(\AlgClFld)$ such that $b\,\hat\sigma^*(c)=c$.
\end{lemma}

\begin{proof}
Let $\hat{\BP}$ be as in Definition~\ref{formal torsor def}. Then $L^+\BP(\AlgClFld)=\hat\BP(\AlgClFld\dbl z\dbr)$. We view $\hat{\BP}$ as the inductive limit $\dirlim\BP_n$, where $\BP_n=\BP\times_\BD \BD_n$ with $\BD_n:=\Spec\BaseOfD\dbl z\dbr/(z^n)$. Let $\wt{G}_n$ denote the linear algebraic group over $\BaseOfD$ given by the Weil restriction $\Res_{\BD_n/\Spec \BaseOfD}(\BP_n)$. The reduction of $b$ mod $z^n$ gives an element $b_n\in \wt{G}_n (\AlgClFld)$.
Since $\BP$ is smooth with connected special fiber, $\wt{G}_n$ is connected by \cite[Proposition~A.5.9]{CGP}. Thus by Lang's theorem \cite[Corollary on p.~557]{Lang} there exists a $c_n\in\wt{G}_n(\AlgClFld)$ such that $b_n \hat{\sigma}^\ast(c_n)=c_n$. Here $\hat{\sigma}$ is the $\BaseOfD$-Frobenius on $\wt{G}_n$ which coincides with the Frobenius $\hat\sigma$ induced from $\BP$. Now consider the reduction map $\alpha_n\colon  \wt{G}_{n+1}(\AlgClFld)\to \wt{G}_n(\AlgClFld)$ and the element $\bar{d}_n:=\alpha_n(c_{n+1})^{-1}c_n $ which satisfies $\hat{\sigma}^\ast(\bar{d}_n)=\bar{d}_n$ and hence lies in $\wt{G}_n(\BaseOfD)=\BP(\BD_{n+1})$. Since $\BP$ is smooth $\bar{d}_n$ lifts to an element $d_n\in \BP(\BD_{n+1})=\wt{G}_{n+1}(\BaseOfD)$. Replacing $c_{n+1}$ by $c_{n+1}d_n$ we may assume that $\alpha_{n}(c_{n+1})=c_n$ and then take $c:=\invlim c_n$.  
\end{proof}

\begin{corollary}\label{CorEtIsTrivial}
If $\BP$ is smooth over $\BD$ with connected special fiber, then every \'etale local $\BP$-shtuka over a separably closed field $\AlgClFld$ is isomorphic to \mbox{$\bigl((L^+\BP)_\AlgClFld,1\!\cdot\!\hat\sigma^*\bigr)$}.
\end{corollary}

\begin{proof}
Let $\ul\CL=(\CL_+,\tauLoc)$ be an \'etale local $\BP$-shtuka over $\AlgClFld$. By Proposition~\ref{formal torsor prop} there is a trivialization of the $L^+\BP$-torsor $\CL_+$ giving rise to an isomorphism $\ul\CL\cong\bigl((L^+\BP)_\AlgClFld,b\!\cdot\!\hat\sigma^*\bigr)$ for some $b\in L^+\BP(\AlgClFld)$. By Lemma~\ref{LSisnonemptyA} there is an element $c\in L^+\BP(\AlgClFld)$ with $b\,\hat\sigma^*(c)=c$ and multiplication with $c$ is an isomorphism $\bigl((L^+\BP)_\AlgClFld,1\!\cdot\!\hat\sigma^*\bigr)\isoto\bigl((L^+\BP)_\AlgClFld,b\!\cdot\!\hat\sigma^*\bigr)$. 
\end{proof}

Local $\BP$-shtukas can be viewed as function field analogs of $p$-divisible groups. This inspires the following notions of quasi-isogenies; compare~\cite[Definition~3.8]{H-V}.

\begin{definition}\label{quasi-isogeny L}
A \emph{quasi-isogeny} $f\colon\ul\CL\to\ul\CL'$ between two local $\BP$-shtukas $\ul{\CL}:=(\CL_+,\tauLoc)$ and $\ul{\CL}':=(\CL_+' ,\tauLoc')$ over $S$ is an isomorphism of the associated $L\genericG$-torsors $f \colon  \CL \to \CL'$ satisfying $f\circ\tauLoc=\tauLoc'\circ\hat{\sigma}^{\ast}f$. We denote by $\QIsog_S(\ul{\CL},\ul{\CL}')$ the set of quasi-isogenies between $\ul{\CL}$ and $\ul{\CL}'$ over $S$, and we write $\QIsog_S(\ul\CL):=\QIsog_S(\ul\CL,\ul\CL)$ for the quasi-isogeny group of $\ul\CL$. 
\end{definition}

As in the theory of $p$-divisible groups, also our quasi-isogenies are rigid. Here we prove the case of local $\BP$-shtukas which is analogous to $p$-divisible groups. Like for abelian varieties, the case of global $\FG$-shtukas only holds in fixed finite characteristics. We will define quasi-isogenies between global $\FG$-shtukas in Section~\ref{SectGlobalGShtukas} and prove rigidity for them in Proposition~\ref{PropGlobalRigidity}.
 
\begin{proposition}[Rigidity of quasi-isogenies for local $\BP$-shtukas] \label{PropRigidityLocal}
Let $S$ be a scheme in $\Nilp_{\BaseOfD\dbl\zeta\dbr}$ and
let $j \colon  \bar{S}\rightarrow S$ be a closed immersion defined by a sheaf of ideals $\CI$ which is locally nilpotent.
Let $\ul{\CL}$ and $\ul{\CL}'$ be two local $\BP$-shtukas over $S$. Then
$$
\QIsog_S(\ul{\CL}, \ul{\CL}') \longto \QIsog_{\bar{S}}(j^*\ul{\CL}, j^*\ul{\CL}') ,\quad f \mapsto j^*f
$$
is a bijection of sets.
\end{proposition}

\begin{proof}
This was proved in \cite[Proposition~3.9]{H-V} when $\BP=G\times_\BaseOfD\BD$ for a constant split reductive group $G$ over $\BaseOfD$. The proof carries over literally. Compare also Proposition~\ref{PropGlobalRigidity}.
\end{proof}

\subsection{Global \texorpdfstring{$\FG$}{G}-Shtukas}\label{SectGlobalGShtukas}

Let $\BF_q$ be a finite field with $q$ elements, let $C$ be a smooth projective geometrically irreducible curve over $\BF_q$, and let $\FG$ be a flat affine group scheme of finite type over $C$. The relation to Section~\ref{LoopsAndSht} is as follows. We are mainly interested in the case where $\BaseOfD\dbl z\dbr\cong A_\nu$ and $\BP=\BP_\nu$.

\begin{definition}\label{Global Sht}
A \emph{global $\FG$-shtuka} $\ul\CG=(\CG,\charsect_1,\ldots,\charsect_n,\tauGlob)$ over an $\BF_q$-scheme $S$ is a tuple where
\begin{itemize}
\item $\CG$ is a $\FG$-torsor over $C_S$, 
\item $\charsect_1,\ldots,\charsect_n \in C(S)$ are $\BF_q$-morphisms called the \emph{characteristic sections of $\ul\CG$}, and
\item $\tau\colon \sigma^*\CG_{|_{{C_S}\setminus{\Gamma_{\charsect_1}\cup\ldots\cup\Gamma_{\charsect_n}}}}\isoto \CG_{|_{{C_S}\setminus{\Gamma_{\charsect_1}\cup\ldots\cup\Gamma_{\charsect_n}}}}$ is an isomorphism of $\FG$-torsors. Here $\Gamma_{\charsect_i}\subset C_S$ denotes the graph of the morphism $\charsect_i$.
\end{itemize}
We write $\sigma=\id_C\times\sigma_S$ for the $\BF_q$-Frobenius endomorphism $\sigma_S \colon  S \to S$ which acts as the identity on the points of $S$ and as the $q$-power map on the structure sheaf. We denote the \emph{moduli stack} of global $\FG$-shtukas by $\nabla_n\scrH^1(C,\FG)$. It is a stack fibered in groupoids over the category of $\BF_q$-schemes. Sometimes we will fix the sections $(\charsect_1,\ldots,\charsect_n)\in C^n(S)$ and simply call $\ul\CG=(\CG,\tauGlob)$ a global $\FG$-shtuka over $S$.
\end{definition}

In \cite[\ThmRepNablaH]{AH_Unif} we prove that $\nabla_n \scrH^1(C, \FG)$ is an ind-algebraic stack over $C^n$ (in the sense of \cite[\DefIndAlgStack]{AH_Unif}) which is ind-separated and locally of ind-finite type. However, we will not use this result in the present article, as we will mainly focus on local $\BP$-shtukas, and the relation between individual global $\FG$-shtukas and local $\BP$-shtukas. For a thorough discussion how our global $\FG$-shtukas and their moduli spaces generalize similar concepts in the literature, we refer to the introduction and to \cite[\RemDrinfeldVarshavsky]{AH_Unif}.

\medskip

There is also a notion of quasi-isogenies for global $\FG$-shtukas.

\begin{definition}\label{quasi-isogenyGlob}
Consider a scheme $S$ together with characteristic morphisms $\charsect_i\colon S\to C$ for $i=1,\ldots,n$ and let $\ul{\CG}=(\CG,\tau)$ and $\ul{\CG}'=(\CG',\tau')$ be two global $\FG$-shtukas over $S$ with the same characteristics $\charsect_i$. A \emph{quasi-isogeny} from $\ul\CG$ to $\ul\CG'$ is an isomorphism $f\colon\CG|_{C_S \setminus D_S}\isoto \CG'|_{C_S \setminus D_S}$ satisfying $\tau'\sigma^{\ast}(f)=f\tau$, where $D$ is some effective divisor on $C$. We denote the \emph{group of quasi-isogenies} of $\ul\CG$ to itself by $\QIsog_S(\ul\CG)$.
\end{definition}

Like for abelian varieties, rigidity of global $\FG$-shtukas only holds in fixed finite characteristics; see Proposition~\ref{PropGlobalRigidity}.

%%%%%%%%%%%%%%%%%%%%%%%%%%%%%%%%%%%%%%%%%%%%%%%%%%%%%%%%%%%%%%%%%%%%%%
%
%  Galois Representations Parametrized by Shtukas
%
%%%%%%%%%%%%%%%%%%%%%%%%%%%%%%%%%%%%%%%%%%%%%%%%%%%%%%%%%%%%%%%%%%%%%%

\section{Tate Modules for Local \texorpdfstring{$\BP$}{P}-Shtukas}\label{GalRepSht}
\setcounter{equation}{0}

In this chapter we assume that $\BP$ is a flat affine group scheme of finite type over $\BD$. For a scheme $S\in\Nilp_{\BaseOfD\dbl\zeta\dbr}$ let $\CO_S\dbl z\dbr$ be the sheaf of $\CO_S$-algebras on $S$ for the \fpqc-topology whose ring of sections on an $S$-scheme $\Test$ is the ring of power series $\CO_S\dbl z\dbr(\Test):=\Gamma(\Test,\CO_\Test)\dbl z\dbr$. Let $\CO_S\dpl z\dpr$ be the \fpqc-sheaf of $\CO_S$-algebras on $S$ associated with the presheaf $\Test\mapsto\Gamma(\Test,\CO_\Test)\dbl z\dbr[\frac{1}{z}]$. A sheaf $M$ of $\CO_S\dbl z\dbr$-modules on $S$ which is finite free $\fpqc$-locally on $S$ is already finite free Zariski-locally on $S$ by \cite[Proposition~2.3]{H-V}. We call those modules \emph{finite locally free sheaves of $\CO_S\dbl z\dbr$-modules}. We denote by $\hat\sigma^*$ the endomorphism of $\CO_S\dbl z\dbr$ and $\CO_S\dpl z\dpr$ that acts as the identity on the variable $z$, and is the $\BaseOfD$-Frobenius $b\mapsto (b)^{\#\BaseOfD}$ on local sections $b\in \CO_S$. For a sheaf $M$ of $\CO_S\dbl z\dbr$-modules on $S$ we set $\hat\sigma^* M:=M\otimes_{\CO_S\dbl z\dbr,\hat\sigma^*}\CO_S\dbl z\dbr$. We recall the definition of local shtukas and their quasi-isogenies from \cite[Definition~4.1]{H-V} and \cite[Definition~2.1.1]{HartlPSp}.

\begin{definition}\label{Local Sht}
\begin{enumerate}
\item 
A \emph{local shtuka over $S$} is a pair $(M, \tauLoc)$ consisting of a locally free sheaf $M$ of $\CO_S\dbl z\dbr$-modules of finite rank on $S$ and an isomorphism $\tauLoc\colon\hat{\sigma}^*M\otimes_{\CO_S\dbl z\dbr}\CO_S\dpl z\dpr\isoto M\otimes_{\CO_S\dbl z\dbr}\CO_S\dpl z\dpr$.
\item 
A local shtuka $\ulM:=(M,\tauLoc)$ is called \emph{\'etale} if $\tauLoc$ comes from an isomorphism of $\CO_S\dbl z\dbr$-modules $\hat{\sigma}^\ast M\isoto M$.
\item 
A \emph{morphism} $f\colon(M,\tauLoc)\to(M',\tauLoc')$ of local shtukas over $S$ is a morphism $f\colon M\to M'$ of $\CO_S\dbl z\dbr$-modules which satisfies $\tauLoc'\circ\hat\sigma^*f=f\circ\tauLoc$. We do not require that $f$ is an isomorphism. We denote by $\Sht_\BD(S)$ the category of local shtukas over $S$ and by $\EtSht_\BD(S)$ the category of \'etale local shtukas over $S$.
\end{enumerate}
\end{definition}

\begin{remark} \label{RemVect}
There is an equivalence of categories between the category $\scrH^1(\Spec\BaseOfD,L^+\GL_r)(S)$ and the category of locally free sheaves of $\CO_S\dbl z\dbr$-modules of rank $r$; see \cite[\S4]{H-V}. It induces an equivalence between the category of local $\GL_r$-shtukas over $S$ and the category consisting of local shtukas over $S$ of rank $r$ with isomorphisms as the only morphisms; see \cite[Lemma~4.2]{H-V}. 
\end{remark}

\begin{definition}\label{DefG.1}
A \emph{quasi-isogeny} between two local shtukas $(M,\tauLoc)\to(M',\tauLoc')$ is an isomorphism of $\CO_S\dpl z\dpr$-modules 
$$
f\colon M\otimes_{\CO_S\dbl z\dbr}\CO_S\dpl z\dpr\isoto M'\otimes_{\CO_S\dbl z\dbr}\CO_S\dpl z\dpr
$$ 
with $\tauLoc'\circ\hat\sigma^* f=f\circ\tauLoc$. 
\end{definition}

In analogy with $p$-divisible groups and abelian varieties, one can also assign a Galois representation to a given \'etale local shtuka as follows. Assume that $S$ is connected. Let $\bar{s}$ be a geometric point of $S$ and let $\pi_1^\et(S,\bar{s})$ denote the algebraic fundamental group of $S$ at $\bar{s}$. We define the (\emph{dual}) \emph{Tate functor} from the category of \'etale local shtukas $\EtSht_\BD(S)$ over $S$ to the category $\FF\FM od_{\BaseOfD\dbl z\dbr[\pi_1^\et(S,\bar{s})]}$ of finite free $\BaseOfD\dbl z\dbr$-modules equipped with a continuous action of $\pi_1^\et(S,\bar{s})$ as follows
\begin{eqnarray*}
{\check{T}_{-}}\colon \EtSht_{\BD}(S) &\longto & \FF\FM od_{\BaseOfD\dbl z\dbr[\pi_1^\et(S,\bar{s})]},\\
\ulM:=(M,\tauLoc) &\longmapsto & \check{T}_\ulM :=(M\otimes_{\CO_S\dbl z\dbr} \kappa(\bar{s})\dbl z\dbr)^{\tauLoc}.
\end{eqnarray*}
Here the superscript $\tauLoc$ denotes $\tauLoc$-invariants. Sometimes also the notation $\Koh^1_\et(\ulM,\BaseOfD\dbl z\dbr):=\check{T}_\ulM$ is used. Inverting $z$ we also consider the \emph{rational} (\emph{dual}) \emph{Tate functor}
\begin{eqnarray*}
{\check{V}_{-}}\colon \EtSht_{\BD}(S) &\longto &\FF\FM od_{\BaseOfD\dpl z\dpr[\pi_1^\et(S,\bar{s})]},\\
\ulM:=(M,\tauLoc)&\longmapsto & \check{V}_\ulM :=(M\otimes_{\CO_S\dbl z\dbr} \kappa(\bar{s})\dbl z\dbr)^{\tauLoc}\otimes_{\BaseOfD\dbl z\dbr}\BaseOfD\dpl z\dpr.
\end{eqnarray*}
where $\FF\FM od_{\BaseOfD\dpl z\dpr[\pi_1^\et(S,\bar{s})]}$ denotes the category of finite $\BaseOfD\dpl z\dpr$-vector spaces equipped with a continuous action of $\pi_1^\et(S,\bar{s})$. The functor $\check{V}_{-}$ transforms quasi-isogenies into isomorphisms.

\begin{proposition}\label{PropTateEquiv}
Let $S\in\Nilp_{\BaseOfD\dbl\zeta\dbr}$ be connected. Then the functor $\check{T}_{-}$ is an equivalence between the categories $\EtSht_{\BD}(S)$ and $\FF\FM od_{\BaseOfD\dbl z\dbr[\pi_1^\et(S,\bar{s})]}$. The functor $\check{V}_{-}$ is an equivalence between the category of \'etale local shtukas over $S$ with quasi-isogenies as morphisms and the category $\FF\FM od_{\BaseOfD\dpl z\dpr[\pi_1^\et(S,\bar{s})]}$ with isomorphisms as the only morphisms. There is a canonical isomorphism \mbox{$\check{T}_\ulM\otimes_{\BaseOfD\dbl z\dbr}\kappa(\bar{s})\dbl z\dbr \isoto \ulM\otimes_{\CO_S\dbl z\dbr} \kappa(\bar{s})\dbl z\dbr$} of $\kappa(\bar{s})\dbl z\dbr$-modules which is equivariant for the action of $\pi_1^\et(S,\bar{s})$ and $\tauLoc$, where $\pi_1^\et(S,\bar{s})$ acts trivially on $\ulM$ and $\tauLoc$ acts trivially on $\check{T}_\ulM$.
\end{proposition}

\begin{proof}
The statement for $\check{T}_{-}$ follows by the same arguments as \cite[Proposition~1.3.7]{HartlPSp}. It is analogous to \cite[Proposition~4.1.1]{Kat} and can be thought of as a positive characteristic analog of the Riemann-Hilbert correspondence. We describe the quasi-inverse functor. Consider an $\BaseOfD\dbl z\dbr[\pi_1^\et(S,\bar s)]$-module of rank $r$ and the corresponding representation $\pi\colon\pi_1^\et(S,\bar{s})\to\GL_r(\BaseOfD\dbl z\dbr)$. For each $m\in\BN$ let $S_m\to S$ be the finite Galois covering corresponding to the kernel of $\pi_1^\et(S,\bar{s})\xrightarrow{\;\pi\,}\GL_r(\BaseOfD\dbl z\dbr)\xrightarrow{\;\mod z^m\,}\GL_r\bigl(\BaseOfD\dbl z\dbr/(z^m)\bigr)$. Let $\wt M_m$ be the free module of rank $r$ over $\CO_{S_m}\dbl z\dbr/(z^m)=\CO_{S_m}\otimes_\BaseOfD\BaseOfD\dbl z\dbr/(z^m)$ and equip it with the Frobenius $\tauLoc:=\hat\sigma\otimes\id$ and the action of $\gamma\in\Gal(S_m/S)$ by \mbox{$\gamma(b\otimes f)=\gamma^*(b)\otimes\pi(\gamma^{-1})(f)$} for $b\in\CO_{S_m}$ and $f\in\BaseOfD\dbl z\dbr/(z^m)^{\oplus r}$. Then $\wt M_m$ descends to a locally free $\CO_S\dbl z\dbr/(z^m)$-module $ M_m$ of rank $r$ and $\tauLoc$ descends to an isomorphism $\tauLoc\colon\hat\sigma^* M_m\isoto M_m$. This makes $M:=\invlim M_m$ into an \'etale local shtuka over $S$ and yields the quasi-inverse to $\check{T}_{-}$.

That $\check{V}_{-}$ is essentially surjective follows from the fact that $\pi_1^\et(S,\bar s)$ is compact which implies that every $\BaseOfD\dpl z\dpr[\pi_1^\et(S,\bar{s})]$-module arises by inverting $z$ from an $\BaseOfD\dbl z\dbr[\pi_1^\et(S,\bar{s})]$-module. To see that $\check{V}_{-}$ is fully faithful consider two local shtukas $\ulM$ and $\ulM'$ over $S$ and an isomorphism $f\colon\check{V}_\ulM\isoto\check{V}_{\ulM'}$. There are powers $z^N$ and $z^{N'}$ of $z$ such that $z^Nf$ and $z^{N'}f^{-1}$ come from morphisms $\check{T}_\ulM\to\check{T}_{\ulM'}$ respectively $\check{T}_{\ulM'}\to\check{T}_{\ulM}$. Under the equivalence $\check{T}_{-}$ these in turn come from morphisms $g\colon\ulM\to\ulM'$ and $g'\colon\ulM'\to\ulM$. Then $gg'=z^{N+N'}$ and this implies that $g$ and $g'$ are quasi-isogenies. Clearly $\check{V}_{z^{-N}\!g}=f$ and $\check{V}_{z^{-N'}\!g'}=f^{-1}$. This proves that $\check{V}_{-}$ is an equivalence of categories.
\end{proof}

Let $\Vect_\BD$ be the groupoid over $\Nilp_{\BaseOfD\dbl\zeta\dbr}$ whose $S$-valued points is the category of locally free sheaves of $\CO_S\dbl z\dbr$-modules with isomorphisms as the only morphisms. Let $\Rep_{\BaseOfD\dbl z \dbr}\BP$ be the category of representations $\rho\colon\BP \to \GL(V)$ of $\BP$ in finite free $\BaseOfD\dbl z\dbr$-modules $V$, that is, $\rho$ is a morphism of algebraic groups over $\BaseOfD\dbl z \dbr$. Any such representation $\rho$ gives a functor 
\[
\rho_\ast\colon \scrH^1(\Spec \BaseOfD,L^+\BP)\to \Vect_\BD
\]
which sends an $L^+\BP$-torsor $\CL_+\in\scrH^1(\Spec \BaseOfD,L^+\BP)(S)$ to the sheaf of $\CO_S\dbl z\dbr$-modules associated with the following presheaf
\begin{equation}\label{EqPresheaf}
\Test\;\longmapsto\;\Bigl(\CL_+(\Test)\times\bigl(V\otimes_{\BaseOfD\dbl z\dbr}\CO_S\dbl z\dbr(\Test)\bigr)\Bigr)\big/L^+\BP\bigl(\Test)\,.
\end{equation}
The functor $\rho_\ast\colon \scrH^1(\Spec \BaseOfD,L^+\BP)\to \Vect_\BD$ induces a functor from the category of local $\BP$-shtukas to the category of local shtukas which we likewise denote $\rho_*$. This functor is also compatible with quasi-isogenies.

\begin{definition}\label{DefTateFunctor}
Let $Funct^\otimes(\Rep_{\BaseOfD\dbl z \dbr}\BP, \FF\FM od_{\BaseOfD\dbl z\dbr[\pi_1^\et(S,\bar{s})]})$, resp.\ $Funct^\otimes(\Rep_{\BaseOfD\dbl z \dbr}\BP, \FF\FM od_{\BaseOfD\dpl z\dpr[\pi_1^\et(S,\bar{s})]})$, denote the category whose objects are tensor functors from $\Rep_{\BaseOfD\dbl z \dbr}\BP$ to $\FF\FM od_{\BaseOfD\dbl z\dbr[\pi_1^\et(S,\bar{s})]}$, respectively to $\FF\FM od_{\BaseOfD\dpl z\dpr[\pi_1^\et(S,\bar{s})]}$, and whose morphisms are isomorphisms of functors.
We define the (\emph{dual}) \emph{Tate functor} $\check{\CT}_{-}$, respectively the \emph{rational} (\emph{dual}) \emph{Tate functor} $\check{\CV}_{-}$ as the functors
\begin{eqnarray*}
\check{\CT}_{-}\colon \es \EtSht_{\BP}^{\BD}(S) &\longto & Funct^\otimes (\Rep_{\BaseOfD\dbl z \dbr}\BP,\; \FF\FM od_{\BaseOfD\dbl z\dbr[\pi_1^\et(S,\bar{s})]})\,\\
\ul \CL &\longmapsto & \check{\CT}_{\ul \CL}\colon\; \rho \mapsto \check{T}_{\rho_\ast \ul{\CL}},\\
\check{\CV}_{-}\colon \es \EtSht_{\BP}^{\BD}(S) &\longto & Funct^\otimes (\Rep_{\BaseOfD\dbl z \dbr}\BP,\; \FF\FM od_{\BaseOfD\dpl z\dpr[\pi_1^\et(S,\bar{s})]})\,\\
\ul \CL &\longmapsto & \check{\CV}_{\ul \CL}\colon\; \rho \mapsto \check{V}_{\rho_\ast \ul{\CL}}.
\end{eqnarray*}
\end{definition}

That $\check{\CT}_{-}$ and $\check{\CV}_{-}$ are indeed tensor functors, follows from the fact that $\ul\CL\mapsto\rho_*\ul\CL$ is a tensor functor and from the equivariant isomorphism $\check{T}_{\rho_*\ul\CL}\otimes_{\BaseOfD\dbl z\dbr}\kappa(\bar{s})\dbl z\dbr \isoto \rho_*\ul\CL\otimes_{\CO_S\dbl z\dbr} \kappa(\bar{s})\dbl z\dbr$ from Proposition~\ref{PropTateEquiv}. If $\BP$ is smooth with connected special fiber and $\ul\CL$ is an \'etale local $\BP$-shtuka then the composition of the tensor functor $\check{\CT}_{\ul\CL}$ followed by the forgetful functor $F\colon\FF\FM od_{\BaseOfD\dbl z\dbr[\pi_1^\et(S,\bar{s})]}\to\FF\FM od_{\BaseOfD\dbl z\dbr}$ is isomorphic to the forgetful fiber functor $\omega^\circ\colon\Rep_{\BaseOfD\dbl z \dbr}\BP\to\FF\FM od_{\BaseOfD\dbl z\dbr}$ by Corollary~\ref{CorEtIsTrivial}. Indeed, the base change $\ul\CL_{\bar s}$ of $\ul\CL$ to $\bar s=\Spec\kappa(\bar s)$ is isomorphic to $\ul\BL_0:=\bigl((L^+\BP)_{\bar s},1\!\cdot\!\hat\sigma^*\bigr)$ and the functor $F\circ\check{\CT}_{\ul\BL_0}$ is isomorphic to $\omega^\circ$. This yields a conjugacy class of isomorphisms $\Aut^\otimes(F\circ\check{\CT}_{\ul\CL})\cong\Aut^\otimes(\omega^\circ)=\BP$. Since every $\gamma\in\pi_1^\et(S,\bar s)$ acts as a tensor automorphism of $\check{\CT}_{\ul\CL}$, the tensor functor $\check{\CT}_{\ul\CL}$ corresponds to a conjugacy class of Galois representations $\pi\colon\pi_1^\et(S,\bar s)\to\BP(\BaseOfD\dbl z\dbr)$. Now Proposition~\ref{PropTateEquiv} generalizes as follows.

\begin{proposition}\label{PropTateEquivP}
Let $\BP$ be smooth over $\BD$ with connected special fiber and let $S\in\Nilp_{\BaseOfD\dbl\zeta\dbr}$ be a connected scheme. Then the functor $\check{\CT}_{-}$ is an equivalence between the category $\EtSht_\BP^\BD(S)$ and the category $Funct^\otimes (\Rep_{\BaseOfD\dbl z \dbr}\BP,\; \FF\FM od_{\BaseOfD\dbl z\dbr[\pi_1^\et(S,\bar{s})]})$. The functor $\check{\CV}_{-}$ from the category of \'etale local $\BP$-shtukas over $S$ with quasi-isogenies as morphisms to the category $Funct^\otimes (\Rep_{\BaseOfD\dbl z \dbr}\BP,\; \FF\FM od_{\BaseOfD\dpl z\dpr[\pi_1^\et(S,\bar{s})]})$ is fully faithful.
\end{proposition}

\begin{proof}
To construct the functor which is quasi-inverse to $\check{\CT}_{-}$ we fix a tensor functor $\CF$ in\\ $Funct^\otimes (\Rep_{\BaseOfD\dbl z \dbr}\BP,\; \FF\FM od_{\BaseOfD\dbl z\dbr[\pi_1^\et(S,\bar{s})]})$. The difference of the two $\BaseOfD\dbl z\dbr$-rational fiber functors $\omega^\circ$ and $F\circ\CF$ on $\Rep_{\BaseOfD\dbl z \dbr}\BP$ is given by the torsor $\Isom^\otimes(\omega^\circ,F\circ\CF)$ over $\BP=\Aut^\otimes(\omega^\circ)$; use \cite[Corollary~5.20]{Wed}. Since the special fiber of $\BP$ is connected, this torsor has an $\BaseOfD$-valued point by Lang's theorem~\cite[Theorem~2]{Lang}. Since $\BP$ and hence $\Isom^\otimes(\omega^\circ,F\circ\CF)$ is smooth over $\BD$ this point lifts to an $\BaseOfD\dbl z\dbr$-valued point of $\Isom^\otimes(\omega^\circ,F\circ\CF)$, that is to a tensor isomorphism $\alpha\colon\omega^\circ\isoto F\circ\CF$ over $\BaseOfD\dbl z\dbr$ inducing an isomorphism $\alpha_*\colon\BP\isoto\Aut^\otimes(F\circ\CF)$. Since $\pi_1^\et(S,\bar{s})$ acts as automorphisms of the fiber functor $F\circ\CF$, the functor $\CF$ corresponds to a representation $\pi\colon\pi_1^\et(S,\bar{s})\to\BP(\BaseOfD\dbl z\dbr)$ which depends on $\alpha$ up to conjugation in $\BP$. For each $m\in\BN$ we let $S_m\to S$ be the finite \'etale Galois covering corresponding to the kernel of $\pi_1^\et(S,\bar{s})\xrightarrow{\es\pi\;}\BP(\BaseOfD\dbl z\dbr)\xrightarrow{\es\mod z^m\;}\BP_m\bigl(\BaseOfD\dbl z\dbr/(z^m)\bigr)$ where $\BP_m:=\BP\times_\BD\Spec\BaseOfD\dbl z\dbr/(z^m)$. Let $\wt\CG_m$ be the trivial $\BP_m$-torsor over $S_m\times_\BaseOfD\Spec\BaseOfD\dbl z\dbr/(z^m)$ and equip it with the Frobenius $\tauLoc:=\id\colon\hat\sigma^*\wt\CG_m\isoto\wt\CG_m$. Via the action of $\Gal(S_m/S)$ through $\pi$ on $\wt\CG_m$ the latter descends to a $\BP_m$-torsor $\CG_m$ over $S\times_\BaseOfD\Spec\BaseOfD\dbl z\dbr/(z^m)$ and $\tauLoc$ descends to an isomorphism $\tauLoc\colon\hat\sigma^* \CG_m\to \CG_m.$ This makes $\hat\CG:=\invlim \CG_m$ into a formal $\hat\BP$-torsor over $S$ together with an isomorphism $\tauLoc\colon\hat\sigma^*\hat\CG\isoto\hat\CG$; see Definition~\ref{formal torsor def}. By Proposition~\ref{formal torsor prop} it corresponds to an $L^+\BP$-torsor $\CL_+$ together with an isomorphism $\tauLoc\colon\hat\sigma^*\CL_+\isoto\CL_+$, that is, to the \'etale local $\BP$-shtuka $\ul\CL=(\CL_+,\tauLoc)$. It satisfies $\check{\CT}_{\ul\CL}\cong\CF$. A different isomorphism $\alpha$ gives a different local $\BP$-shtuka which is canonically isomorphic to $\ul\CL$. This yields the quasi-inverse to $\check{\CT}_{-}$.

To prove that $\check{\CV}_{-}$ is fully faithful let $\ul\CL=(\CL_+,\tauLoc)$ and $\ul\CL'=(\CL'_+,\tauLoc')$ be two \'etale local $\BP$-shtukas over $S$ and let $\delta\colon\check{\CV}_{\ul\CL}\isoto\check{\CV}_{\ul\CL'}$ be an isomorphism of tensor functors. We consider the following functor
$$
\wh \CM_{-} \colon \scrH^1(\Spec \BaseOfD,L^+\BP)(S)\to Funct^\otimes(\Rep_{\BaseOfD\dbl z\dbr}\BP, \FM od_{\CO_S\dbl z\dbr}),
$$
which sends an $L^+\BP$-torsor $\CL_+$ to the tensor functor mapping the representation $\rho$ to the $\CO_S\dbl z\dbr$-module $\rho_\ast \CL_+$ from \eqref{EqPresheaf}. By Proposition~\ref{PropTateEquiv} the isomorphism $\delta|_{\rho}$ between $\check{\CV}_{\ul\CL}(\rho)=\check{V}_{\rho_*\ul\CL}$ and $\check{\CV}_{\ul\CL'}(\rho)=\check{V}_{\rho_*\ul\CL'}$ comes from a quasi-isogeny between $\rho_*\ul\CL=\bigl(\wh\CM_{\CL_+}(\rho),\rho_*\tauLoc\bigr)$ and $\rho_*\ul\CL'=\bigl(\wh\CM_{\CL'_+}(\rho),\rho_*\tauLoc'\bigr)$. Therefore the isomorphism $\delta$ induces an isomorphism $\wh\CM_{\CL_+}\otimes_{\CO_S\dbl z\dbr} \CO_S\dpl z\dpr\isoto\wh\CM_{\CL'_+}\otimes_{\CO_S\dbl z\dbr} \CO_S\dpl z\dpr$. Take an \fppf-cover $S'\to S$ trivializing $\CL_+$ and $\CL'_+$ and fix trivializations $\CL_+\cong(L^+\BP)_{S'}\cong\CL'_+$. Then we have
$$
\Isom^\otimes\bigl(\wh \CM_{(\CL_+)_{S'}}\otimes_{\CO_{S'}\dbl z\dbr} \CO_{S'}\dpl z\dpr\,,\,\wh \CM_{(\CL_+)_{S'}}\otimes_{\CO_{S'}\dbl z\dbr} \CO_{S'}\dpl z\dpr\bigr)\;\cong\;\Aut^\otimes(\omega^\circ)\bigl(\CO_{S'}\dpl z\dpr\bigr)\;=\;L\genericG(S'),
$$
because $\wh\CM_{(L^+\BP)_{S'}}=\omega^\circ\otimes_{\BaseOfD\dbl z\dbr} \CO_{S'}\dpl z\dpr$ and $\BP=\Aut^\otimes(\omega^\circ)$ by \cite[Corollary~5.20]{Wed}.
Therefore $\delta$ gives an isomorphism $h_{S'}\colon (L\genericG)_{S'}\to (L\genericG)_{S'}$. The morphism $h_{S'}$ inherits the descent datum coming from the fact that $\delta$ is defined over $S$, and hence it defines an isomorphism $h\colon \CL\isoto\CL'$, where $\CL$ and $\CL'$ denote the $L\genericG$-torsors associated with $\CL_+$ and $\CL'_+$. One easily checks that $h$ satisfies $\tauLoc'\circ \hat{\sigma}^\ast h=h\circ\tauLoc$ and gives a quasi-isogeny $h\colon \ul\CL\to\ul\CL'$. 
\end{proof}

\begin{remark}\label{RemTateEquivP}
In general the functor $\check{\CV}_{-}$ does not need to be an equivalence, not even onto the category of those tensor functors $\CF\colon\Rep_{\BaseOfD\dbl z \dbr}\BP\to \FF\FM od_{\BaseOfD\dpl z\dpr[\pi_1^\et(S,\bar{s})]}$ for which $F\circ\CF\cong\omega^\circ\otimes_{\BaseOfD\dbl z\dbr}\BaseOfD\dpl z\dpr$. For example let $\BP$ be the Iwahori subgroup of $\GL_2$, that is,
\[
\BP(\BaseOfD\dbl z\dbr)=\{\,A\in\GL_2(\BaseOfD\dbl z\dbr)\colon A\equiv \TS\binom{*\:*}{0\:*}\mod z\,\}\,.
\]
Let $x$ be transcendental over $\BaseOfD$ and let $S=\Spec\BaseOfD(x)$. Set $G:=\pi_1^\et(S,\bar{s})=\Gal\bigl(\BaseOfD(x)^\sep/\BaseOfD(x)\bigr)$ and consider a representation $\pi\colon G\to\GL_2(\BaseOfD\dbl z\dbr)$ such that the residual representation $\ol\pi\colon G\to\GL_2(\BaseOfD)$ is irreducible. This implies that $\pi(G)\not\subset\BP(\BaseOfD\dbl z\dbr)$ and hence the tensor functor $\CF\colon\Rep_{\BaseOfD\dbl z \dbr}\BP\to \FF\FM od_{\BaseOfD\dpl z\dpr[G]}$ given by 
\[
\rho\longmapsto\bigl[ G\xrightarrow{\;\pi\,}\GL_2(\BaseOfD\dbl z\dbr)\into\genericG\bigl(\BaseOfD\dpl z\dpr\bigr)\xrightarrow{\;\rho\:}\GL\bigl(\omega^\circ(\rho)\bigr)\bigl(\BaseOfD\dpl z\dpr\bigr)\bigr]
\]
cannot come from a local $\BP$-shtuka over $\BaseOfD(x)$.

Note that such a representation $\pi$ exists. For example if $\phi\colon\BF_q[t]\to\End\BG_{a,\BaseOfD(x)}$ is a Drinfeld-$\BF_q[t]$-module of rank $2$ over $\BaseOfD(x)$ without potential complex multiplication, then for almost all primes $\nu$ of $\BF_q[t]$ the Galois representation $\pi_{\phi,\nu}\colon G\to\GL_2(A_\nu)$ on the $\nu$-adic Tate module of $\phi$ has this property by \cite[Theorem~A]{PT06b} or \cite[Theorem~0.1]{PR09a}. For a concrete example let $\phi_t=1-x\tauGlob+\tauGlob^2$ and $\nu=(t)$. Then $\phi[t]=\{y\in\BaseOfD(x)^\sep\colon y^{q^2}-xy^q+y=0\}$. If $q=2$ it is easy to see that $y^3-xy+1$ is irreducible in $\BaseOfD(x)[y]$ and has splitting field of degree $6$ over $\BaseOfD(x)$. This implies that $\ol\pi_{\phi,(t)}(G)=\GL_2(\BF_2)$. The reason for the failure of $\check{\CV}_{-}$ to be an equivalence of course lies in the fact that the Drinfeld-module $\phi$ does not carry a level structure over $\BaseOfD(x)$ whereas any \'etale local $\BP$-shtuka for the above Iwahori group $\BP$ carries a $\Gamma_0(\nu)$-level structure.

Even if we assume that $\BP$ is a maximal parahoric subgroup of $\genericG$ as in Remark~\ref{Flagisquasiproj}, we expect that it depends on the group $\genericG$ whether $\check\CV_-$ is an equivalence. Namely, in the proof of Proposition~\ref{PropTateEquivP}, when we try to extend the construction of the quasi-inverse of $\check{\CT}_{-}$ to $\check\CV_-$ we obtain a representation $\pi_1^\et(S,\bar s)\to\genericG(\BaseOfD\dpl z\dpr)$. For $\check\CV_-$ to be an equivalence we need that up to conjugation this representation factors through $\BP(\BaseOfD\dbl z\dbr)$. We know that $\pi_1^\et(S,\bar s)$ is a profinite group and hence compact. Therefore the image of the representation is contained in a maximal compact subgroup. So the question arises whether every maximal compact subgroup of $\genericG(\BaseOfD\dpl z\dpr)$ is conjugate to $\BP(\BaseOfD\dbl z\dbr)$. This is true when $\BP=\GL_r$ or $\SL_r$ and in this case $\check\CV_-$ is an equivalence.

But in general the answer may be negative for two reasons. First of all, although every maximal parahoric subgroup is maximally compact, the converse may fail. For example for $\genericG=\PGL_2$ the subgroup generated by the Iwahori subgroup $\{\,A\in\PGL_2(\BaseOfD\dbl z\dbr)\colon A\equiv \binom{*\:*}{0\:*}\mod z\,\}$ and by $\binom{0\:1}{z\:0}$ is maximally compact but not parahoric, because it is the stabilizer of the midpoint of an edge in the Bruhat--Tits tree of $\PGL_2$. This group contains the Iwahori subgroup with index $2$. Secondly, not all maximal parahoric subgroups need to be conjugate, because they are the stabilizers of $0$-simplices in the Bruhat--Tits building, but not all $0$-simplices are conjugate in general. This occurs for example when $\genericG=\Spm_{2r}$.
\end{remark}

%%%%%%%%%%%%%%%%%%%%%%%%%%%%%%%%%%%%%%%%%%%%%%%%%%%%%%%%%%%%%%%%%%%%%%
%
%  Representability of The Rapoport--Zink functor\label{R-Z Space}
%
%%%%%%%%%%%%%%%%%%%%%%%%%%%%%%%%%%%%%%%%%%%%%%%%%%%%%%%%%%%%%%%%%%%%%%

\section{The Rapoport--Zink Spaces for Local \texorpdfstring{$\BP$}{P}-Shtukas}\label{R-Z Space}
\setcounter{equation}{0}

In this chapter we assume that $\BP$ is a smooth affine group scheme over $\BD$. 

Rapoport and Zink  constructed a moduli space for $p$-divisible groups together with a quasi-isogeny to a fixed one (and with some extra structure such as a polarization, endomorphisms, or a level structure). They proved that this moduli space is ind-representable by a formal scheme locally formally of finite type over $\BZ_p$.

We already mentioned that local $\BP$-shtukas behave analogously to $p$-divisible groups. However, this analogy is not perfect, unless we restrict to ``bounded'' local $\BP$-shtukas as the analogous objects corresponding to $p$-divisible groups. More precisely we bound the \emph{Hodge polygon} of a local $\BP$-shtuka $(\CL_+,\tauLoc)$, that is, the relative position of $\hat\sigma^*\CL_+$ and $\CL_+$ under the isomorphism $\tauLoc$; see Definition~\ref{DefBDLocal}\ref{DefBDLocal_B} below. This is motivated by the fact that $F$-isocrystals and Dieudonn\'e-modules also have bounded Hodge slopes. We will show that Rapoport--Zink spaces for bounded local $\BP$-shtukas are formal schemes locally formally of finite type over $\Spf\BaseOfD\dbl\zeta\dbr$. When $\BP:=G_0\times_\BaseOfD\BD$ for a connected split reductive group $G_0$ over $\BaseOfD$ this was proved in \cite[Theorem~6.3]{H-V}. In the non-constant case for a smooth affine group $\BP$ over $\BD$ we will give an axiomatic definition of the boundedness condition in Section~\ref{BC}. We start with the unbounded situation.

%%%%%%%%%%%%%%%%%%%%%%%%%%%%%%%%%%%%%%%%%%%%%%%%%%%%%%%%%%%%%%%%%%%%%%

\subsection{Unbounded Rapoport--Zink Spaces}\label{UnboundedRZ}

For a scheme $\SSS$ in $\Nilp_{\BaseOfD\dbl\zeta\dbr}$ let $\bar{\SSS}$ denote the closed subscheme $\Var_\SSS(\zeta)\subseteq \SSS$. On the other hand for a scheme $\bar \TTT$ over $\BaseOfD$ we set $\wh{\TTT}:=\bar \TTT\whtimes_{\Spec\BaseOfD}\Spf\BaseOfD\dbl\zeta\dbr$. Then $\wh \TTT$ is a $\zeta$-adic formal scheme with $\bar \TTT=\Var_{\wh \TTT}(\zeta)$. So the underlying topological spaces of $\bar\TTT$ and $\wh{\TTT}$ coincide. We let $\Nilp_{\wh{\TTT}}$ be the category of $\wh{\TTT}$-schemes on which $\zeta$ is locally nilpotent.

\begin{definition}\label{RZ space}
With a given local $\BP$-shtuka $\ul{\CL}_0$ over an $\BaseOfD$-scheme $\bar{\TTT}$ we associate the functor
\begin{eqnarray*}
\ul{\CM}_{\ul{\CL}_0}\colon (\Nilp_{\wh{\TTT}})^o &\longto&  \Sets\\
\SSS &\longmapsto & \big\{\text{Isomorphism classes of }(\ul{\CL},\bar{\delta})\colon\;\text{where }\ul{\CL}~\text{is a local $\BP$-shtuka}\\ 
&&~~~~ \text{over $\SSS$ and }\bar{\delta}\colon  \ul{\CL}_{\bar{\SSS}}\to \ul{\CL}_{0,\bar{\SSS}}~\text{is a quasi-isogeny  over $\bar{\SSS}$}\big\}. 
\end{eqnarray*}
Here we say that $(\ul\CL,\bar\ppsi)$ and $(\ul\CL',\bar\ppsi')$ are isomorphic if $\bar\ppsi^{-1}\circ\bar\ppsi'$ lifts via Proposition~\ref{PropRigidityLocal} to an isomorphism $\ul\CL'\isoto\ul\CL$. The group $\QIsog_{\bar \TTT}(\ul{\CL}_0)$ of quasi-isogenies of $\ul{\CL}_0$ acts on the functor $\ul{\CM}_{\ul{\CL}_0}$ via $g\colon(\ul\CL,\bar\delta)\mapsto(\ul\CL,g\circ\bar\delta)$ for $g\in\QIsog_{\bar \TTT}(\ul{\CL}_0)$. We will show that $\ul{\CM}_{\ul{\CL}_0}$ is representable by an ind-scheme which we call an \emph{unbounded Rapoport--Zink space for local $\BP$-shtukas}.
\end{definition}

\begin{remark}\label{RZ_space_remark}
Note that by rigidity of quasi-isogenies (Proposition~\ref{PropRigidityLocal}) the functor $\ul{\CM}_{\ul{\CL}_0}$ is naturally isomorphic to the functor
\begin{eqnarray*}
\SSS &\mapsto & \big\{\text{Isomorphism classes of}~(\ul{\CL},\delta)\colon \;\text{where }\ul{\CL}~\text{is a local $\BP$-shtuka}\\ 
&&~~~~~~~~~\text{over $\SSS$ and}~\delta\colon  \ul{\CL}_\SSS\to \ul{\CL}_{0,\SSS}~\text{is a quasi-isogeny over $\SSS$}\big\}. 
\end{eqnarray*}
This also shows that $\id_{\ul\CL}$ is the only automorphism of $(\ul\CL,\delta)$, and for this reason we do not need to consider $\ul\CM_{\ul\CL_0}$ as a stack.
\end{remark}

\begin{remark}\label{Flagisquasiproj}
In order to show that $\ul{\CM}_{\ul{\CL}_0}$ is representable by an ind-scheme we recall the definition of the affine flag variety $\SpaceFl_\BP$. It is defined to be the $fpqc$-sheaf associated with the presheaf
$$
R\;\longmapsto\; L\genericG(R)/L^+\BP(R)\;=\;\BP\left(R\dpl z \dpr \right)/\BP\left(R\dbl z\dbr\right)
$$ 
on the category of $\BaseOfD$-algebras; compare Definition~\ref{DefLoopGps}. Pappas and Rapoport~\cite[Theorem~1.4]{PR2} show that $\SpaceFl_\BP$ is represented by an ind-scheme which is ind-quasi-projective over $\BaseOfD$, and hence ind-separated and of ind-finite type over $\BaseOfD$. Moreover, they show that the quotient morphism $L\genericG \to \SpaceFl_\BP$ admits sections locally for the \'etale topology. They proceed as follows. When $\BP = \SL_{r,\BD}$, the \fpqc-sheaf $\SpaceFl_\BP$ is called the \emph{affine Grassmanian}. It is an inductive limit of projective schemes over $\BaseOfD$, that is, ind-projective over $\BaseOfD$; see \cite[Theorem~4.5.1]{B-D} or \cite{Faltings03,Ngo-Polo}. By \cite[Proposition~1.3]{PR2} and \cite[Proposition~2.1]{AH_Unif} there is a faithful representation $\BP\into\SL_r$ with quasi-affine quotient. Pappas and Rapoport show in the proof of \cite[Theorem~1.4]{PR2} that $\SpaceFl_{\BP} \to \SpaceFl_{\SL_r}$ is a locally closed embedding, and moreover, if $\SL_r/\BP$ is affine, then $\SpaceFl_{\BP} \to \SpaceFl_{\SL_r}$ is even a closed embedding and $\SpaceFl_{\BP}$ is ind-projective. More generally, if the fibers of $\BP$ over $\BD$ are geometrically connected, it was proved by Richarz \cite[Theorem~A]{Richarz13} that $\SpaceFl_\BP$ is ind-projective if and only if $\BP$ is a parahoric group scheme in the sense of Bruhat and Tits \cite[D\'efinition~5.2.6]{B-T}; see also \cite{H-R}. Note that, in particular, a parahoric group scheme is smooth with connected fibers and reductive generic fiber.
\end{remark}

\medskip

Let us view the formal scheme $\wh\TTT=\bar\TTT\whtimes_\BaseOfD\Spf \BaseOfD\dbl\zeta\dbr$ as the ind-scheme $\dirlim\bar\TTT\times_\BaseOfD\Spec \BaseOfD[\zeta]/(\zeta^{m})$. We may form the fiber product $\wh{\SpaceFl}_{\BP,\wh\TTT}:=\SpaceFl_{\BP}\whtimes_\BaseOfD\wh\TTT$ in the category of ind-schemes (see \cite[7.11.1]{B-D}). Note that this fiber product can be either viewed as the restriction of the sheaf $\SpaceFl_{\BP}$ to the \fpqc-site of schemes in $\Nilp_{\wh\TTT}$ or also as the formal completion of $\SpaceFl_\BP\times_\BaseOfD\,\bar\TTT \times_\BaseOfD \Spec \BaseOfD\dbl\zeta\dbr$ along the special fiber $\Var(\zeta)$.

\begin{theorem} \label{ThmModuliSpX}
If $\BP$ is a smooth affine group scheme over $\BD$ the functor $\ul\CM_{\ul\CL_0}$ from Definition~\ref{RZ space} is represented by an ind-scheme, ind-quasi-projective over $\wh \TTT=\bar \TTT\whtimes_{\BaseOfD}\Spf\BaseOfD\dbl\zeta\dbr$, hence ind-separated and of ind-finite type over $\wh\TTT$. If the fibers of $\BP$ over $\BD$ are connected then $\ul\CM_{\ul\CL_0}$ is ind-projective if and only if $\BP$ is parahoric in the sense of Bruhat and Tits \cite[D\'efinition~5.2.6]{B-T} and \cite{H-R}.

If $\ul\CL_0$ is trivialized by an isomorphism $\alpha\colon\ul\CL_0\isoto\bigl((L^+\BP)_{\bar \TTT},b\hat\sigma^*\bigr)$ over $\bar\TTT$ with $b\in L\genericG(\bar \TTT)$ then $\ul\CM_{\ul\CL_0}$ is represented by the ind-scheme $\wh\SpaceFl_{\BP,\wh\TTT}:=\SpaceFl_\BP\whtimes_\BaseOfD\wh \TTT$.
\end{theorem}

\begin{proof}
We first assume that $\ul\CL_0$ is trivialized by an isomorphism $\alpha$. We regard $\ul{\CM}_{\ul{\CL}_0}$ in the equivalent form mentioned in Remark~\ref{RZ_space_remark}. Consider a pair $(\ul\CL, \delta)=((\CL_+,\tauLoc),\delta)\in\ul{\CM}_{\ul{\CL}_0}(S)$. Choose an \fppf-covering $S'\to S$ which trivializes $\ul\CL$, then the quasi-isogeny $\alpha_{S'}\circ\delta$ is given by an element $g'\in L\genericG(S')$. The image of the element $g'\in L\genericG(S')$ in $\wh{\SpaceFl}_{\BP,\wh\TTT}(S')$ is independent of the choice of the trivialization, and since $(\ul \CL,\delta)$ is defined over $S$, it descends to a point $x \in \wh{\SpaceFl}_{\BP,\wh\TTT}(S)$.
Note in particular that $\tauLoc_{S'}$ is determined by $b$ and $g'$ through the diagram
$$
\CD
\hat{\sigma}^\ast (L\genericG)_{S'}@>{\tauLoc_{S'}}>>(L\genericG)_{S'}\\
@V{\hat{\sigma}^\ast(g')}VV @VV{g'}V\\
\hat{\sigma}^\ast (L\genericG)_{S'}@>b>>(L\genericG)_{S'}.
\endCD
$$

\bigskip

Conversely let $x\in\wh{\SpaceFl}_{\BP,\wh\TTT}(S)$ for a scheme $S$ in $\Nilp_{\wh\TTT}$. The projection morphism $L\genericG \to \SpaceFl_\BP$ admits local sections for the \'etale topology by \cite[Theorem~1.4]{PR2}. Consequently there is an \'etale covering $S'\to S$ such that $x$ is represented by an element $g'\in L\genericG(S')$. We set $S'':=S'\times_SS'$ and define $(\CL_+',\tauLoc',\ppsi')$ over $S'$ as follows. Let $\CL'_+:=(L^+\BP)_{S'}$, let the quasi-isogeny $\ppsi'\colon (\CL_+',\tauLoc')\to\bigl((L^+\BP)_{S'},b\hat\sigma^*\bigr)$ be given by $y\mapsto g'y$, and the Frobenius by $\tauLoc':=(g')^{-1} b\hat\sigma^*(g')\hat\sigma^*$.
We descend $(\CL_+',\tauLoc',\ppsi')$ to $S$. For an $S$-scheme $\Test$ let $\Test'=\Test\times_SS'$ and $\Test''=\Test'\times_\Test\Test'=\Test\times_SS''$, and let $p_i\colon \Test''\to \Test'$ be the projection onto the $i$-th factor. Since $g'$ comes from an element $x\in\wh{\SpaceFl}_{\BP,\wh\TTT}(S)$ there is an $h\in L^+\BP(S'')$ with $p_1^\ast(g')=p_2^\ast(g')\cdot h$. Consider the \fpqc-sheaf $\CL_+$ on $S$ whose sections over an $S$-scheme $\Test$ are given by
\[
\CL_+(\Test)\;:=\;\bigl\{\,y'\in L^+\BP(\Test')\colon \es p_1^\ast(y')=h^{-1}\cdot p_2^\ast(y')\text{ in } L^+\BP(\Test'')\,\bigr\}
\]
on which $L^+\BP(\Test)$ acts by right multiplication. Then $\ul\CL$ is an $L^+\BP$-torsor on $S$ because over $\Test=S'$ there is a trivialization
\[
(L^+\BP)_{S'}\isoto (\CL_+)_{S'}\,,\quad f\mapsto h\,p_1^\ast(f)\es\in\;(L^+\BP)(S'')
\]
due to the cocycle condition for $h$. Moreover, $\tauLoc'$ descends to an isomorphism 
$$
\tauLoc\colon \hat\sigma^*\CL(\Test)\isoto\CL(\Test)\,,\,\hat\sigma^*(y')\mapsto(g')^{-1} b\hat\sigma^*(g')\hat\sigma^*(y')
$$ 
making $(\CL_+,\tauLoc)$ into a local $\BP$-shtuka over $S$. Also $\delta'$ descends to a quasi-isogeny of local $\BP$-shtukas
\begin{eqnarray*}
\delta\colon \CL(\Test) &\ \isoto &\ L\genericG (\Test)=\bigl\{\,f'\in L\genericG(\Test')\colon\;p_1^\ast(f')=p_2^\ast(f')\text{ in }L\genericG(\Test'')\,\bigr\}\;,\quad \\
y' &\ \mapsto &\ g'y'\,.
\end{eqnarray*}
Note that this is well defined. Namely, if $g'$ is replaced by $\tilde g'$ with $u'=(\tilde g')^{-1}g'\in L^+\BP(S')$ then left multiplication with $u'$ defines an isomorphism 
$$
\bigl((L^+\BP)_{S'},(g')^{-1} b\hat\sigma^*(g')\hat\sigma^*,g'\bigr)\isoto\bigl((L^+\BP)_{S'},(\tilde g')^{-1} b\hat\sigma^*(\tilde g')\hat\sigma^*,\tilde g'\bigr).
$$ 
Also $\tilde h=p_2^\ast(u')\,h\,p_1^\ast(u')^{-1}$ and hence left multiplication with $u'$ descends to an isomorphism $\ul\CL\isoto\ul{\wt\CL}$ over $S$. This establishes the last statement of the theorem.

\medskip\noindent
To prove the first assertion we may choose an \fppf-covering $\bar\TTT'\to\bar\TTT$ and a trivialization $\alpha\colon\ul\CL_{0,\bar\TTT'}\isoto\bigl((L^+\BP)_{\bar \TTT'},b\hat\sigma^*\bigr)$ over $\bar\TTT'$. We set $\wh\TTT':=\bar\TTT'\whtimes_{\BaseOfD}\Spf\BaseOfD\dbl\zeta\dbr$ and $\wh\TTT'':=\wh\TTT'\whtimes_{\wh\TTT}\wh\TTT'$, and let $pr_i\colon\wh\TTT''\to\wh\TTT'$ be the projection onto the $i$-th factor. By what we have proved above, we obtain an isomorphism 
$$\ul\CM_{\ul\CL_0}\whtimes_{\wh\TTT}\wh\TTT'\isoto\wh{\SpaceFl}_{\BP,\wh\TTT'},\,\delta\mapsto\alpha_{S'}\circ\delta=g'\cdot L^+\BP(S')\,.$$ 
Note that $\wh{\SpaceFl}_{\BP,\wh\TTT'}$ is an ind-scheme which is ind-quasi-projective over $\wh\TTT'$ by Remark~\ref{Flagisquasiproj} and, when $\BP$ has connected fibers, even ind-projective if and only if $\BP$ is parahoric. Over $\wh\TTT''$ there is an isomorphism
\begin{equation}\label{EqDescent}
\xymatrix @R=0pc {
pr_1^*(\wh{\SpaceFl}_{\BP,\wh\TTT'}) \ar[r]^\cong & \ul\CM_{\ul\CL_0}\whtimes_{\wh\TTT}\wh\TTT'' \ar[r]^{\cong\qquad\qquad} & pr_2^*(\wh{\SpaceFl}_{\BP,\wh\TTT'})\qquad\qquad\qquad\\
g'\cdot L^+\BP(S') \ar@{|->}[r] & pr_1^*\alpha_{S'}^{-1}\circ g' \ar@{|->}[r] & pr_2^*\alpha_{S'}\circ pr_1^*\alpha_{S'}^{-1}\circ g'\cdot L^+\BP(S')
}
\end{equation}
It is given by left multiplication on $pr_1^*(\wh{\SpaceFl}_{\BP,\wh\TTT'})=\wh{\SpaceFl}_{\BP,\wh\TTT''}=pr_2^*(\wh{\SpaceFl}_{\BP,\wh\TTT'})$ with the element $pr_2^*\alpha_{S'}\circ pr_1^*\alpha_{S'}^{-1}\in L^+\BP(\wh\TTT'')$. 

We write $\SpaceFl_\BP=\dirlim\SpaceFl_\BP^{(N)}$ for quasi-projective $\BaseOfD$-schemes $\SpaceFl_\BP^{(N)}$. There is a line bundle $\CF$ on $\SpaceFl_{\BP}$ which is ``ample'' in the sense that its restriction to any $\SpaceFl_\BP^{(N)}$ is ample, and which is equivariant for the $L^+\BP$-action by left multiplication, i.e.\ ``$L^+\BP$-linearized'' in the sense of \cite[Definition~1.6]{GIT}. For example one can take a faithful representation $\BP\into\SL_{r,\BD}$ with quasi-affine quotient (Remark~\ref{Flagisquasiproj}), take the fundamental line bundles $\CL_i$ on $\SpaceFl_{\SL_r}$ from \cite[p.~46]{Faltings03}, take $\CF$ as the pullback of a tensor product of strictly positive powers of the $\CL_i$, see \cite[p.~54]{Faltings03}, and take the $\SpaceFl_\BP^{(N)}$ as the preimages in $\SpaceFl_\BP$ of the Schubert varieties in $\SpaceFl_{\SL_r}$. Then \eqref{EqDescent} defines a descent datum on the pair $(\SpaceFl_\BP^{(N)}\times_\BaseOfD\bar\TTT'\times_\BaseOfD\BaseOfD\dbl\zeta\dbr/(\zeta^m),\,\CF)$ for all $N$ and $m$. This descent datum is effective by \cite[VIII, Proposition~7.8]{SGA1}, see also \cite[\S\,6.1, Theorem~7]{BLR}. Therefore there is a quasi-projective scheme $\ul\CM_{\ul\CL_0}^{(N,m)}$ over $\bar\TTT\times_\BaseOfD\BaseOfD\dbl\zeta\dbr/(\zeta^m)$ with $\ul\CM_{\ul\CL_0}^{(N,m)}\times_{\bar\TTT}\bar\TTT'\,\cong\,\SpaceFl_\BP^{(N)}\times_\BaseOfD\bar\TTT'\times_\BaseOfD\BaseOfD\dbl\zeta\dbr/(\zeta^m)$. Moreover, if the fibers of $\BP$ are connected then $\ul\CM_{\ul\CL_0}^{(N,m)}$ is projective if and only if $\BP$ is parahoric by Remark~\ref{Flagisquasiproj}. It follows that $\ul\CM_{\ul\CL_0}=\dirlim[N,m]\ul\CM_{\ul\CL_0}^{(N,m)}$ is an ind-(quasi-)projective ind-scheme over $\hat\TTT$.
\end{proof}

\subsection{Bounded Local \texorpdfstring{$\BP$}{P}-Shtukas}\label{BC}

We want to introduce boundedness conditions for local $\BP$-shtukas where $\BP$ is again a smooth affine group scheme over $\BD$. Due to the problem discussed in Example~\ref{ExampleBd} below we will base our boundedness conditions on an axiomatic definition of ``bounds''. For this purpose we fix an algebraic closure $\BaseOfD\dpl\zeta\dpr^\alg$ of $\BaseOfD\dpl\zeta\dpr$. Since its ring of integers is not complete we prefer to work with finite extensions of discrete valuation rings $R/\BaseOfD\dbl\zeta\dbr$ such that $R\subset\BaseOfD\dpl\zeta\dpr^\alg$. For such a ring $R$ we denote by $\kappa_R$ its residue field, and we let $\Nilp_R$ be the category of $R$-schemes on which $\zeta$ is locally nilpotent. We also set $\wh{\SpaceFl}_{\BP,R}:=\SpaceFl_\BP\whtimes_{\BaseOfD}\Spf R$ and $\wh{\SpaceFl}_\BP:=\wh{\SpaceFl}_{\BP,\BaseOfD\dbl\zeta\dbr}$. Before we can define ``bounds'' we need to make the following observations.

\begin{definition}\label{DefEqClClosedInd}
(a) For a finite extension $\BaseOfD\dbl\zeta\dbr\subset R\subset\BaseOfD\dpl\zeta\dpr^\alg$ of discrete valuation rings we consider closed ind-subschemes $\hat{Z}_R\subset\wh{\SpaceFl}_{\BP,R}$. We call two closed ind-subschemes $\hat{Z}_R\subset\wh{\SpaceFl}_{\BP,R}$ and $\hat{Z}'_{R'}\subset\wh{\SpaceFl}_{\BP,R'}$ \emph{equivalent} if there is a finite extension of discrete valuation rings $\BaseOfD\dbl\zeta\dbr\subset\wt R\subset\BaseOfD\dpl\zeta\dpr^\alg$ containing $R$ and $R'$ such that $\hat{Z}_R\whtimes_{\Spf R}\Spf\wt R \,=\,\hat{Z}'_{R'}\whtimes_{\Spf R'}\Spf\wt R$ as closed ind-subschemes of $\wh{\SpaceFl}_{\BP,\wt R}$.

\medskip\noindent
(b) Let $\hat{Z}=[\hat{Z}_R]$ be an equivalence class of closed ind-subschemes $\hat{Z}_R\subset\wh{\SpaceFl}_{\BP,R}$ and consider the group $G_{\hat{Z}}:=\{\gamma\in\Aut_{\BaseOfD\dbl\zeta\dbr}(\BaseOfD\dpl\zeta\dpr^\alg)\colon \gamma(\hat{Z})=\hat{Z}\,\}$. We define the \emph{ring of definition $R_{\hat{Z}}$ of $\hat{Z}$} as the intersection of the fixed field of $G_{\hat{Z}}$ in $\BaseOfD\dpl\zeta\dpr^\alg$ with all the finite extensions $R\subset\BaseOfD\dpl\zeta\dpr^\alg$ of $\BaseOfD\dbl\zeta\dbr$ over which a representative $\hat{Z}_R$ of $\hat{Z}$ exists.
\end{definition}

Let us give some explanations for this definition.

\begin{remark}[about Definition~\ref{DefEqClClosedInd}(a)]\label{RemEqClClosedInd}
Consider an ind-scheme structure on $\SpaceFl_\BP$ given as an inductive limit $\SpaceFl_\BP=\dirlim \SpaceFl_\BP^{(m)}$ of quasi-compact $\BaseOfD$-schemes $\SpaceFl_\BP^{(m)}$ indexed by $m\in\BN$. Then $\wh{\SpaceFl}_{\BP,R}= \dirlim \bigl(\SpaceFl_\BP^{(m)}\times_{\BaseOfD}\Spec R/(\zeta^m)\bigr)$ is an ind-scheme structure on $\wh{\SpaceFl}_{\BP,R}$. A closed ind-subscheme $\hat{Z}_R\subset\wh{\SpaceFl}_{\BP,R}$ is of the form $\hat{Z}_R=\dirlim\hat{Z}_{R,m}$ for an inductive system of closed subschemes $\hat{Z}_{R,m}\subset \SpaceFl_\BP^{(m)}\times_{\BaseOfD}\Spec R/(\zeta^m)$. The latter correspond to sheaves of ideals $\CI_m\subset \CO_{\SpaceFl_\BP^{(m)}}\otimes_{\BaseOfD}R/(\zeta^m)$ for all $m$.

If $R\subset\wt R$ and $\hat{Z}_{R}\subset\wh{\SpaceFl}_{\BP,R}$ is a closed ind-subscheme then $\hat{Z}_{\wt R}:=\hat{Z}_R\whtimes_{\Spf R}\Spf\wt R\subset\wh{\SpaceFl}_{\BP,\wt R}$ is a closed ind-subscheme, such that $\hat{Z}_R$ is the ind-scheme theoretic image of $\hat{Z}_{\wt R}$ in $\wh{\SpaceFl}_{\BP,R}$. Indeed, in terms of the ideal sheaves $\CI_m\subset \CO_{\SpaceFl_\BP^{(m)}}\otimes_{\BaseOfD}R/(\zeta^m)$ of $\hat{Z}_R$ and $\wt\CI_m\subset \CO_{\SpaceFl_\BP^{(m)}}\otimes_{\BaseOfD}\wt R/(\zeta^m)$ of $\hat{Z}_{\wt R}$ this means $\wt\CI_m=\CI_m\otimes_R\wt R$ and $\CI_m=\wt\CI_m\cap\CO_{\SpaceFl_\BP^{(m)}}\otimes_{\BaseOfD}R/(\zeta^m)$. The latter equality follows from the commutative diagram
\[
\xymatrix @R-1pc{
0 \ar[r] & **{!U(0.1)} \CI_m \ar[r] \ar@{^{ (}->}[d] & **{!D(0.1)} \CI_m\otimes_R\wt R \ar[r] \ar@{^{ (}->}[d] & **{!D(0.1)} \CI_m\otimes_R\wt R/R \ar[r] \ar@{^{ (}->}[d] & 0\,\;\\
0 \ar[r] & **{!U(0.3)} \CO_{\SpaceFl_\BP^{(m)}}\otimes_{\BaseOfD}R/(\zeta^m) \ar[r] & **{!U(0.2)} \CO_{\SpaceFl_\BP^{(m)}}\otimes_{\BaseOfD}\wt R/(\zeta^m) \ar[r] & **{!U(0.2)} \CO_{\SpaceFl_\BP^{(m)}}\otimes_{\BaseOfD}R/(\zeta^m)\otimes_R\wt R/R \ar[r] & 0\,,
}
\]
in which the rows are exact and the vertical morphisms are injective because $\wt R/R$ and $\wt R$ are finite free $R$-modules.

It follows that two closed ind-subschemes $\hat{Z}_R\subset\wh{\SpaceFl}_{\BP,R}$ and $\hat{Z}'_{R'}\subset\wh{\SpaceFl}_{\BP,R'}$ are equivalent if and only if $\hat{Z}_R\whtimes_{\Spf R}\Spf\wt R \,=\,\hat{Z}'_{R'}\whtimes_{\Spf R'}\Spf\wt R$ for \emph{every} finite extension of discrete valuation rings $\BaseOfD\dbl\zeta\dbr\subset\wt R\subset\BaseOfD\dpl\zeta\dpr^\alg$ containing $R$ and $R'$.

Another consequence is, that a morphism $f\colon S\to\wh{\SpaceFl}_{\BP,R}$ for $S\in\Nilp_R$ factors through $\hat{Z}_R$ if and only if the morphism $f\times\id_{\wt R}\colon S\whtimes_R\Spf\wt R\to\wh{\SpaceFl}_{\BP,\wt R}$ factors through $\hat{Z}_{\wt R}$. Indeed, this can be checked by the vanishing of the ideals $f^*\CI_m$, respectively $(f\times\id_{\wt R})^*\wt\CI_m$, using the injectivity $\CO_S\into\CO_S\otimes_R\wt R$.
\end{remark}

\begin{remark}[about Definition~\ref{DefEqClClosedInd}(b)]\label{RemReflexRing}
Let $R\subset\BaseOfD\dpl\zeta\dpr^\alg$ be a finite extension of $\BaseOfD\dbl\zeta\dbr$ over which a representative $\hat{Z}_R$ of $\hat{Z}$ exists.
\begin{enumerate}
\item \label{RemReflexRing_A}
$\gamma(\hat{Z})=\hat{Z}$ then means that $\gamma(\hat{Z}_R)\subset\wh{\SpaceFl}_{\BP,\gamma(R)}$ is equivalent to $\hat{Z}_R$. In particular, if $\gamma(R)=R$ then, by our previous remark, $\gamma(\hat{Z})=\hat{Z}$ means that $\gamma(\hat{Z}_R)=\hat{Z}_R$. For example if $\Quot(R)$ is a normal field extension of $\BaseOfD\dpl\zeta\dpr$ then $\gamma(R)=R$.
\item \label{RemReflexRing_B}
It follows that $\Aut_R(\BaseOfD\dpl\zeta\dpr^\alg)\subset G_{\hat{Z}}$ because all $R$-automorphisms fix $\hat{Z}_R$.
\item \label{RemReflexRing_C}
We let $R^{G_{\hat{Z}}}:=\{x\in R\colon\gamma(x)=x\text{ for all }\gamma\in G_{\hat{Z}}\,\}$. It equals the intersection of $R$ with the fixed field of $G_{\hat{Z}}$ in $\BaseOfD\dpl\zeta\dpr^\alg$.
\item \label{RemReflexRing_D}
Let $i(R):=[\Quot(R):\BaseOfD\dpl\zeta\dpr]_{\rm insep}$ be the inseparability degree. Then $\BaseOfD\dbl\!\sqrt[i(R)]{\zeta}\,\dbr\subset R^{G_{\hat{Z}}}$ and $R^{G_{\hat{Z}}}$ equals the ring of integers $\CO_K$ in the fixed field $K$ of $G_{\hat{Z}}$ inside the separable closure $\BaseOfD\dpl\!\sqrt[i(R)]{\zeta}\,\dpr^\sep$ of $\BaseOfD\dpl\!\sqrt[i(R)]{\zeta}\,\dpr$. To prove this we use the fact from field theory, that $\BaseOfD\dpl\!\sqrt[i(R)]{\zeta}\,\dpr$ is contained in $\Quot(R)$ and that this is a separable extension. In particular $\Quot(R)\subset\BaseOfD\dpl\!\sqrt[i(R)]{\zeta}\,\dpr^\sep$ and $R^{G_{\hat{Z}}}\subset\CO_K$. From \ref{RemReflexRing_B} it follows that $K\subset\Quot(R)$, and hence $\CO_K\subset R^{G_{\hat{Z}}}$. Finally, since $\Aut_{\BaseOfD\dbl\zeta\dbr}(\BaseOfD\dpl\zeta\dpr^\alg)$ fixes all elements of $\BaseOfD\dpl\!\sqrt[i(R)]{\zeta}\,\dpr$ we find $\BaseOfD\dbl\!\sqrt[i(R)]{\zeta}\,\dbr\subset R^{G_{\hat{Z}}}$.
\item \label{RemReflexRing_E}
If $R'$ is another finite extension of $\BaseOfD\dbl\zeta\dbr$ over which a representative $\hat{Z}_{R'}$ of $\hat{Z}$ exists, such that $i(R)\le i(R')$, then $R^{G_{\hat{Z}}}\subset (R')^{G_{\hat{Z}}}$ by \ref{RemReflexRing_D}. If $i(R)=i(R')$ then $R^{G_{\hat{Z}}}=(R')^{G_{\hat{Z}}}$.
\item \label{RemReflexRing_F}
We conclude from \ref{RemReflexRing_E} that the ring of definition of $\hat{Z}$ may be computed as follows. We choose a finite extension $R\subset\BaseOfD\dpl\zeta\dpr^\alg$ of $\BaseOfD\dbl\zeta\dbr$ over which a representative $\hat{Z}_R$ of $\hat{Z}$ exists, and for which $i(R)$ is minimal. Then $R_{\hat{Z}}=R^{G_{\hat{Z}}}$. Moreover, let $\wt R$ be the ring of integers in the normal closure of $\Quot(R)$ over $\BaseOfD\dpl\zeta\dpr$. Then $i(\wt R) = i(R)$ and therefore $\Quot(\wt R)$ is Galois over $\Quot(R_{\hat{Z}})$ with Galois group \mbox{$\Aut_{R_{\hat{Z}}}(\wt R)\,=\,\{\gamma\in\Aut_{\BaseOfD\dbl\zeta\dbr}(\wt R)\text{ with }\gamma(\hat{Z}_{\wt R})=\hat{Z}_{\wt R}\}\,\subset\,\Aut_{\BaseOfD\dbl\zeta\dbr}(\wt R)$}. We conclude that
\[
R_{\hat{Z}}\;=\;\bigl\{x\in{\wt R}\colon\gamma(x)=x\text{ for all }\gamma\in\Aut_{\BaseOfD\dbl\zeta\dbr}(\wt R)\text{ with }\gamma(\hat{Z}_{\wt R})=\hat{Z}_{\wt R}\bigr\}\,.
\]
\item 
We do not know whether in general $\hat{Z}$ has a representative $\hat{Z}_{R_{\hat{Z}}}$ over the ring of definition $R_{\hat{Z}}$, although this is true in many cases; see our Examples~\ref{ExampleSchubert} to \ref{ExampleBd}.
\item 
Note that our Definition~\ref{DefEqClClosedInd}(b) is a direct translation of the analogous situation over number fields, taking the inseparability problem into account. Namely in the number field case one considers cocharacters $\mu\colon\BG_{m,\BQ_p^\alg}\to G_{\BQ_p^\alg}$ for a reductive group $G$ over $\BQ_p$, and one considers a conjugacy class $\CC(\mu)=\{\Int_g\circ\mu\colon g\in G(\BQ_p^\alg)\}$; see \cite[3.7]{DeligneTravauxDeShimura} or \cite[1.31]{RZ}. (Our Example~\ref{ExampleHV} corresponds to this.) One defines the \emph{field of definition $E_\mu$ of $\CC(\mu)$} as the fixed field inside $\BQ_p^\alg$ of \mbox{$\{\gamma\in\Gal(\BQ_p^\alg/\BQ_p)\colon\gamma(\CC(\mu)):=\CC(\gamma(\mu))=\CC(\mu)\}$}. The field of definition $E_\mu$ is automatically contained in every field over which a representative of $\CC(\mu)$ exists. Our above discussion applies mutatis mutandis. If the group $G$ is quasi-split over $E_\mu$, Kottwitz \cite[Lemma~1.1.3]{Kottwitz84} proved that $\CC(\mu)$ has a representative over $E_\mu$.
\end{enumerate}
\end{remark}

After these preparatory observations we finally come to the announced

\begin{definition}\label{DefBDLocal}
\begin{enumerate}
\item \label{DefBDLocal_A}
We define a \emph{bound} to be an equivalence class $\hat{Z}:=[\hat{Z}_R]$ of closed ind-subschemes $\hat{Z}_R\subset\wh{\SpaceFl}_{\BP,R}$, such that all the ind-subschemes $\hat{Z}_R$ are stable under the left $L^+\BP$-action on $\SpaceFl_\BP$, and the special fibers $Z_R:=\hat{Z}_R\whtimes_{\Spf R}\Spec\kappa_R$ are quasi-compact subschemes of $\SpaceFl_\BP\whtimes_{\BaseOfD}\Spec\kappa_R$. The ring of definition $R_{\hat{Z}}$ of $\hat{Z}$ is called the \emph{reflex ring} of $\hat{Z}$. Since the Galois descent for closed ind-subschemes of $\SpaceFl_\BP$ is effective, the $Z_R$ arise by base change from a unique closed subscheme $Z\subset\SpaceFl_\BP\whtimes_\BaseOfD\kappa_{R_{\hat{Z}}}$. We call $Z$ the \emph{special fiber} of the bound $\hat{Z}$. It is a quasi-projective scheme over $\kappa_{R_{\hat{Z}}}$ by Remark~\ref{Flagisquasiproj} and \cite[Lemma~5.4]{H-V} which implies that every morphism from a quasi-compact scheme to an ind-quasi-projective ind-scheme factors through a quasi-projective subscheme. If $\BP$ is parahoric in the sense of Bruhat and Tits \cite[D\'efinition~5.2.6]{B-T} and \cite{H-R} then $Z$ is projective.
\item \label{DefBDLocal_B}
Let $\hat{Z}$ be a bound with reflex ring $R_{\hat{Z}}$. Let $\CL_+$ and $\CL_+'$ be $L^+\BP$-torsors over a scheme $S$ in $\Nilp_{R_{\hat{Z}}}$ and let $\delta\colon \CL\isoto\CL'$ be an isomorphism of the associated $L\genericG$-torsors. We consider an \fppf-covering $S'\to S$ over which trivializations $\alpha\colon\CL_+\isoto(L^+\BP)_{S'}$ and $\alpha'\colon\CL'_+\isoto(L^+\BP)_{S'}$ exist. Then the automorphism $\alpha'\circ\delta\circ\alpha^{-1}$ of $(L\genericG)_{S'}$ corresponds to a morphism $S'\to L\genericG\whtimes_\BaseOfD\Spf R_{\hat{Z}}$. We say that $\delta$ is \emph{bounded by $\hat{Z}$} if for every such trivialization and for every finite extension $R$ of $\BaseOfD\dbl\zeta\dbr$ over which a representative $\hat{Z}_R$ of $\hat{Z}$ exists the induced morphism 
\[
S'\whtimes_{R_{\hat{Z}}}\Spf R\to L\genericG\whtimes_\BaseOfD\Spf R\to \wh{\SpaceFl}_{\BP,R}
\]
 factors through $\hat{Z}_R$. Furthermore we say that a local $\BP$-shtuka $(\CL_+, \tauLoc)$ is \emph{bounded by $\hat{Z}$} if the isomorphism $\tauLoc$ is bounded by $\hat{Z}$. 
\end{enumerate}
\end{definition}

\begin{remark}\label{RemBound}
The condition of Definition~\ref{DefBDLocal}\ref{DefBDLocal_B} is satisfied for \emph{all} trivializations and for \emph{all} such finite extensions $R$ of $\BaseOfD\dbl\zeta\dbr$ if and only if it is satisfied for \emph{one} trivialization and for \emph{one} such finite extension. Indeed, by the $L^+\BP$-invariance of $\hat Z$ the definition is independent of the trivializations. That one finite extension suffices follows from Remark~\ref{RemEqClClosedInd}. 
\end{remark}

\medskip

In Examples~\ref{ExampleSchubert} to \ref{ExampleBd} below we discuss the motivation for this definition and the relation to other boundedness conditions like in \cite{H-V}. Note that the definition of ``bounds'' given above suffices for our purposes in this article and in \cite{AH_Unif}. For other purposes one may need more restrictive hypotheses on bounds; see for example~\cite[Section~2]{HV3}. 

\begin{remark}\label{RemBdIsFormSch}
Let the ind-scheme structure on $\SpaceFl_\BP$ be given as the limit $\SpaceFl_\BP=\dirlim\SpaceFl_\BP^{(m)}$ and the one on $\wh\SpaceFl_{\BP,R}$ as $\wh\SpaceFl_{\BP,R}=\dirlim\SpaceFl_\BP^{(m)}\times_\BaseOfD\Spec R/(\zeta^m)$. Let $\hat Z_R^{(m)}:=\hat Z_R\whtimes_{\wh\SpaceFl_{\BP,R}}(\SpaceFl_\BP^{(m)}\times_\BaseOfD\Spec R/(\zeta^m))$. Then $\hat Z_R=\dirlim\hat Z_R^{(m)}$ and $\hat Z_{R,\red}:=\dirlim(\hat Z_R^{(m)})_\red=\dirlim(\hat Z_R\whtimes_{\Spf R}\Spec\kappa_R\times_{\SpaceFl_\BP}\SpaceFl_\BP^{(m)})_\red=(Z_R)_\red$ is a scheme. This means that $\hat Z_R$ is a ``reasonable formal scheme'' over $\Spf R$ in the sense of \cite[7.11.1 and 7.12.17]{B-D}, and hence a formal scheme in the sense of \cite[I$_{\rm new}$]{EGA}.
\end{remark}

\begin{proposition}\label{PropBoundedClosed}
Let $\hat Z$ be a bound with reflex ring $R_{\hat{Z}}$. Let $\CL_+$ and $\CL'_+$ be $L^+\BP$-torsors over a scheme $S\in\Nilp_{R_{\hat{Z}}}$ and let $\delta\colon \CL\isoto\CL'$ be an isomorphism of the associated $L\genericG$-torsors. Then the condition that $\delta$ is bounded by $\hat Z$ is represented by a closed subscheme of $S$.
\end{proposition}

\begin{proof}
We consider a representative $\hat{Z}_R$ of the bound $\hat{Z}$ over a finite exten\-sion $R_{\hat{Z}}\subset R\subset\BaseOfD\dpl\zeta\dpr^\alg$. As in Definition~\ref{DefBDLocal} we consider trivializations of $\CL_+$ and $\CL'_+$ over an \fppf-covering $S'\to S$ and the induced morphism \mbox{$S'\whtimes_{R_{\hat{Z}}}\Spf R\to L\genericG\whtimes_\BaseOfD\Spf R\to \wh{\SpaceFl}_{\BP,R}$}. Due to the $L^+\BP$-invariance of $\hat Z_R$ the closed subscheme $S'\whtimes_{\wh{\SpaceFl}_{\BP,R}}\hat Z_R$ of $S'$ descends to a closed subscheme of $S$. By Remark~\ref{RemBound} this closed subscheme represents the boundedness by $\hat{Z}$.
\end{proof}

\begin{example}\label{ExampleSchubert}
Assume that $\BP$ is parahoric in the sense of Bruhat and Tits \cite[D\'efinition~5.2.6]{B-T} and \cite{H-R}; see Remark~\ref{Flagisquasiproj}. Consider the base change $\genericG_L$ of $\genericG$ to $L=\BaseOfD^\alg\dpl z\dpr$. Let $A$ be a maximal split torus in $\genericG_L$ and let $T$ be its centralizer. Since $\BaseOfD^\alg$ is algebraically closed, $\genericG_L$ is quasi-split by \cite[\S\,II.2.3, Th\'eor\`eme~$1'$ and Remarque~1, p.~140]{SerreCohoGal} and so $T$ is a maximal torus in $\genericG_L$. Let $N = N(T)$ be the normalizer of $T$ and let $\CT^0$ be the identity component of the N\'eron model of $T$ over $\CO_L=\BaseOfD^\alg\dbl z\dbr$.

The \emph{Iwahori--Weyl group} associated with $A$ is the quotient group $\wt{W}= N(L)\slash\CT^0(\CO_L)$. It is an extension of the finite Weyl group $W_0 = N(L)/T(L)$ by the coinvariants $X_\ast(T)_I$ under $I=\Gal(L^\sep/L)$:
$$
0 \to X_\ast(T)_I \to \wt W \to W_0 \to 1.
$$
By \cite[Proposition~8]{H-R} there is a bijection
\begin{equation}\label{EqSchubertCell}
L^+\BP(\BaseOfD^\alg)\backslash L\genericG(\BaseOfD^\alg)/L^+\BP(\BaseOfD^\alg) \isoto \wt{W}^\BP  \backslash \wt{W}\slash \wt{W}^\BP
\end{equation}
where $\wt{W}^\BP := (N(L)\cap \BP(\CO_L))\slash \CT^0(\CO_L)$.

Let $\omega\in \wt{W}^\BP\backslash \wt{W}/\wt{W}^\BP$ and let $\BaseOfD_\omega$ be the fixed field in $\BaseOfD^\alg$ of $\{\gamma\in\Gal(\BaseOfD^\alg/\BaseOfD)\colon \gamma(\omega)=\omega\}$. We show that $\omega$ has a representative $g_\omega\in L\genericG(\BaseOfD_\omega)$. Indeed, let $g\in L\genericG(\BaseOfD^\alg)$ be any representative of $\omega$ and let $\gamma$ be the $\BaseOfD_\omega$-Frobenius which generates $\Gal(\BaseOfD^\alg/\BaseOfD_\omega)$. Since $\gamma(\omega)=\omega$ there are elements $b_1,b_2\in L^+\BP(\BaseOfD^\alg)$ with $\gamma(g)=b_1^{-1}g\,b_2$. By Lemma~\ref{LSisnonemptyA} we find elements $c_1,c_2\in L^+\BP(\BaseOfD^\alg)$ with $b_i\,\gamma(c_i)=c_i$ for $i=1,2$. Then $g_\omega:=c_1^{-1}g\,c_2=\gamma(c_1^{-1}g\,c_2)\in L\genericG(\BaseOfD_\omega)$ is the desired representative of $\omega$ over $\BaseOfD_\omega$. Clearly, by definition of $\BaseOfD_\omega$ there are no representatives of $\omega$ over proper subfields of $\BaseOfD_\omega$.

We define the \emph{Schubert variety $\CS(\omega)$ associated with $\omega$} as the ind-scheme theoretic closure of the $L^+\BP$-orbit of $g_\omega$ in $\SpaceFl_\BP\whtimes_{\BaseOfD}\BaseOfD_\omega$. It is a reduced projective variety over $\BaseOfD_\omega$. For further details see \cite{PR2} and \cite{Richarz}. The equivalence class of $\hat{Z}_{\BaseOfD_\omega\dbl\zeta\dbr}:=\CS(\omega)\whtimes_{\BaseOfD_\omega}\Spf \BaseOfD_\omega\dbl\zeta\dbr$ defines a bound with reflex ring $\BaseOfD_\omega\dbl\zeta\dbr$. Instead of ``bounded by $[\hat{Z}_{\BaseOfD_\omega\dbl\zeta\dbr}]$'' we also say ``\emph{bounded by $\omega$}'' in this case.  
\end{example}

\begin{example}\label{ExampleHV}
In \cite{H-V}, Viehmann and the second author considered the case where $\BP=G\times_{\BaseOfD}\BD$ for a split connected reductive group $G$ over $\BaseOfD$. In this case $\wt W^\BP=W_0$ and $\wt{W}^\BP  \backslash \wt{W}\slash \wt{W}^\BP=X_*(T)$, and any element $\mu\in X_*(T)$ has a representative over $\BaseOfD_\mu=\BaseOfD$. If $\mu\in X_*(T)$ one could consider the bound $\hat Z:=\bigl[\CS(\mu)\whtimes_\BaseOfD\Spf \BaseOfD\dbl\zeta\dbr\bigr]$ as in Example~\ref{ExampleSchubert} above.

However, in \cite{H-V} we proceeded differently and instead fixed a Borel subgroup $B$ of $G$ and its opposite Borel $\olB$. We considered a finite generating system $\Lambda$ of the monoid of dominant weights $X^*(T)_\dom$, and for all $\lambda\in\Lambda$ the Weyl module $V_\lambda:=\bigl(\Ind_\olB^G(-\lambda)_\dom\bigr)\dual$. For the representation \mbox{$\rho_\lambda\colon G\to\GL(V_\lambda)$} we considered the sheaves of $\CO_S\dbl z\dbr$-modules $\rho_{\lambda*}\CL_+$ and $\rho_{\lambda*}\CL'_+$ associated in \eqref{EqPresheaf} with the $L^+\BP$-torsors $\CL_+$ and $\CL'_+$ over $S$. (For the definition of $\CO_S\dbl z\dbr$ see Chapter~\ref{GalRepSht}.) The isomorphism $\delta$ of the associated $L\genericG$-torsors corresponds to an isomorphism $\rho_{\lambda*}\delta\colon\rho_{\lambda*}\CL_+\otimes_{\CO_S\dbl z\dbr}\CO_S\dpl z\dpr\isoto\rho_{\lambda*}\CL'_+\otimes_{\CO_S\dbl z\dbr}\CO_S\dpl z\dpr$. We said in \cite[Definition~3.5]{H-V} that ``$\delta$ is bounded by $(\mu,\tilde z)$'', where $\tilde z=z$ or $\tilde z=z-\zeta$, if for all $\lambda\in\Lambda$
\begin{equation}\label{EqBoundHV1}
\rho_{\lambda*}\delta\,(\rho_{\lambda*}\CL_+)\;\subset\;\tilde z\;^{-\langle\,(-\lambda)_\dom,\mu\rangle}\,(\rho_{\lambda*}\CL_+'),
\end{equation}
and if for all geometric points $\bar s$ of $S$ the image of the isomorphism $\delta_{\bar s}$ at $\bar s$ under the isomorphism \eqref{EqSchubertCell} has the same image in $\pi_1(G)$ than $\mu$. Note that the more important case $\tilde z=z-\zeta$ is useful to define bounds on local $\BP$-shtukas, while the case $\tilde z=z$ is only useful to define bounds on quasi-isogenies between local $\BP$-shtukas. In that sense the bound $\bigl[\CS(\mu)\whtimes_\BaseOfD\Spf \BaseOfD\dbl\zeta\dbr\bigr]$ from Example~\ref{ExampleSchubert}, which corresponds to $\tilde z=z$, is not the right one to define bounds on local $\BP$-shtukas. Further note that in case $\tilde z=z-\zeta$ the term $\tilde z\;^{-\langle\,(-\lambda)_\dom,\mu\rangle}$ in \eqref{EqBoundHV1} can be viewed as the image under $(-\lambda)_\dom\colon LG(S)\to L\BG_m(S)$ of the element $\mu(z-\zeta)^{-1}\in LG(S)$, which itself is the image of $\mu(z-\zeta)^{-1}\in G\bigl(\BaseOfD\dbl\zeta,z\dbr[\tfrac{1}{z-\zeta}]\bigr)$ in $LG(S)=G\bigl(\CO_S\dbl z\dbr[\tfrac{1}{z-\zeta}]\bigr)$ using that $\zeta$ is locally nilpotent on $S$. 

In terms of Definition~\ref{DefBDLocal} the boundedness condition \eqref{EqBoundHV1} can be described as follows. Consider the universal matrix $M\in L\GL(V_\lambda)(S_\lambda)$ over the ind-scheme $S_\lambda:=L\GL(V_\lambda)\whtimes_\BaseOfD\Spf\BaseOfD\dbl\zeta\dbr$. Let $\olS_\lambda$ be the closed ind-subscheme of $S_\lambda$ defined by the condition that the matrix $\tilde z^{\,\langle\,(-\lambda)_\dom,\mu\rangle}M$ has entries in $\CO_{S_\lambda}\dbl z\dbr$, and let 
\[
\hat Z_\lambda\,:=\,\olS_\lambda/(L^+\GL(V_\lambda)\whtimes_\BaseOfD\Spf\BaseOfD\dbl\zeta\dbr)\,\subset\,\wh\SpaceFl_{\GL(V_\lambda)}\,. 
\]
Let $\mu^\#\in\pi_1(G)$ be the image of $\mu$ in the fundamental group $\pi_1(G)$ and let $(\wh\SpaceFl_\BP)_{\mu^\#}$ be the connected component of $\wh\SpaceFl_\BP$ corresponding to $\mu^\#$ under the isomorphism $\pi_0(\wh\SpaceFl_\BP)\cong\pi_1(G)$; see \cite[Proposition~4.5.4]{B-D} or \cite[Theorem~0.1]{PR2}. Write $\Lambda=\{\lambda_1,\ldots,\lambda_m\}$ and for each $\lambda_i$ consider the morphism $\rho_{\lambda_i*}\colon(\wh\SpaceFl_\BP)_{\mu^\#}\to\wh\SpaceFl_{\GL(V_{\lambda_i})}$ induced from $\rho_{\lambda_i}$. Then the base change $\hat Z$ of the closed ind-subscheme $\hat Z_{\lambda_1}\whtimes_{\Spf\BaseOfD\dbl\zeta\dbr}\ldots\whtimes_{\Spf\BaseOfD\dbl\zeta\dbr}\hat Z_{\lambda_m}$ under the morphism $\prod_i\rho_{\lambda_i*}\colon(\wh\SpaceFl_\BP)_{\mu^\#}\to\wh\SpaceFl_{\GL(V_{\lambda_1})}\whtimes_{\BaseOfD\dbl\zeta\dbr}\ldots\whtimes_{\BaseOfD\dbl\zeta\dbr}\wh\SpaceFl_{\GL(V_{\lambda_m})}$ is the bound representing the ``boundedness by $(\mu,\tilde z)$'' from \cite[Definition~3.5]{H-V}. It has reflex ring $\BaseOfD\dbl\zeta\dbr$ and $\hat{Z}$ is a representative of this bound over the reflex ring.

Instead of the Weyl modules $V_\lambda$ one could of course also work with other representations of $G$. If for example one takes the induced modules $\Ind_\olB^G(-\lambda)_\dom$, or tilting modules, one obtains different ind-subschemes $\hat Z$, but the underlying reduced ind-subschemes of these $\hat Z$ still coincide. If $\tilde z=z$, this reduced ind-subscheme equals the Schubert variety $\CS(\mu)\whtimes_\BaseOfD\Spec\BaseOfD\dbl\zeta\dbr$. This already indicates that it is reasonable to consider boundedness in terms of closed ind-subschemes of $\wh\SpaceFl_\BP$. Note that here also for $\tilde z=z-\zeta$ the bound $\hat Z$ only depends on $\mu$ and the class of modules considered (Weyl modules, etc.). This is no longer true for general $\BP$ as one sees from the next example.
\end{example}

\begin{example}\label{ExampleBd}
We discuss a special case of Example~\ref{ExampleSchubert}. Assume that $\charakt \BF_q\neq 2$ and set $K:=\BF_q\dpl z\dpr$. Let $E:=K[y]/(y^2-z)$ be the ramified quadratic field extension with $y^2=z$. Let $T$ be the one dimensional torus $\ker(N_{E/K}\colon \Res_{E/K}\BG_m\rightarrow\BG_m)$. Explicitly $T=\Spec K[a,b]/(a^2-b^2z-1)$, with the multiplication $(a,b)*(c,d)=(ac+bdz,ad+bc)$. Sending  $a\mapsto \frac{1}{2}(t+t^{-1})$ and $b\mapsto \frac{1}{2y}(t^{-1}-t)$ defines an isomorphism $\BG_{m,E}=\Spec E[t,t^{-1}] \cong T_E$ which we will use in the sequel to identify $X_*(T):=X_*(T_E)$ with $\BZ$. Here the finite Weyl group $W_0=(1)$ is trivial and the inertia group $I=\Gal(E/K)=\{1,\gamma\}$ acts on $X_*(T)= \BZ$ via $\gamma(\mu)=-\mu$. Therefore $\wt W=X_*(T)_I=\BZ/2\BZ$. 

Consider the N\'eron-model $\CT=\ker (N_{\cO_E/\cO_K}\colon \Res_{\cO_E/\cO_K}\BG_m \rightarrow \BG_m)$. As a scheme it is isomorphic to $\Spec\BF_q\dbl z\dbr[a,b]/(a^2-b^2z-1)$. Its special fiber has two connected components distinguished by $a\equiv 1$ or $-1\mod z$. Therefore the connected component of identity of $\CT$ is 
\[
\CT^0:=\Spec \BF_q\dbl z\dbr[a',b]/(2a'+z(a')^2-b^2), 
\]
where $a=1+za'$. In particular $(-1,0)\notin\CT^0\bigl(\BF_q\dpl z\dpr\bigr)$. We take $\BP$ as the group scheme $\CT^0$ over $\BD=\Spec\BF_q\dbl z\dbr$. Here $\BaseOfD=\BF_q$ and $\hat{\sigma}=\sigma$ is the $q$-Frobenius. By \cite[Lemma~5 and its proof]{H-R} the group $\CT^0$ is the unique parahoric group scheme with generic fiber $T$.

In this example the isomorphism \eqref{EqSchubertCell} is given by the Kottwitz map $\kappa_T\colon LT(\BF_q^\alg)\to X_*(T)_I$, see \cite[2.5]{Ko1}. Its inverse associates with each element of $X_*(T)_I$ a $\hat\sigma$-conjugacy class in $LT(\BF_q^\alg)$. For example if we take $\bar\mu=1\in X_*(T)_I=\BZ/2\BZ$ one has to choose a lift $\mu\in X_*(T_E)$. If we choose $\mu=1$ then with $\bar\mu=1$ it associates 
\begin{eqnarray*}
N_{E/K}(\mu(y))&=&\mu(y)\cdot\gamma(\mu(y))\\[2mm]
&=&\Bigl(\frac{1}{2}(y+y^{-1}),\frac{1}{2y}(y^{-1}-y)\Bigr)\cdot\Bigl(\frac{1}{2}(-y-y^{-1}),\frac{1}{-2y}(-y^{-1}-(-y))\Bigr)\\[2mm]
&=&(-1,0)\in T(\BF_q^{\alg}\dpl z\dpr).
\end{eqnarray*}
The $\hat\sigma$-conjugacy class given by $(-1,0)$ is independent of the choice of $\mu$ and of the uniformizer $y$ (and of $E$) by \cite[2.5]{Ko1}. The Schubert variety $\CS(\bar\mu)$ for $\bar\mu=1$ from Example~\ref{ExampleSchubert} therefore equals 
\[
L^+\CT^0(\BF_q)\cdot(-1,0)\cdot L^+\CT^0(\BF_q)/L^+\CT^0(\BF_q).
\]
However, as we have mentioned in the discussion after equation~\eqref{EqBoundHV1}, the bound $\bigl[\CS(\bar\mu)\whtimes_{\BF_q}\Spf\BF_q\dbl\zeta\dbr\bigr]$ is only useful to bound quasi-isogenies between local $\BP$-shtukas.

Instead we want to define a bound $[\hat{Z}_R]$ which is useful to bound local $\BP$-shtukas, and whose fiber $Z_R:=\hat{Z}_R\whtimes_{\Spf R}\Spec\kappa_R$ over $\kappa_R$ equals the Schubert variety $\CS(\bar\mu)\times_{\BF_q}\Spec\kappa_R\subset\SpaceFl\times_{\BF_q}\Spec\kappa_R$. In Example~\ref{ExampleHV} we were able to achieve this by lifting $\mu(z)\in LG(\BF_q)$ to an element $\mu(z-\zeta)\in G\bigr(\BF_q\dbl z,\zeta\dbr[\tfrac{1}{z-\zeta}]\bigr)$; see the discussion after \eqref{EqBoundHV1}. This is not possible here. We can only lift $(-1,0)$ over the ramified extension $\BF_q\dbl\zeta\dbr[\eeeta]/(\eeeta^2-\zeta)$ of $\BF_q\dbl\zeta\dbr$. Namely, consider the isomorphism $K=\BF_q\dpl z\dpr\rightarrow \BF_q\dpl\zeta\dpr, z\mapsto \zeta$. Fix an extension $i\colon E\hookrightarrow \BF_q\dpl\zeta\dpr^{\alg}$ of this isomorphism, set $\eeeta :=i(y)$ and lift $N_{E/K}(\mu(y))$ to $g(E,\mu,i):=N_{E/K}\bigl(\mu\bigl(y-i(y)\bigr)\bigr)$. For example if $\mu=1\in\BZ= X_*(T_E)$ we compute
\begin{eqnarray*} 
& \left(\alpha, \beta \right)\;:=\; g(E,1,i)\;=\;\mu(y-\eeeta)\cdot \gamma(\mu(y-\eeeta))\;= \\[3mm]
& \DS\left(\frac{(y-\eeeta)+(y-\eeeta)^{-1}}{2}\,,\frac{(y-\eeeta)^{-1}-(y-\eeeta)}{2y}\right)
\cdot\gamma\left(\frac{(y-\eeeta)+(y-\eeeta)^{-1}}{2}\,,\frac{(y-\eeeta)^{-1}-(y-\eeeta)}{2y}\right),
\end{eqnarray*}
with $\gamma(y)=-y$ and $\gamma|_{\BF_q\dpl\eeeta\dpr}=\id$. Then
\begin{eqnarray*}
\alpha&=&\tfrac{1}{4}\left( (y-\eeeta)+(y-\eeeta)^{-1}\right) \left( (-y-\eeeta)+(-y-\eeeta)^{-1}\right)\\
&&-\tfrac{1}{4}\left( (y-\eeeta)^{-1}-(y-\eeeta)\right) \left( (-y-\eeeta)^{-1}-(-y-\eeeta)\right)\\[2mm]
&=&\frac{1}{2}\left(\dfrac{(y-\eeeta)^{2}+(-y-\eeeta)^{2}}{(-y-\eeeta)(y-\eeeta)} \right)\\[2mm]
&=&\DS\frac{\zeta+z}{\zeta-z},
\end{eqnarray*}
and
\begin{eqnarray*}
\beta&=&\tfrac{-1}{4y}\left( (y-\eeeta)+(y-\eeeta)^{-1}\right) \left( (-y-\eeeta)^{-1}-(-y-\eeeta)\right)\\[1mm]
&&+\tfrac{1}{4y}\left( (y-\eeeta)^{-1}-(y-\eeeta)\right) \left( (-y-\eeeta)+(-y-\eeeta)^{-1}\right)\\[2mm]
&=&\DS\frac{2\eeeta}{\zeta-z}.
\end{eqnarray*}
Thus, as a lift of $(-1,0)$ we get
$$ 
g(E,1,i)=N_{E/K}(\mu(y-\eeeta))=\left(\frac{\zeta+z}{\zeta-z},\frac{2\eeeta}{\zeta-z}\right).
$$ 
This shows that we can define the desired bound by 
\[
\hat Z_{\mu,i,\BF_q\dbl\eeeta\dbr}:=\ol{L^+\CT^0\cdot g(E,\mu,i)\cdot L^+\CT^0/L^+\CT^0}\,\subset\, \SpaceFl_{\CT^0,\BF_q\dbl\eeeta\dbr}.
\]
However, this bound depends on the choice of $\mu$ and of the embedding $i$. We first compute how $g(E,\mu,i)$ depends on the chosen embedding $i$. For this purpose we compute $\dfrac{N_{E/K}(\mu(y-i(y)))}{N_{E/K}(\mu(y-i\circ \gamma(y))}$. Changing $i$ to $i\circ\gamma$ replaces $\eeeta$ by $-\eeeta$ and we have
$$ 
g(E,1,i\circ \gamma)= N_{E/K}(\mu(y+\eeeta))= \left(\frac{\zeta+z}{\zeta-z},\frac{-2\eeeta}{\zeta-z}\right).
$$
Note that $g(E,1,i)=g(E,1,i\circ\gamma)^{-1}$. Therefore 
$$
\frac{g(E,1,i)}{g(E,1,i \circ \gamma)}=g(E,1,i)^2=\left(\frac{(z+\zeta)^2+4\zeta z}{(\zeta-z)^2},\frac{4\eeeta(z+\zeta)}{(\zeta-z)^2}\right).
$$
This also shows what happens if we replace $\mu\in X_*(T_E)=\BZ$ by another lift of $\bar\mu\in\BZ/2\BZ$, i.e.
$$
\frac{g(E,\mu+2,i)}{g(E,\mu,i)}=g(E,2,i)=g(E,1,i)^2.
$$

Observe that $\DS\frac{g(E,1,i)}{g(E,1,i\circ \gamma)} \in \CT^0(\BF_q\dbl z, \tfrac{\eeeta}{z} \dbr)\setminus\CT^0(\BF_q\dbl\eeeta, z\dbr)$. So the element $g(E,1,\gamma\circ i)$ does not lie in the closure of the subscheme 
\[
L^+\CT^0\cdot g(E,1,i)\cdot L^+\CT^0\,\subset\, L\CT^0\whtimes_{\Spec\BF_q}\Spf\BF_q\dbl\eeeta\dbr. 
\]
In particular the bounds $\hat Z_{\mu,i}:=[\hat Z_{\mu,i,\BF_q\dbl\eeeta\dbr}]$ depend on the chosen embedding $i\colon E\hookrightarrow \BF_q\dpl\zeta\dpr^{\alg}$ and on the lift $\mu\in X_*(T)$ of $\bar\mu\in X_*(T)_I$. For this reason we decided to treat bounds axiomatically in Definition~\ref{DefBDLocal}. Our discussion also shows that the reflex ring of the bound $\hat{Z}_{\mu,i}$ is $\BF_q\dbl\eeeta\dbr$, because $\wt\gamma(\hat{Z}_{\mu,i})=\hat{Z}_{\mu,i\circ\gamma}\ne\hat{Z}_{\mu,i}$ for the non-trivial element $\wt\gamma\in\Aut_{\BF_q\dbl\zeta\dbr}(\BF_q\dbl\xi\dbr)$.
\end{example}

%%%%%%%%%%%%%%%%%%%%%%%%%%%%%%%%%%%%%%%%%%%%%%%%%%%%%%%%%%%%%%%%%%%%%%%%%%%%%%%%

\subsection{Representability of the Bounded Rapoport--Zink Functor}\label{Rep_R-Z}

In this section we assume that $\BP$ is a smooth affine group scheme over $\BD$ with connected reductive generic fiber $\genericG$. Let $b\in L\genericG(\BaseFldOfLocSht)$ for some field $\BaseFldOfLocSht\in\Nilp_{\BaseOfD\dbl\zeta\dbr}$. 
With $b$ Kottwitz associates a slope homomorphism
$$
\nu_b\colon  D_{\BaseFldOfLocSht\dpl z\dpr}\to \genericG_{\BaseFldOfLocSht\dpl z\dpr},
$$
called Newton polygon of $b$; see \cite[4.2]{Ko1}. Here $D$ is the diagonalizable pro-algebraic group over ${\BaseFldOfLocSht\dpl z\dpr}$ with character group $\BQ$. The slope homomorphism is characterized by assigning the slope filtration of $(V\otimes_{\BaseOfD\dpl z\dpr}{\BaseFldOfLocSht\dpl z\dpr},\rho(b)\cdot(\id\otimes \hat{\sigma}))$ to any $\BaseOfD\dpl z\dpr$-rational representation $(V,\rho)$ of $\genericG$; see \cite[Section~3]{Ko1}.  
We assume that $b\in L\genericG(\BaseFldOfLocSht)$ satisfies a \emph{decency equation for a positive integer $s$}, that is,
\begin{equation}\label{EqDecency}
(b\hat{\sigma})^s\;=\;s\nu_b(z)\,\hat{\sigma}^s\qquad\text{in}\quad L\genericG(\BaseFldOfLocSht)\rtimes \langle\hat{\sigma}\rangle.
\end{equation}

\begin{remark}\label{RemDecent}
Assume that $b\in L\genericG(\BaseFldOfLocSht)$ is decent with the integer $s$ and let $\ell\subset\BaseFldOfLocSht^\alg$ be the finite field extension of $\BaseOfD$ of degree $s$. Then $b\in L\genericG(\ell)$ because by \eqref{EqDecency} the element $b$ has values in the fixed field of $\hat\sigma^s$ which is $\ell$. Note that if $\BaseFldOfLocSht$ is algebraically closed and the generic fiber $\genericG$ of $\BP$ is connected reductive, any $\hat{\sigma}$-conjugacy class in $L\genericG(\BaseFldOfLocSht)$ contains an element satisfying a decency equation; see \cite[4.3]{Ko1} and use \cite[\S\,II.2.3, Th\'eor\`eme~$1'$ and Remarque~1, p.~140]{SerreCohoGal} instead of Steinberg's theorem. 
\end{remark}

\begin{remark}\label{RemJb}
With the element $b\in L\genericG(\BaseFldOfLocSht)$ one can associate a connected algebraic group $J_b$ over $\BaseOfD\dpl z\dpr$ which is defined by its functor of points that assigns to an $\BaseOfD\dpl z\dpr$-algebra $R$ the group
$$
J_b(R):=\big\{g \in \genericG(R\otimes_{\BaseOfD\dpl z\dpr} {\BaseFldOfLocSht\dpl z\dpr})\colon g^{-1}b\hat{\sigma}(g)=b\big\}.
$$
Let $b$ satisfy a decency equation for the integer $s$ and let $F_s$ be the fixed field of $\hat{\sigma}^s$ in ${\BaseOfD^\alg\dpl z\dpr}$. Then $\nu_b$ is defined over $F_s$ and $J_b\times_{\BaseOfD\dpl z\dpr}F_s$ is the centralizer of the 1-parameter subgroup $s\nu_b$ of $\genericG$ and hence a Levi subgroup of $\genericG_{F_s}$; see \cite[Corollary~1.14]{RZ}. In particular $J_b(\BaseOfD\dpl z\dpr)\subset \genericG(F_s)\subset L\genericG(\ell)$ where $\ell$ is the finite field extension of $\BaseOfD$ of degree $s$.
\end{remark}

\medskip

In the remaining part of the chapter we consider the bounded Rapoport--Zink functor and prove that it is ind-representable by a formal scheme in the following important special situation. Let $\hat{Z}$ be a bound with reflex ring $R_{\hat{Z}}=\kappa\dbl\xi\dbr$ and special fiber $Z\subset\SpaceFl_\BP\whtimes_\BaseOfD\Spec\kappa$ ; see Definition~\ref{DefBDLocal}. Let $\ul\BL_0=(L^+\BP,b\hat{\sigma}^*)$ be a trivialized local $\BP$-shtuka over a field $\BaseFldOfLocSht$ in $\Nilp_{\BaseOfD\dbl\zeta\dbr}$. Assume that $b$ is decent with integer $s$ and let $\ell\subset\BaseFldOfLocSht^\alg$ be the compositum of the residue field $\kappa$ of $R_{\hat{Z}}$ and the finite field extension of $\BaseOfD$ of degree $s$. Then $b\in L\genericG(\ell)$ by Remark~\ref{RemDecent}. So $\ul\BL_0$ is defined over $\ell$ and we may replace $\BaseFldOfLocSht$ by $\ell$. Note that $\ell\dbl\xi\dbr$ is the unramified extension of $R_{\hat{Z}}$ with residue field $\ell$.

\renewcommand{\BaseFldOfLocSht}{\ell}

\begin{definition}
Keep the notation from above and set $\bar\TTT:=\Spec\BaseFldOfLocSht$ and $\ul\CL_0:=\ul\BL_0$. \\
(a) Consider the base change $\ul\CM_{\ul\BL_0}\whtimes_{\BaseFldOfLocSht\dbl\zeta\dbr}\Spf\BaseFldOfLocSht\dbl\xi\dbr$ of the functor $\ul\CM_{\ul\BL_0}$ from Definition~\ref{RZ space} and its subfunctor
\begin{eqnarray*}
\ul{\CM}_{\ul{\BL}_0}^{\hat{Z}}\colon (\Nilp_{\BaseFldOfLocSht\dbl\xi\dbr})^o &\longto & \Sets\\
S&\longmapsto & \Bigl\{\,\text{Isomorphism classes of }(\ul\CL,\bar\ppsi)\in\ul{\CM}_{\ul{\BL}_0}(S)\colon\ul\CL\text{~is bounded by $\hat{Z}$}\Bigr\}
\end{eqnarray*}
Note that the functor $\ul\CM_{\ul\BL_0}\whtimes_{\BaseFldOfLocSht\dbl\zeta\dbr}\Spf\BaseFldOfLocSht\dbl\xi\dbr$ is represented by the ind-scheme $\wh\SpaceFl_{\BP,\BaseFldOfLocSht\dbl\xi\dbr}:=\SpaceFl_{\BP}\whtimes_{\BaseOfD}\Spf\BaseFldOfLocSht\dbl\xi\dbr$ by Theorem~\ref{ThmModuliSpX}, and $\ul{\CM}_{\ul{\BL}_0}^{\hat{Z}}$ is a closed ind-subscheme by Proposition~\ref{PropBoundedClosed}.

\medskip\noindent
(b) We define the associated \emph{affine Deligne--Lusztig variety} over $\ell$ as the reduced closed ind-subscheme $X_Z(b)\subset\SpaceFl_\BP\whtimes_\BaseOfD\Spec\BaseFldOfLocSht$ whose $K$-valued points (for any field extension $K$ of $\BaseFldOfLocSht$) are given by
$$
X_Z(b)(K):=\big\{ g\in \SpaceFl_\BP(K)\colon g^{-1}\,b\,\hat\sigma^\ast(g) \in Z(K)\big\}.
$$
If $\omega \in \wt{W}$ and $Z=\CS(\omega)$ is the Schubert variety from Example~\ref{ExampleSchubert}, we set $X_{\preceq\omega}(b):=X_{\CS(\omega)}(b)$.
\end{definition}

\begin{theorem} \label{ThmRRZSp}
If $\BP$ is a smooth affine group scheme over $\BD$ with connected reductive generic fiber, the functor $\RZ\colon (\Nilp_{\BaseFldOfLocSht\dbl\xi\dbr})^o\to\Sets$ is ind-representable by a formal scheme over $\Spf \BaseFldOfLocSht\dbl\xi\dbr$ which is locally formally of finite type and separated. Its underlying reduced subscheme equals $X_{Z}(b)$. In particular $X_Z(b)$ is a scheme locally of finite type and separated over $\BaseFldOfLocSht$. The formal scheme representing $\RZ$ is called a \emph{bounded Rapoport--Zink space for local $\BP$-shtukas}.
\end{theorem}

Recall that a formal scheme over $\BaseFldOfLocSht\dbl\xi\dbr$ in the sense of \cite[I$_{new}$, 10]{EGA} is called \emph{locally formally of finite type} if it is locally noetherian and adic and its reduced subscheme is locally of finite type over $\BaseFldOfLocSht$. It is called \emph{formally of finite type} if in addition it is quasi-compact.

\begin{remark}\label{RemJActsOnRZ}
By our assumptions $\QIsog_{\BaseFldOfLocSht}(\ul{\BL}_0)$ equals the group $J_b(\BaseOfD\dpl z\dpr)$ from Remark~\ref{RemJb}.  This group acts on the functor $\RZ$ via $g\colon(\ul\CL,\bar\delta)\mapsto(\ul\CL,g\circ\bar\delta)$ for $g\in\QIsog_{\BaseFldOfLocSht}(\ul{\BL}_0)$.
\end{remark}

\begin{proof}[Proof of Theorem~\ref{ThmRRZSp}]
The proof will use a sequence of lemmas and will eventually be complete after Lemma~\ref{UnifDist2}. Consider the universal local $\BP$-shtuka $\ul{\CL}_{univ}$ over $\wh{\SpaceFl}_{\BP,\BaseFldOfLocSht\dbl\xi\dbr}$ (see Theorem~\ref{ThmModuliSpX}). Let $\CM_{\ul{\BL}_0}^{\hat{Z}}$ be the closed ind-subscheme of $\wh{\SpaceFl}_{\BP,\BaseFldOfLocSht\dbl\xi\dbr}$ over which $\ul{\CL}_{univ}$ is bounded by $\hat{Z}$; see Proposition~\ref{PropBoundedClosed}. By construction $\CM_{\ul{\BL}_0}^{\hat{Z}}$ ind-represents the functor $\ul{\CM}_{\ul{\BL}_0}^{\hat{Z}}$. It is clear that the reduced ind-subscheme equals $X_Z(b)$. We have to show that $\CM_{\ul{\BL}_0}^{\hat{Z}}$ is a formal scheme locally formally of finite type and separated. Note that the separatedness over $\Spf\ell\dbl\xi\dbr$ follows from the ind-separatedness of $\wh{\SpaceFl}_{\BP,\BaseFldOfLocSht\dbl\xi\dbr}$; see Theorem~\ref{ThmModuliSpX}. By rigidity of quasi-isogenies the functor $\RZ$ is equivalent to the following functor
\begin{eqnarray*}
(\Nilp_{\BaseFldOfLocSht\dbl\xi\dbr})^o &\to & \Sets\\
S&\longmapsto & \Bigl\{\,(\ul\CL,\ppsi)\colon 
 \ul\CL \text{ is a local $\BP$-shtuka over $S$ bounded by $\hat{Z}$}\\ 
&&\qquad \text{ and }\ppsi\colon \ul\CL\to\ul\BL_{0,S}\text{ is a quasi-isogeny}\,\Bigr\}/\sim.
\end{eqnarray*}
We take a representation $\iota\colon  \BP \to \SL_{r,\BD}=:H$ with quasi-affine quotient $H/\BP$; see \cite[Proposition~1.3]{PR2} and \cite[Proposition~2.1]{AH_Unif}. It induces a 1-morphism 
$$
\iota_*:=\scrH^1(\iota)\colon \scrH^1(\Spec\BaseOfD, L^+\BP)(S)\to \scrH^1(\Spec\BaseOfD,L^+H)(S).
$$
For an $L^+\BP$-torsor $\CL_+$ over $S$ we denote by $\CV(\iota_*\CL_+)$ the sheaf of $\CO_S\dbl z\dbr$-modules associated with the image of $\CL_+$ in $\scrH^1(\Spec\BaseOfD,L^+\GL_r)(S)$ by Remark~\ref{RemVect}. In particular $\CV(\iota_*(L^+\BP)_S)=\CO_S\dbl z\dbr^{\oplus r}$. Let $\rho\dual$ be the half-sum of all positive coroots of $H$ with respect to the Borel subgroup of upper triangular matrices in $H=\SL_r$.

For each non-negative integer $n\in \BN_0$ let $\CM^n$ be the closed ind-subscheme of $\RZ$ defined by the following sub functor of $\RZ$
\begin{eqnarray*}
\ul{\CM}^n\colon (\Nilp_{\BaseFldOfLocSht\dbl\xi\dbr})^o &\longto & \Sets\\
S&\longmapsto & \Bigl\{\,(\ul\CL,\ppsi)\in \RZ (S)\colon\scrH^1(\iota)(\delta) \text{ is bounded by $2n\rho\dual$ }\,\Bigr\},
\end{eqnarray*}
where $\ul\CL=(\CL_+,\tauLoc) $ and where we say that \emph{$\scrH^1(\iota)(\delta)$ is bounded by $2n\rho\dual$} if for all $j=1,\ldots,r$
\begin{equation}\label{EqBdHV}
\TS\bigwedge^j_{\CO_S\dbl z\dbr}\scrH^1(\iota)(\delta)\bigl(\CV(\iota_*\CL_+)\bigr)\;\subset\;z^{n(j^2-jr)}\cdot\bigwedge^j_{\CO_S\dbl z\dbr} \CV(\iota_*(L^+\BP)_S)\,.
\end{equation}
By \cite[Lemma~4.3]{H-V} the latter is equivalent to the boundedness condition considered in \cite[Definition~3.5]{H-V}, see Example~\ref{ExampleHV}, because $\rho\dual=(r-1,\ldots,1-r)$ and \eqref{EqBdHV} is automatically an equality for $j=r$ as $\iota$ factors through $H$.

\begin{lemma} \label{M^n}
The ind-scheme $\CM^n$ representing the functor $\ul\CM^n$ is a $\xi$-adic noetherian formal scheme over $\BaseFldOfLocSht\dbl\xi\dbr$, whose underlying topological space $(\CM^n)_\red$ is a quasi-projective scheme over $\Spec\BaseFldOfLocSht$ and even projective if $\BP$ is parahoric in the sense of Bruhat and Tits \cite[D\'efinition~5.2.6]{B-T} and \cite{H-R}.
\end{lemma}

\begin{proof}
Since $H/\BP$ is quasi-affine, the induced morphism $\iota_*\colon \wh{\SpaceFl}_{\BP,\BaseFldOfLocSht\dbl\xi\dbr} \to \wh{\SpaceFl}_{H,\BaseFldOfLocSht\dbl\xi\dbr}$ is a locally closed embedding by Remark~\ref{Flagisquasiproj}. The representation $\iota$ induces a functor $\iota_\ast$ from local $\BP$-shtukas to local $H$-shtukas. Consider the local $H$-shtuka $\ul\BH_0:= \iota_\ast \ul\BL_0=((L^+H)_\BaseFldOfLocSht,\iota(b)\hat\sigma^*)$ over $\BaseFldOfLocSht$ and view $\wh{\SpaceFl}_{H,\BaseFldOfLocSht\dbl\xi\dbr}$ as a moduli space representing the functor $\ul \CM_{\ul\BH_0}$, parametrizing local $H$-shtukas together with a quasi-isogeny $\delta_H$ to $\ul\BH_0$; see Theorem~\ref{ThmModuliSpX}. Let $\wh{\SpaceFl}_{H,\BaseFldOfLocSht\dbl\xi\dbr}^{(n)}$ be the closed ind-subscheme of $\wh{\SpaceFl}_{H,\BaseFldOfLocSht\dbl\xi\dbr}$ defined by condition~\eqref{EqBdHV}, that is, by bounding $\delta_H$ by $2n\rho \dual$. It is a $\xi$-adic noetherian formal scheme over $\Spf\BaseFldOfLocSht\dbl\xi\dbr$ by \cite[Proposition~5.5]{H-V} whose underlying topological space is a projective scheme over $\Spec\BaseFldOfLocSht$. Thus for all $i$
\[
\CM^n(i)\;:=\;\CM^n\whtimes_{\Spf\BaseFldOfLocSht\dbl\xi\dbr}\Spec\BaseFldOfLocSht\dbl\xi\dbr/(\xi^i)\;=\; \ul\CM_{\ul\BL_0}^{\hat Z}\whtimes_{\wh{\SpaceFl}_{H,\BaseFldOfLocSht\dbl\xi\dbr}}\wh{\SpaceFl}_{H,\BaseFldOfLocSht\dbl\xi\dbr}^{(n)}\whtimes_{\Spf {\BaseFldOfLocSht\dbl\xi\dbr}} \Spec \BaseFldOfLocSht\dbl\xi\dbr/(\xi^i)
\]
is a locally closed subscheme of $\wh{\SpaceFl}_{H,\BaseFldOfLocSht\dbl\xi\dbr}^{(n)}\whtimes_{\Spf {\BaseFldOfLocSht\dbl\xi\dbr}} \Spec \BaseFldOfLocSht\dbl\xi\dbr/(\xi^i)$, and hence a scheme of finite type over $\Spec \BaseFldOfLocSht\dbl\xi\dbr/(\xi^i)$ with underlying topological space $\CM^n(1)$ independent of $i$. Moreover, $\CM^n(1)$ is \mbox{(quasi-)} projective, because it is closed in the ind-quasi-projective ind-scheme $\SpaceFl_\BP\whtimes_\BaseOfD\Spec\BaseFldOfLocSht$ which is even projective if $\BP$ is parahoric. Now our claim follows from \cite[I$_{\rm new}$, Corollary~10.6.4]{EGA}.   
\end{proof}

\bigskip

For each non-negative integer $n$ we define the following sub functor of $\RZ$
\begin{eqnarray*}
\ul{\CM}_n\colon (\Nilp_{\BaseFldOfLocSht\dbl\xi\dbr})^o &\longto & \Sets\\
S&\longmapsto & \Bigl\{\,(\ul\CL,\ppsi)\in \RZ (S)\colon\text{ for any point $s$ in $S$,} \\ 
&&\text{\quad}~\scrH(\iota)(\delta_s) \text{ is bounded by}~2n\rho\dual \,\Bigr\}.
\end{eqnarray*}
This functor is represented by an ind-scheme $\CM_n$ which is  the formal completion of $\RZ$ along the quasi-compact closed subscheme $(\CM^n)_{red}$.

\begin{lemma}\label{LemmaMn}
$\CM_n$ is a formal scheme formally of finite type over $\Spf \BaseFldOfLocSht\dbl\xi\dbr$.
\end{lemma}

To prove the lemma we need to start with the following definition. Recall that $R_{\hat{Z}}=\kappa\dbl\xi\dbr$.

\begin{definition}
Let $R=\invlim R_m\in\Nilp_{\kappa\dbl\xi\dbr}$ where $(R_m,u_{m,m'})$ is a projective system of discrete rings indexed by $\BN_0$. Suppose that all homomorphisms $R\to R_m$ are surjective, and that all kernels $I_m:=\ker (u_{m,0}\colon R_m\to R_0)\subset R_m$ are nilpotent. A local $\BP$-shtuka over $\Spf R$ is a projective system $(\ul\CL_m)_{m\in\BN_0}$ of local $\BP$-shtukas $\ul\CL_m$ over $R_m$ with $\ul\CL_{m-1}\cong\ul\CL_m\otimes_{R_m}R_{m-1}$.
\end{definition}

\begin{lemma}\label{remdejong}
Let $R$ in $\Nilp_{\kappa\dbl\xi\dbr}$ be as in the above definition. The pull back functor defines an equivalence between the category of local $\BP$-shtukas over $\Spec R$  bounded by $\hat{Z}$  and the category of local $\BP$-shtukas over $\Spf R$  bounded by $\hat{Z}$.
\end{lemma}

\begin{proof}
Since $R$ is in $\Nilp_{\kappa\dbl\xi\dbr}$ there is an integer $e\in \BN$ such that $\xi^e=0$ on $R$. Let $\ul{\CL}:=(\ul{\CL}_m)_{m\in \BN_0}$ be a local $\BP$-shtuka over $\Spf R$. By Proposition~\ref{formal torsor prop} there is an \'etale covering $R_0'\to R_0$ which trivializes $\ul{\CL}_0$.  By \cite[Th\'eor\`eme I.8.3]{SGA1} there is a unique \'etale $R$-algebra $R'$ with $R'\otimes_R R_0\cong R_0'$. As in \cite[Proposition~2.2(c)]{H-V} this gives rise to compatible trivializations $\ul{\CL}_m \otimes_{R_m} R_m'\cong ((L^+\BP)_{R_m'},b_m\hat{\sigma}^\ast)$ over $R'_m:= R'\otimes_R R_m$ for all $m$. Here $b_m \in L\genericG(R_m')$ and $b_m \otimes_{R_m'} R_{m-1}' =b_{m-1}$.

By \cite[Proposition~1.3]{PR2} and \cite[Proposition~2.1]{AH_Unif} we may take a faithful representation $\iota\colon\BP \hookrightarrow \SL_{r,\BD}$ and consider the induced closed immersion $L\genericG \into L\SL_{r,\BD}$. The ind-scheme structure on $\SpaceFl_{\SL_{r,\BD}}$ is given by $\SpaceFl_{\SL_{r,\BD}}=\dirlim\SpaceFl^{(n)}_{\SL_{r,\BD}}$ where $\SpaceFl^{(n)}_{\SL_{r,\BD}}$ is defined by condition~\eqref{EqBdHV}. Let $L\SL_{r,\BD}^{(n)}=L\SL_{r,\BD}\times_{\SpaceFl_{\SL_{r,\BD}}}\SpaceFl_{\SL_{r,\BD}}^{(n)}$. Then 
\[
L\SL_{r,\BD}^{(n)}(S)\;=\;\bigl\{\,g\in L\SL_{r,\BD}(S)\colon \text{ all $j\times j$-minors of $g$ lie in $z^{n(j^2-jr)}\CO_S(S)\dbl z\dbr\es\forall\;j$}\,\bigr\}.
\]
This implies that $L\SL_{r,\BD}^{(n)}$, and hence also $L\genericG^{(n)}:=L\genericG\times_{L\SL_{r,\BD}}L\SL_{r,\BD}^{(n)}$ is an infinite dimensional affine scheme over $\BaseOfD$. By Remark~\ref{RemReflexRing}\ref{RemReflexRing_F} we may choose a representative $\hat{Z}_{\wt{R}}\subset\wh\SpaceFl_{\BP,\wt{R}}$ of the bound $\hat{Z}$ over a finite extension $\wt{R}$ of $R_{\hat{Z}}=\kappa\dbl\xi\dbr$ with $\Quot(\wt R)$ Galois over $\Quot(R_{\hat{Z}})$. By Remark~\ref{RemBound} the boundedness by $\hat{Z}$ can be checked using $\hat{Z}_{\wt R}$. Since $Z_{\wt{R},e}:=\hat{Z}_{\wt{R}}\whtimes_{\Spf \wt{R}} \Spec \wt{R}/(\xi^e)$ has the same underlying topological space as $Z_{\wt R}:=\hat{Z}_{\wt R}\whtimes_{\wt R}\Spec\kappa_{\wt R}$, it is quasi-compact. So there is an $n\in \BN$ such that $Z_{\wt{R},e}\subseteq \SpaceFl^{(n)}_{\SL_{r,\BD}}\whtimes_{\BaseOfD}\Spec \wt{R}/(\xi^e)$ by \cite[Lemma~5.4]{H-V}. As one sees from the following diagram
\[
\xymatrix @C+1pc {
\Spec R_m'\otimes_{\kappa\dbl\xi\dbr}\wt{R}\ar[d]\ar@{-->}^{b_m\times\id\qquad}[r] & L\SL^{(n)}_{r,\BD}\;\whtimes_{\BaseOfD}\;\Spec \wt{R}/(\xi^e) \ar[d] \ar[r] & L\SL_{r,\BD}\;\whtimes_{\BaseOfD}\;\Spec \wt{R}/(\xi^e) \ar[d]\\
Z_{\wt{R},e} \ar[r] & \SpaceFl^{(n)}_{\SL_{r,\BD}}\;\whtimes_{\BaseOfD}\;\Spec \wt{R}/(\xi^e) \ar[r]& \SpaceFl_{\SL_{r,\BD}}\;\whtimes_{\BaseOfD}\;\Spec \wt{R}/(\xi^e)\,,
}
\] 
the morphism $b_m\times\id\colon \Spec R_m'\otimes_{\kappa\dbl\xi\dbr}\wt{R}\to L\genericG$ factors through $L\genericG^{(n)}=L\genericG\times_{L\SL_{r,\BD}}L\SL_{r,\BD}^{(n)}$ for all $m$. Since $L\genericG^{(n)}$ is affine, the compatible collection of morphisms $b_m\colon \Spec R_m'\otimes_{\kappa\dbl\xi\dbr}\wt{R}\to L\genericG^{(n)}$ is given by a compatible collection of homomorphisms $\CO(L\genericG^{(n)})\to R'_m\otimes_{\kappa\dbl\xi\dbr}\wt{R}$. It corresponds to a homomorphism $\tilde b_\infty^*\colon\CO(L\genericG^{(n)})\to R'\otimes_{\kappa\dbl\xi\dbr}\wt{R}$, because $R'=\invlim R'_m$ and $\wt R$ is a finite free $\kappa\dbl\xi\dbr$-module. By Remark~\ref{RemReflexRing}\ref{RemReflexRing_F}, $\gamma(\hat{Z}_{\wt{R}})=\hat{Z}_{\wt{R}}$ for all $\gamma\in\Aut_{\kappa\dbl\xi\dbr}(\wt R)$ and thus by construction the homomorphism $\tilde b_\infty^*$ is invariant under $\Aut_{\kappa\dbl\xi\dbr}(\wt R)$. It follows that $\tilde b_\infty^*$ factors through a homomorphism $b_\infty^*\colon\CO(L\genericG^{(n)})\to R'$. The latter corresponds to a morphism $b_\infty\colon\Spec R'\to L\genericG^{(n)}$. This gives the local $\BP$-shtuka $((L^+\BP)_{R'},b_\infty\hat{\sigma}^\ast)$ over $\Spec R'$ bounded by $\hat{Z}$, which carries a descent datum from the $\ul{\CL}_m$ and hence induces a local $\BP$-shtuka over $\Spec R$.
\end{proof}

Let us come back to the 

\begin{proof}[Proof of Lemma~\ref{LemmaMn}]
For each $m\geq n$ let $\CM_n^m$ be the formal completion of $\CM^m$ along $(\CM_n)_\red$. It is a noetherian adic formal scheme over $\BaseFldOfLocSht\dbl\xi\dbr$. Let $U$ be an affine open subscheme of $(\CM_n)_\red$. By \cite[I$_{\rm new}$, Proposition~2.3.5]{EGA} this defines an affine open formal subscheme $\Spf R_m$ of $\CM_n^m$ with underlying set $U$.  Let $R$ be the inverse limit of the projective system $R_{m+1}\to R_m$ and let $\Fa_m\subset R$ denote the ideal such that $R_m=R\slash \Fa_m$. Let $J$ be the inverse image in $R$ of the largest ideal of definition in $R_n$. We want to show that $R$ is $J$-adic. We make the following

\medskip\noindent
\emph{Claim.} For any integer $c>0$ there is an integer $m_0$ such that for any $m\geq m_0$ the natural map $R_{m}/ J^c R_m \to R_{m_0}/ J^c R_{m_0}$ is an isomorphism.

\medskip\noindent
To prove the claim let $\ul\CL_m$ be the universal bounded local $\BP$-shtuka over $\Spf R_m$. Consider the local $\BP$-shtuka $(\ul\CL_m)_m$ over $\Spf R$ and its pullback over $\Spf R/J^c$. 
The latter comes from a local $\BP$-shtuka $\ul\CL$ over $\Spec R/J^c\in\Nilp_{\BaseFldOfLocSht\dbl\xi\dbr}$ by Lemma~\ref{remdejong}.
By rigidity of quasi-isogenies the quasi-isogeny $\delta_n$ over \mbox{$R\slash J=R_n/J$} lifts to a quasi-isogeny $\delta$ over $R\slash J^c$. Since $\Spec R/J^c$ is quasi-compact, the quasi-isogeny $\scrH^1(\iota)(\delta)$ satisfies condition~\eqref{EqBdHV} for some $m_0$, that is, it is bounded by $2m_0\rho\dual$. By the universal property of $\CM_n^{m_0}$ the tuple $(\ul\CL, \delta)$ induces a morphism $R_{m_0}\to R/J^c$ making the following diagram commutative, from which the claim follows
\[
\xymatrix @C+3pc {
R\ar@{->>}[d] \ar@{->>}[r] & R_{m}\ar@{->>}[d] \ar@{->>}[r]& R_{m_0}\ar@{-->}[dll]\ar@{->>}[d]\\
R\slash J^c \ar@{->>}[r] & R_{m}\slash J^c R_{m}\ar@{->>}[r] & R_{m_0}\slash J^c R_{m_0}.
}
\]

The claim implies that the chain $\Fa_n+J^c \supseteq \Fa_{n+1}+J^c \supseteq \ldots \supseteq \Fa_i+J^c \supseteq \ldots$ stabilizes. Now set $\CJ_c:=\bigcap_i\Fa_i+J^c=\Fa_m+J^c$ for $m\gg0$ and consider the descending chain $\CJ_1 \supseteq \CJ_2 \supseteq \ldots$. Note that $\CJ_1=J$ and $\CJ_{c+1}+\CJ_1^c=\CJ_c$. Since $\CJ_1\slash \CJ_2\cong JR_m/J^2R_m$ for $m\gg0$, it is a finitely generated $R$-module. Therefore $\CM_n$ is a locally noetherian adic formal scheme over $\Spf\BaseFldOfLocSht\dbl\xi\dbr$ by \cite[Proposition~2.5]{RZ}. It is formally of finite type because $(\CM_n)_\red=(\CM^n)_\red$ is quasi-projective over $\BaseFldOfLocSht$ by Lemma~\ref{M^n}.
\end{proof}

From now on we use that $\ul\BL_0$ is decent with the integer $s$. In the sequel we consider points $x\in\RZ(K)$ for varying field extensions $K$ of $\BaseFldOfLocSht$. For two points $x:= (\ul{\CL}, \delta)$ and $x':= (\ul{\CL}', \delta')$ in $\RZ(K)$ we define
\begin{equation}\label{d_tilde}
\tilde{d}(x,x'):= \min\big\{ n\in\BN_0\colon \scrH^1(\iota)(\delta^{-1}\delta')\text{ is bounded by } 2n\rho\dual\big\}.
\end{equation}
By the definition of ``being bounded by $2n\rho\dual$'' in \eqref{EqBdHV} we conclude that if $x''\in\RZ(K)$ is a third point then $\tilde{d}(x,x'')\le\tilde{d}(x,x')+\tilde{d}(x',x'')$. Moreover, in the situation where $\delta=g$ and $\delta'=g'$ for $g,g'\in L\genericG(K)$, as well as $x=((L^+\BP)_{K},g^{-1}b\hat\sigma^*(g),g)$ and $x'=((L^+\BP)_{K},(g')^{-1}b\hat\sigma^*(g'),g')$, we also write $\tilde d(g,g'):=\tilde d(x,x')$. Note that a point $x\in\RZ(K)$ belongs to $\CM_n$ if and only if $\tilde d((\ul\BL_0,\id),x)\le n$.

Although we will not use this, note that $\tilde d$ is a metric on $\RZ$. This follows from the fact that \eqref{EqBdHV} for $n=0$ together with Cramer's rule implies that $\scrH^1(\iota)(\delta_1^{-1}\delta_2)$ is an isomorphism $\CV(\iota_*\ul\CL_2)\isoto\CV(\iota_*\ul\CL_1)$; compare the discussion around \EqBounded.

\begin{lemma}\label{unif_distribution_lemma}
For every integer $c>0$ there is an integer $d_0>0$ with the following property. For every $g\in L\genericG(K)$ with $\tilde d(g,\,b\,\hat{\sigma}^\ast(g)) <c$ there is a $g'\in L\genericG(\BaseFldOfLocSht)$ with $\tilde d(g,g')\le d_0$.
\end{lemma}

\begin{proof}
This is just a reformulation of \cite[Theorem~1.4 and Subsection~2.1]{RZ2} taking into account that by functorial properties of Bruhat--Tits buildings (see \cite{Lan1}) the representation $\iota$ induces an injective isometric map of Bruhat--Tits buildings $\scrB(\genericG)\to \scrB(H)$.
\end{proof}

\begin{lemma}\label{UnifDist2}
There is an integer $d_0\in \BN_0$ such that~ $\sup\big\{\tilde{d}(x,\RZ (\BaseFldOfLocSht))\colon x\in \RZ \big\}\leq d_0$.
\end{lemma}

\begin{proof}
Let $x:=(\ul \CL^+,\delta)$ be a point in $\RZ(K)$. After replacing $K$ by a separable field extension, we take a trivialization $(\ul \CL^+,\delta)\cong ((L^+\BP)_K,h\hat\sigma^*,g)$ with $\delta=g\in L\genericG(K)$ and $h=g^{-1}b\hat\sigma^*(g)$. Since the local $\BP$-shtuka $((L^+\BP)_K,h\hat\sigma^*)$ is bounded by $\hat{Z}$ and $Z$ is quasi-compact by definition of the boundedness condition (see Definition~\ref{DefBDLocal}), we have $\tilde d(g,\,b\,\hat{\sigma}^\ast(g))=\min\{n\in\BN_0\colon\scrH^1(\iota)(h)\text{ is bounded by }2n\rho\dual\} <c$ for a natural number $c$ which is independent of $x$. Let $d_0=d_0(c)$ be the integer from Lemma~\ref{unif_distribution_lemma}. Then there is a $g'\in L\genericG(\BaseFldOfLocSht)$ with $\tilde d(g,g')\le d_0$. Now the associated point $x':=((L^+\BP)_\BaseFldOfLocSht,{g'}^{-1}b\hat\sigma^* g',g')$ of $\RZ (\BaseFldOfLocSht)$ satisfies $\tilde d(x,x')\le d_0$ and the lemma follows.  
\end{proof}

For a point $y\in\RZ (\BaseFldOfLocSht)$ set $B(y,d_0):=\big\{x\in\RZ\colon\tilde{d}(x,y)\leq d_0\big\}$. Here $x\in\RZ(K)$ for a field extension $K$ of $\BaseFldOfLocSht$. We set $B_n(y,d_0)=B(y,d_0)\cap \CM_n$. Note that these are closed subsets.\\ 
For each integer $r$ let 
$$
\CZ_n^r=\bigcup_{y\in \RZ(\BaseFldOfLocSht),\,\tilde{d}\left((\ul\BL_0,\id),y\right)\geq r}B_n(y,d_0)
$$
Then $\CZ_n^r=\CZ_{n+1}^r\cap\CM_n$. If $y\notin \CM_{n+d_0}$, that is $\tilde{d}((\ul\BL_0,\id),y)> n+d_0$, and if $x\in\CM_n$, that is $\tilde d((\ul\BL_0,\id),x)\le n$, then
$$
\tilde{d}(x,y)\geq \tilde{d}((\ul\BL_0,\id),x)-n+\tilde{d}(x,y)\geq \tilde{d}((\ul\BL_0,\id),y)-n > d_0
$$
and thus $B_n(y,d_0)=\emptyset$. We get 
$$
\CZ_n^r=\bigcup_{y\in \CM_{n+d_0}(\BaseFldOfLocSht), \tilde{d}((\ul\BL_0,\id),y)\geq r}B_n(y,d_0).
$$
Since $\bigl(\CM_{n+d_0}\bigr)_\red=\bigl(\CM^{n+d_0}\bigr)_\red$ is quasi-projective over $\BaseFldOfLocSht$ by Lemma~\ref{M^n}, this union is finite.

Let $\CU_n^r$ be the open formal sub-scheme of $\CM_n$ whose underlying reduced set is $\CM_n\setminus \CZ_n^r$. We claim that the chain $\CU_n^r \hookrightarrow \CU_{n+1}^r \hookrightarrow \ldots$ of open formal sub-schemes of $\CM_n$ stabilizes. By the definition of $\CM_n$ it is enough to verify this on the underlying set of points. Suppose that there is some element $x\in \CU_{n+1}^r\setminus \CU^r_n$. By Lemma~\ref{UnifDist2} there exists a $y\in \ul\CM_{\ul\BL_0}^{\hat{Z}}(\BaseFldOfLocSht)$ such that $\tilde{d}(x,y)\leq d_0$. Since $x\in\CM_{n+1}\setminus\CZ_{n+1}^r$ we must have $\tilde{d}((\ul\BL_0,\id),y)< r$. Then
\begin{equation}\label{EqWith_d_tilde}
\tilde{d}((\ul\BL_0,\id), x)\leq \tilde{d}((\ul\BL_0,\id), y)+\tilde{d}(x,y)< r+d_0.
\end{equation}
Therefore, if $n\geq r+d_0$ then $\tilde{d}((\ul\BL_0,\id), x)\leq n$ and $x\in\CM_n$ which is a contradiction. Consequently there is no such $x$.

Let $\CU^r= \bigcup_n\CU_n^r$ (which equals $\CU_n^r$ for $n\geq r+d_0$). Note that every point $x$ of $\RZ$ lies in the union of the $\CU^r$s. Indeed, if $\tilde d\bigl((\ul\BL_0,\id),x\bigr)<r-d_0$ for some $r$, then $x$ is contained in $\CU^r$, because otherwise there is a $y\in \RZ(\BaseFldOfLocSht)$ with $\tilde d(x,y)\leq d_0$ and $\tilde d\bigl((\ul\BL_0,\id),y\bigr)\geq r$, a contradiction. Now consider the chain 
$$
\CU^r \hookrightarrow \CU^{r+1}\ldots \hookrightarrow \RZ
$$ 
of open immersions of formal schemes formally of finite type,  note that $\CU^r$ is open in $\RZ$. Indeed the underlying topological space of $\CU^r$ is open in $\CM_n$ for every $n$ and the ind-scheme $\CM_{\BL_0}^{\hat{Z}}$ carries the limit topology of the limit over the $\CM_n$. This shows that the formal scheme $\CU^r$ equals the formal completion of the open ind-scheme $\RZ|_{|\CU^r|}$ of $\RZ$ supported on $|\CU^r|$ along the whole set $|\CU^r|$ and thus $\RZ|_{|\CU^r|}=\CU^r$. Since $\CU^r$ is locally formally of finite type this implies that $\RZ= \bigcup_r\CU^r$ is locally formally of finite type as well. This completes the proof of Theorem~\ref{ThmRRZSp}
\end{proof}

\begin{corollary}\label{CorQC}
The irreducible components of the topological space $\RZ$ are quasi-projective schemes over $\BaseFldOfLocSht$. In particular they are quasi-compact. They are projective if $\BP$ is parahoric in the sense of Bruhat and Tits \cite[D\'efinition~5.2.6]{B-T} and \cite{H-R}.
\end{corollary}

\begin{proof}
Let $T$ be an irreducible component and let $x$ be its generic point. As in the proof of the theorem there is an $r$ such that $x\in \CU^r=\CU^r_{r+d_0}\subset\CM_{r+d_0}$. Since the underlying topological spaces of $\CM_{r+d_0}$ and $\CM^{r+d_0}$ coincide, are closed in $\RZ$, and (quasi-)projective over $\BaseFldOfLocSht$ by Lemma~\ref{M^n}, we see that $T\subset\CM^{r+d_0}$ is a closed subscheme and the corollary follows.
\end{proof}

In the rest of this section we fix an integer $n$ and consider complete discrete valuation rings $\BF_i\dbl z_i\dbr$ for $i=1,\ldots,n$ with finite residue fields $\BF_i$, and fraction fields $Q_i=\BF_i\dpl z_i\dpr$. Let $\BP_{i}$ be a smooth affine group scheme over $\Spec\BF_i\dbl z_i\dbr$ with connected reductive generic fiber $\genericG_i:=\BP_i\times_{\BF_i\dbl z_i\dbr}\Spec\BF_i\dpl z_i\dpr$, and let $\hat{Z}_{i}=[\hat{Z}_{i,R'_i}]$ with $\hat{Z}_{i,R'_i}\subset\wh{\SpaceFl}_{\BP_{i},R'_i}:=\SpaceFl_{\BP_i}\whtimes_{\BF_i}\Spf R'_i$ be a bound in the sense of Definition~\ref{DefBDLocal} with reflex ring $R_{\hat{Z}_i}=:R_i=\kappa_i\dbl\xi_i\dbr$. Let $\BaseFldInSectUnif$ be a field containing all $\kappa_i$. For all $i$ let $\ul{\BL}_i$ be a local $\BP_i$-shtuka over $\BaseFldInSectUnif$ which is trivialized and decent. 
By Theorem~\ref{ThmRRZSp} the Rapoport--Zink space $\CM_{\ul{\BL}_i}^{\hat{Z}_i}$ is a formal scheme locally formally of finite type over $\Spf\BaseFldInSectUnif\dbl\xi_i\dbr$. Therefore the product $\prod_i \CM_{\ul{\BL}_i}^{\hat{Z}_i}:=\CM_{\ul{\BL}_1}^{\hat{Z}_1}\whtimes_\BaseFldInSectUnif\ldots\whtimes_\BaseFldInSectUnif\CM_{\ul{\BL}_n}^{\hat{Z}_n}$ is a formal scheme locally formally of finite type over $\Spf\BaseFldInSectUnif\dbl\ul\xi\dbr:=\Spf\BaseFldInSectUnif\dbl\xi_1,\ldots,\xi_n\dbr$. Recall that the group $J_{\ul{\BL}_i}(Q_i)=\QIsog_\BaseFldInSectUnif(\ul{\BL}_i)$ of quasi-isogenies of $\ul\BL_i$ over $\BaseFldInSectUnif$ acts naturally on $\CM_{\ul{\BL}_i}^{\hat{Z}_i}$; see Remark~\ref{RemJActsOnRZ}. 

Let $\Gamma \subseteq \prod_i J_{\ul{\BL}_i}(Q_i)$ be a subgroup which is discrete for the product of the $z_i$-adic topologies. We say that $\Gamma$ is \emph{separated} if it is separated in the profinite topology, that is, if for every $1\ne g \in \Gamma$ there is a normal subgroup of finite index that does not contain $g$.

\begin{proposition}\label{quotisfstack}
 Let $\Gamma \subseteq \prod_i J_{\ul{\BL}_i}(Q_i)$ be a separated discrete subgroup. Then the quotient
$\Gamma\backslash \prod_i \CM_{\ul{\BL}_i}^{\hat Z_i}$ is a locally noetherian, adic formal algebraic $\Spf \BaseFldInSectUnif\dbl \ul\xi \dbr$-stack locally formally of finite type. Moreover, the 1-morphism $\prod_i\CM_{\ul{\BL}_i}^{\hat Z_i}\to \Gamma\backslash \prod_i \CM_{\ul{\BL}_i}^{\hat Z_i}$ is adic and \'etale.
\end{proposition}

Here we say that a formal algebraic $\Spf\BaseFldInSectUnif\dbl \ul\xi \dbr$-stack $\CX$ (see \cite[Definition~A.5]{Har1}) is \emph{$\CJ$-adic} for a sheaf of ideals $\CJ\subset\CO_\CX$, if for some (any) presentation $X\to\CX$ the formal scheme $X$ is $\CJ\CO_X$-adic, that is, $\CJ^r\CO_X$ is an ideal of definition of $X$ for all $r$. We then call $\CJ$ an \emph{ideal of definition of $\CX$}. We say that $\CX$ is \emph{locally formally of finite type} if $\CX$ is locally noetherian, adic, and if the closed substack defined by the largest ideal of definition (see \cite[A.7]{Har1}) is an algebraic stack locally of finite type over $\Spec\BaseFldInSectUnif$.
Before proving the above proposition let us state the following lemma. Recall that (also an infinite dimensional) scheme is quasi-compact if and only if it is a finite union of affine schemes. Further recall that every morphism from a quasi-compact scheme to an ind-quasi-projective ind-scheme factors through a quasi-projective subscheme by \cite[Lemma~5.4]{H-V}.

\begin{lemma}\label{Gamma'}
Let $\Gamma\subseteq \prod_i J_{\ul{\BL}_i}(Q_i)$ be a separated discrete subgroup. Let $U_i\subset \SpaceFl_{\BP_i}\whtimes_{\Spec\BF_i}\Spec \BaseFldInSectUnif$ be a quasi-compact subscheme and set $U=U_1\times_\BaseFldInSectUnif\ldots\times_\BaseFldInSectUnif U_n$. Then the set $\{\gamma\in\Gamma\colon\gamma U\cap U \neq \emptyset\}$ is finite.
\end{lemma}

\begin{proof}
Note that $J_{\ul{\BL}_i}(Q_i)$ is contained in $L\genericG_i(\BaseFldInSectUnif^\alg)$. By Theorem~\ref{ThmModuliSpX} any point $\ul x \in U(\BaseFldInSectUnif^\alg)$ can be represented by a tuple $(\ul\CL_i, g_i)_i$, where $\ul\CL_i$ is a trivialized local $\BP_i$-shtuka over $\BaseFldInSectUnif^\alg$ and the quasi-isogeny $g_i\colon\ul\CL_i\to\ul{\BL}_i$ is given by an element $g_i\in L\genericG_i(\BaseFldInSectUnif^\alg)$. By \cite[Theorem~1.4]{PR2} the projection $L\genericG_i\to \SpaceFl_{\BP_i}$ admits sections locally for the \'etale topology, and hence \'etale locally on $\SpaceFl_{\BP_i}$ the loop group $L\genericG_i$ is isomorphic to the product $\SpaceFl_{\BP_i}\times_{\BF_i}\,L^+\BP_i$. In particular, by \cite[IV$_2$, Proposition~2.7.1]{EGA} the projection $\prod_i L\genericG_i\whtimes_{\prod_i \SpaceFl_{\BP_i}}\,U\,\to\, U$ is an affine morphism of schemes and therefore $\wt{U}:=\prod_i L\genericG_i\whtimes_{\prod_i \SpaceFl_{\BP_i}}\,U\subseteq \prod_i L\genericG_i$ is a quasi-compact scheme. Consider the morphism
\begin{eqnarray}\label{EqMorphTildeU}
&\wt{U}\times_{\BaseOfD} \wt{U} & \longto \es {\TS\prod_i}\SpaceFl_{\BP_i}\\
&(\,g_i\,, g'_i\,)_i & \longmapsto \es (\,g'_i\, g_i^{-1} L^+\BP_i\,/\,L^+\BP_i)_i.\nonumber
\end{eqnarray}
Since the $\SpaceFl_{\BP_i}$ are ind-quasi-projective ind-schemes also $\prod_i\SpaceFl_{\BP_i}$ is. Therefore \eqref{EqMorphTildeU} factors through some quasi-projective subscheme $V\subset \prod_i\SpaceFl_{\BP_i}$ by \cite[Lemma~5.4]{H-V}. Since $\ul\BL_i$ is decent the group of quasi-isogenies $J_{\ul{\BL}_i}(Q_i)\subset L\genericG_i(\BaseFldInSectUnif^\alg)$ is contained in $L\genericG_i(\BaseFldOfLocSht)$ for some finite field $\BaseFldOfLocSht$; see Remark~\ref{RemJb}. Let $\gamma\in \Gamma$ be such that $\ul x=(\ul\CL_i,g_i)_i \in U$ and $\gamma \ul x =(\ul\CL_i,\gamma_i g_i)\in U$ where $\gamma_i\in L\genericG_i(\BaseFldOfLocSht)$ is the projection of $\gamma$ onto the $i$-th factor. Then $(\,g_i\,,\gamma_i g_i)_i\in\wt U\times\wt U$ and the image of $\gamma$ under the projection map $\pi\colon\prod_i L\genericG_i\to\prod_i \SpaceFl_{\BP_i}$ lies in the finite set $V(\BaseFldOfLocSht)$. Thus $\gamma$ lies in the compact set $S=\pi^{-1}(V(\BaseFldOfLocSht))\cap\prod_i J_{\ul{\BL}_i}(Q_i)$. On the other hand $\Gamma$ is discrete and thus has finite intersection with $S$.  
\end{proof}

\begin{proof}[Proof of Proposition~\ref{quotisfstack}] 
By Lemma~\ref{UnifDist2} there is a finite field $\BaseFldOfLocSht_i$ and a constant $d_i$ such that any ball in $\CM_{\ul{\BL}_i}^{\hat Z_i}$ with radius $d_i$ contains a rational point in  $\CM_{\ul{\BL}_i}^{\hat Z_i} (\BaseFldOfLocSht_i)$. Let $d$ be the maximum of the integers $d_i$ and let $\BaseFldOfLocSht$ be the compositum of the fields $\BaseFldOfLocSht_i$. Let $x:=(x_i)_i$ be a point of $\prod_i \CM_{\ul{\BL}_i}^{\hat Z_i}(\BaseFldOfLocSht)$. We use the notation of the proof of Theorem~\ref{ThmRRZSp} and consider in particular the metric $\tilde d_i$ on $\CM_{\ul{\BL}_i}^{\hat Z_i}$ defined as in \eqref{d_tilde} after choosing some representation $\iota_i$ of $\BP_i$. For any positive integer $c$ we define the closed subscheme of $(\CM_{\ul{\BL}_i}^{\hat Z_i})_\red$ given by
\[
B(x_i,c):=\{z\in\CM_{\ul{\BL}_i}^{\hat Z_i}\colon \tilde d_i(z,x_i)\le c\}
.\]
Then every $z\in B(x_i,c)$ satisfies $\tilde{d}_i((\ul\BL_i,\id),z)\le\tilde{d}_i((\ul\BL_i,\id),x_i)+c=:m$ and so $B(x_i,c)\subset(\CM_m)_\red=(\CM^m)_\red$. Therefore Lemma~\ref{M^n} implies that $B(x_i,c)$ is a quasi-projective scheme over $\Spec\BaseFldOfLocSht$. In particular the set of $\BaseFldOfLocSht$-valued points $B(x_i,c)(\BaseFldOfLocSht)$ is finite. Thus for all $n\ge 2d$ the subscheme
\[
\CU^{2d}_n(x_i)\es:=\es B(x_i,n)\setminus\bigcup_{y\in B(x_i,n+d)(\BaseFldOfLocSht),\,\tilde{d}_i(y,x_i)> 2d}B(y,d)\es\subset\es(\CM_{\ul{\BL}_i}^{\hat Z_i})_\red
\]
is open in $B(x_i,n+d)$ and quasi-compact. Note that for $n\ge 3d$ all $\CU^{2d}_n(x_i)$ coincide with $\CU^{2d}_{3d}(x_i)$ by an argument similar to \eqref{EqWith_d_tilde}. Since $(\CM_{\ul{\BL}_i}^{\hat Z_i})_\red=\dirlim B(x_i,n+d)$ carries the limit topology, the subscheme $\CU^{2d}(x_i):=\CU^{2d}_{3d}(x_i)\subset(\CM_{\ul{\BL}_i}^{\hat Z_i})_\red$ is open and quasi-compact. By Lemma~\ref{UnifDist2} the union of the $\CU^{2d}(x_i)$ for all $x_i\in\CM_{\ul{\BL}_i}^{\hat Z_i}(\BaseFldOfLocSht)$ equals $(\CM_{\ul{\BL}_i}^{\hat Z_i})_\red$.

For $x=(x_i)_i$ set $U_x:=\prod_i \CU^{2d}(x_i)$. Then $\gamma. U_x=U_{\gamma.x}$ and the open subsets $U_x$ cover $\prod_i \CM_{\ul{\BL}_i}^{\hat Z_i}$, for varying $x\in \prod_i \CM_{\ul{\BL}_i}^{\hat Z_i}(\BaseFldOfLocSht)$. Let $I\subset \prod_i \CM_{\ul{\BL}_i}^{\hat Z_i}(\BaseFldOfLocSht)$ be a set of representatives of the $\Gamma$-orbits in $\prod_i \CM_{\ul{\BL}_i}^{\hat Z_i}(\BaseFldOfLocSht)$. Fix an $x\in I$. Since $\Gamma$ is separated, we may choose by Lemma~\ref{Gamma'} a normal subgroup $\Gamma_x' \subset \Gamma$ of finite index in $\Gamma$ such that $U_x\cap \gamma' U_x= \emptyset$ for all $\gamma'\neq1$ in $\Gamma_x'$. For all $\gamma\in\Gamma$ the natural morphism $\gamma U_x\to\Gamma_x'\backslash\prod_i \CM_{\ul{\BL}_i}^{\hat Z_i}$ is an open immersion. Let $V_x$ be the (finite) union of the images of these morphisms for all $\gamma\in\Gamma$. Then $(\Gamma_x'\backslash\Gamma)\backslash V_x$ is an open substack of $\Gamma\backslash \prod_i  \CM_{\ul{\BL}_i}^{\hat Z_i}$. Moreover, $(\Gamma_x'\backslash\Gamma)\backslash V_x$ is a finite \'etale quotient of $V_x$ and the map $\coprod_{\gamma\in\Gamma}\gamma U_x\to(\Gamma_x'\backslash\Gamma)\backslash V_x$ is adic and \'etale. Therefore the morphism $\prod_i  \CM_{\ul{\BL}_i}^{\hat Z_i}\to\Gamma\backslash \prod_i  \CM_{\ul{\BL}_i}^{\hat Z_i}$ is adic and \'etale above $(\Gamma_x'\backslash\Gamma)\backslash V_x$. Since $(V_x)_\red$ is a scheme of finite type over $\BaseFldInSectUnif$, its quotient $\bigl((\Gamma_x'\backslash\Gamma)\backslash V_x\bigr)_\red=(\Gamma_x'\backslash\Gamma)\backslash(V_x)_\red$ is an algebraic stack of finite type over $\BaseFldInSectUnif$. Since the open subsets $U_x$ cover $\prod_i \CM_{\ul{\BL}_i}^{\hat Z_i}$ also the $(\Gamma_x'\backslash\Gamma)\backslash V_x$ cover $\Gamma\backslash \prod_i  \CM_{\ul{\BL}_i}^{\hat Z_i}$ and the proposition follows.
\end{proof}

\begin{remark}
For us it does not make sense to strengthen Proposition~\ref{quotisfstack} like Rapoport and Zink~\cite[Proposition~2.37]{RZ} do, who obtain that for Rapoport--Zink spaces for $p$-divisible groups the quotient is a formal algebraic space if $\Gamma$ is torsion free. Namely in our case all unipotent subgroups of $\prod_i J_{\ul{\BL}_i}(Q_i)$ are torsion. So they cannot act fixed point free on $\prod_i \CM_{\ul{\BL}_i}^{\hat Z_i}$ and the corresponding quotients cannot be formal algebraic spaces.
\end{remark}

%%%%%%%%%%%%%%%%%%%%%%%%%%%%%%%%%%%%%%%%%%%%%%%%%%%%%%%%%%%%%%%%%%%%%%
%
%  Construction of the Uniformization map
%
%%%%%%%%%%%%%%%%%%%%%%%%%%%%%%%%%%%%%%%%%%%%%%%%%%%%%%%%%%%%%%%%%%%%%%

\section{The Relation Between Global \texorpdfstring{$\FG$}{G}-Shtukas and Local \texorpdfstring{$\BP$}{P}-Shtukas}\label{UnboundedUnif}
\setcounter{equation}{0}

\subsection{Preliminaries on \texorpdfstring{$\FG$}{G}-Torsors}\label{SectPrelimGTors}

In Chapter~\ref{UnboundedUnif} we assume that $\FG$ is a flat affine group scheme of finite type over $C$. Let $\nu\in C$ be a closed point of $C$ and set $C':=C\setminus\{\nu\}$. We let $\scrH_e^1(C',\FG)$ denote the category fibered in groupoids over the category of $\BF_q$-schemes, such that $\scrH_e^1(C',\FG)(S)$ is the full subcategory of $[C'_S/\FG](C'_S)$ consisting of those $\FG$-torsors over $C'_S$ that can be extended to a $\FG$-torsor over the whole relative curve $C_S$. We denote by $\dot{(~)}$ the restriction morphism 
$$
\dot{(~)}\colon\scrH^1(C,\FG)\longto \scrH_e^1(C', \FG)
$$ 
which assigns to a $\FG$-torsor $\CG$ over $C_S$ the $\FG$-torsor $\dot\CG:=\CG\times_{C_S}C'_S$ over $C'_S$. 
Let $\wt\BP_\nu:=\Res_{\BF_\nu/\BF_q}\BP_\nu$ and $\wt\genericG_\nu:=\Res_{\BF_\nu/\BF_q}\genericG_\nu$ be the Weil restrictions. Then $\wt\BP_\nu$ is a flat affine group scheme of finite type over $\Spec\BF_q\dbl z\dbr$. We apply Definition~\ref{formal torsor def} for $\BaseOfD=\BF_q$ and let $\hat{\wt\BP}_\nu:=\wt\BP_\nu\whtimes_{\Spec\BF_q\dbl z\dbr}\Spf\BF_q\dbl z\dbr=\Res_{\BF_\nu/\BF_q}\hat\BP_\nu$ be the $\nu$-adic completion. We write $A_\nu\cong\BF_\nu\dbl z\dbr$ for a uniformizer $z\in\BF_q(C)$. Then for every $\BF_q$-algebra $R$ we have
\begin{eqnarray*}
A_\nu\wh\otimes_{\BF_q}R\es\cong & (R\otimes_{\BF_q}\BF_\nu)\dbl z\dbr & =\es R\dbl z\dbr\otimes_{\BF_q}\BF_\nu \quad\text{and}\\[1mm]
Q_\nu\wh\otimes_{\BF_q}R\es\cong & (R\otimes_{\BF_q}\BF_\nu)\dpl z\dpr & =\es R\dpl z\dpr\otimes_{\BF_q}\BF_\nu\,.
\end{eqnarray*}
This implies that 
\[
L^+\wt\BP_\nu(R)\,=\,\wt\BP_\nu(R\dbl z\dbr)\,=\,\BP_\nu(A_\nu\wh\otimes_{\BF_q}R) \quad\text{and}\quad
L\wt\genericG_\nu(R)\,=\,\wt\genericG_\nu(R\dpl z\dpr)\,=\,\genericG_\nu(Q_\nu\wh\otimes_{\BF_q}R).
\]
If $\CG\in\scrH^1(C,\FG)(S)$, its completion $\wh\CG_\nu:=\CG\whtimes_{C_S}(\Spf A_\nu\whtimes_{\BF_q}S)$ is a formal $\hat\BP_\nu$-torsor (Definition~\ref{formal torsor def}) over $\Spf A_\nu\whtimes_{\BF_q}S$. The Weil restriction $\Res_{\BF_\nu/\BF_q}\wh\CG_\nu$ is a formal $\hat{\wt\BP}_\nu$-torsor over $\Spf\BF_q\dbl z\dbr\whtimes_{\BF_q}S$ and corresponds by Proposition~\ref{formal torsor prop} to an $L^+\wt\BP_\nu$-torsor over $S$ which we denote $L^+_\nu(\CG)$. We obtain the morphism
$$
L^+_\nu\colon  \scrH^1(C,\FG)\longto \scrH^1(\Spec\BF_q,L^+\wt\BP_\nu)\,,\quad\CG\mapsto L^+_\nu(\CG)\,.
$$
Finally there is a morphism
$$
L_\nu\colon \scrH_e^1(C', \FG)(S)\longto\scrH^1(\Spec\BF_q,L\wt\genericG_\nu)(S)\,,\quad\dot\CG\mapsto L_\nu(\dot\CG)
$$
which sends the $\FG$-torsor $\dot\CG$ over $C'_S$, having some extension $\CG$ over $C_S$, to the $L\wt\genericG_\nu$-torsor $L(L^+_\nu(\CG))$ associated with $L^+_\nu\CG$ under \eqref{EqLoopTorsor}. We claim that $L_\nu(\CG)$ is independent of the extension $\CG$, and that we therefore may write $L_\nu(\dot\CG):=L(L^+_\nu(\CG))$. Indeed, let $\CG'$ be a second extension of $\dot\CG$ and let \mbox{$f\colon\dot\CG\isoto\dot\CG'$} be an isomorphism of $\FG$-torsors over $C'_S$. Without loss of generality $S=\Spec R$ is affine. We may choose an \fppf-covering $\Spec\wt R\to C_S$ over which trivializations $\alpha\colon\CG\times_{C_S}\Spec\wt R\isoto\FG\times_C\Spec\wt R$ and \mbox{$\alpha'\colon\CG'\times_{C_S}\Spec\wt R\isoto\FG\times_C\Spec\wt R$} exist. Then $\alpha'\circ f\circ\alpha^{-1}$ equals multiplication with an element $g\in\FG(\Spec\wt R\times_CC')$. The ring homomorphism $\BF_\nu\to A_\nu/\nu^n=\BF_\nu\dbl z\dbr/(z^n)$ and the short exact sequence $0 \to \BF_\nu\xrightarrow{z^n\cdot\,} A_\nu/\nu^{n+1}\to A_\nu/\nu^n\to0$ of $A$-modules yield by tensoring with the flat $A$-algebra $\wt R$ the vertical maps and the bottom row in the following commutative diagram of $A_\nu$-modules
\[
\xymatrix @C+2pc {
0 \ar[r] & \wt R/\nu \ar@{=}[d] \ar[r]^{z^n\cdot\qquad\quad} & (\wt R/\nu)\dbl z\dbr/(z^{n+1}) \ar[r] \ar[d] & (\wt R/\nu)\dbl z\dbr/(z^n) \ar[r]\ar[d] & 0\;\, \\
0 \ar[r] & \wt R/\nu \ar[r]^{z^n\cdot} & \wt R/\nu^{n+1} \ar[r] & \wt R/\nu^n \ar[r] & 0\;,
}
\]
which proves by induction that $(\wt R/\nu)\dbl z\dbr/(z^n)\isoto\wt R/\nu^n$ and $(\wt R/\nu)\dbl z\dbr \isoto\invlim\wt R/\nu^n$. The $\nu$-adic completion $\hat\alpha$ of $\alpha$ is an isomorphism between $\wh\CG_\nu\whtimes_{\Spf(R\otimes_{\BF_q}\BF_\nu)\dbl z\dbr}\Spf(\wt R/\nu)\dbl z\dbr\,=\,\CG\whtimes_{C_S}\Spf(\wt R/\nu)\dbl z\dbr$ and the trivial formal $\hat\BP_\nu$-torsor $(\FG\times_C\Spec\wt R)\whtimes_{\Spec\wt R}\Spf(\wt R/\nu)\dbl z\dbr\,=\,\hat\BP_\nu\whtimes_{\Spf A_\nu}\Spf(\wt R/\nu)\dbl z\dbr$ over $\Spf(\wt R/\nu)\dbl z\dbr$. For the base change $(\Res_{\BF_\nu/\BF_q}\wh\CG_\nu)\whtimes_{R\dbl z\dbr}\Spf(\wt R/\nu)\dbl z\dbr=\prod_{\Gal(\BF_\nu/\BF_q)}\wh\CG_\nu\whtimes_{\Spf(R\otimes_{\BF_q}\BF_\nu)\dbl z\dbr}\Spf(\wt R/\nu)\dbl z\dbr$ we obtain an isomorphism 
\[
\prod_{\Gal(\BF_\nu/\BF_q)}\hspace{-1em}\hat\alpha\;\colon(\Res_{\BF_\nu/\BF_q}\wh\CG_\nu)\whtimes_{R\dbl z\dbr}\Spf(\wt R/\nu)\dbl z\dbr\,\isoto\,\prod_{\Gal(\BF_\nu/\BF_q)}\FG\times_C\Spf(\wt R/\nu)\dbl z\dbr\,=\,\hat{\wt\BP}_\nu\whtimes_{\BF_q\dbl z\dbr}\Spf(\wt R/\nu)\dbl z\dbr
\]
which under Proposition~\ref{formal torsor prop} corresponds to an isomorphism
\[
\prod_{\Gal(\BF_\nu/\BF_q)}\hspace{-1em}\hat\alpha\;\colon L^+_\nu(\CG)\times_{R}\Spec(\wt R/\nu)\,\isoto\,(L^+\wt\BP_\nu)_{\Spec(\wt R/\nu)}.
\]
We also have the analogous isomorphism for $\wh\CG'_\nu:=\CG'\whtimes_{C_S}(\Spf A_\nu\whtimes_{\BF_q}S)$ and $L^+_\nu(\CG')$. Under the $\nu$-adic completion morphism, $g$ is mapped to an element $\hat g\in\FG\bigl((\wt R/\nu)\dpl z\dpr\bigr)=\genericG_\nu\bigl((\wt R/\nu)\dpl z\dpr\bigr)$. It yields the element $(\hat g,\ldots,\hat g)\in\prod_{\Gal(\BF_\nu/\BF_q)}\genericG_\nu\bigl((\wt R/\nu)\dpl z\dpr\bigr)\,=\,(\Res_{\BF_\nu/\BF_q}\genericG_\nu)\bigl((\wt R/\nu)\dpl z\dpr\bigr)\,=\,L\wt\genericG_\nu(\wt R/\nu)$. The composition $(\prod_{\Gal(\BF_\nu/\BF_q)}\hat{\alpha}')^{-1}\circ(\hat g,\ldots,\hat g)\circ\prod_{\Gal(\BF_\nu/\BF_q)}\hat{\alpha}$ defines an isomorphism $L(L^+_\nu(\CG))\times_{R}\Spec(\wt R/\nu)\,\isoto\,L(L^+_\nu(\CG'))\times_{R}\Spec(\wt R/\nu)$ which inherits the descent datum from $f$ and defines the desired isomorphism $L(L^+_\nu(\CG))\,\isoto\,L(L^+_\nu(\CG'))$. This proves our claim that $L_\nu(\CG)$ is independent of the extension $\CG$. We write $L_\nu(\dot\CG):=L(L^+_\nu(\CG))$.

\begin{lemma}\label{LemmaBL}
The above maps assign to each $\FG$-torsor $\CG$ over $C_S$ a triple $(\dot\CG,L^+_\nu(\CG),\phi)$ where $\phi\colon L_\nu(\dot\CG)\isoto L(L^+_\nu(\CG))$ is the canonical isomorphism of $L\wt\genericG_\nu$-torsors. $\scrH^1(C, \FG)(S)$ is equivalent to the category of such triples. In other words, the following diagram of groupoids is cartesian
$$
\CD
\scrH^1(C, \FG)@>{\TS\dot{(~)}}>> \scrH_e^1(C', \FG)\\
@V{\TS L^+_\nu}VV @VV{\TS L_\nu}V\\
\scrH^1(\Spec\BF_q, L^+\wt\BP_\nu)@>{\TS L}>> \scrH^1(\Spec\BF_q, L\wt\genericG_\nu)\,.\\
\endCD
$$
\end{lemma}

\begin{proof}
This follows from the glueing result of Beauville and Laszlo~\cite{B-L}. Let us give more details. We construct the inverse of the morphism 
\[
\scrH^1(C,\FG)\longto\scrH_e^1(C',\FG)\times_{\scrH^1(\Spec\BF_q, L\wt\genericG_\nu)}\scrH^1(\Spec\BF_q, L^+\wt\BP_\nu). 
\]
Over an $\BF_q$-scheme $S$ we consider $S$-valued points $\dot\CG\in\scrH_e^1(C',\FG)(S)$ and $\CL^+_\nu\in\scrH^1(\Spec\BF_q, L^+\wt\BP_\nu)(S)$, as well as an isomorphism $\phi\colon L_\nu(\dot\CG)\isoto L(\CL^+_\nu)$ in $\scrH^1(\Spec\BF_q, L\wt\genericG_\nu)(S)$. Let $\CG$ be an extension of $\dot\CG$ to $C_S$. There exists an \fppf-covering $S'\to S$ and trivializations $\hat\alpha\colon(\CL^+_\nu)_{S'}\isoto(L^+\wt\BP_\nu)_{S'}$ and $\hat\beta\colon L^+_\nu(\CG)_{S'}\isoto(L^+\wt\BP_\nu)_{S'}$ of the pullbacks to $S'$. We may assume that $S'$ is the disjoint union of affine schemes of the form $\Spec R'$. 

We glue $\dot\CG_{S'}$ via the isomorphism $\hat\alpha\circ\phi\colon L_\nu(\dot\CG)_{S'}\isoto(L\wt\genericG_\nu)_{S'}$ to $(L^+\wt\BP_\nu)_{S'}$ as follows. Consider the algebraic torsor $\CG\times_{C_S}\Spec(A_\nu\wh\otimes_{\BF_q}R')$ for the group scheme 
\[
\FG\times_C\Spec(A_\nu\wh\otimes_{\BF_q}R')\,=\,\BP_\nu\times_{\Spec A_\nu}\Spec(A_\nu\wh\otimes_{\BF_q}R')
\]
over $\Spec(A_\nu\wh\otimes_{\BF_q}R')$. By \cite[Proposition~2.7.1]{EGA} it is affine of the form $\Spec B$. Its $\nu$-adic completion $(\wh\CG_\nu)_{R'}=\CG\whtimes_{C_S}\Spf(A_\nu\wh\otimes_{\BF_q}R')=\Spf\wh B$ is affine with $\wh B=\invlim B/\nu^mB$. Recall that $L^+_\nu(\CG)_{R'}$ is the $L^+\wt\BP_\nu$-torsor associated with the Weil restriction $\Res_{\BF_\nu/\BF_q}(\wh\CG_\nu)_{R'}$ by Proposition~\ref{formal torsor prop}. The trivialization $\hat\beta$ induces a trivialization $\hat\beta\colon\Res_{\BF_\nu/\BF_q}(\wh\CG_\nu)_{R'}\isoto\hat{\wt\BP}_\nu\whtimes_{\Spf\BF_q\dbl z\dbr}\Spf R'\dbl z\dbr$ which is determined by the section
\begin{eqnarray*}
\hat\beta^{-1}(1)&\in&\Hom_{\Spf R'\dbl z\dbr}(\Spf R'\dbl z\dbr\,,\,\Res_{\BF_\nu/\BF_q}(\wh\CG_\nu)_{R'})\\
&=&\Hom_{\Spf(A_\nu\wh\otimes_{\BF_q}R')}(\Spf(A_\nu\wh\otimes_{\BF_q}R')\,,\,\Spf\wh B)\\
&=& \Hom^\cont_{A_\nu\wh\otimes_{\BF_q}R'}(\wh B\,,\,A_\nu\wh\otimes_{\BF_q}R')\\
&=& \Hom_{A_\nu\wh\otimes_{\BF_q}R'}(B\,,\,A_\nu\wh\otimes_{\BF_q}R')\\
&=& \Hom_{\Spec(A_\nu\wh\otimes_{\BF_q}R')}(\Spec(A_\nu\wh\otimes_{\BF_q}R')\,,\,\Spec B).
\end{eqnarray*}
The latter induces a trivialization $\beta\colon\CG\times_{C_S}\Spec(A_\nu\wh\otimes_{\BF_q}R')\isoto\BP_\nu\times_{\Spec A_\nu}\Spec(A_\nu\wh\otimes_{\BF_q}R')$. Similarly the automorphism $\hat\psi:=\hat\alpha\phi\hat\beta^{-1}$ of $(L\wt\genericG_\nu)_{R'}$ is determined by the image 
\[
\hat\psi(1)\es\in\es L\wt\genericG_\nu(R')\es=\es\genericG_\nu(\Spec Q_\nu\wh\otimes_{\BF_q}R')
\]
and thus induces an automorphism $\psi$ of $\genericG_\nu\times_{\Spec Q_\nu}\Spec(Q_\nu\wh\otimes_{\BF_q}R')$. By \cite{B-L} we may glue $\dot\CG_{R'}$ with $\BP_\nu\times_{\Spec A_\nu}\Spec(A_\nu\wh\otimes_{\BF_q}R')$ via the isomorphism $\psi\beta$ to obtain a uniquely determined $\FG$-torsor $\CG'$ on $C_{R'}$.

We descend $\CG'$ to $C_S$ as follows. Let $S''=S'\times_SS'$ and let $p_i\colon S''\to S'$ be the projection onto the $i$-th factor. Consider the element $h:=(p_1^*\hat\alpha\circ p_2^*\hat\alpha^{-1})(1)\,\in\,L^+\wt\BP_\nu(S'')\,=\,\BP_\nu(A_\nu\wh\otimes_{\BF_q}\Gamma(S'',\CO_{S''}))$. The isomorphism
\[
(\id_{\dot\CG_{S''}},h\cdot)\colon p_2^*(\dot\CG_{S'},\,\BP_\nu\times_{\Spec A_\nu}\Spec(A_\nu\wh\otimes_{\BF_q}R'),\,\psi\beta)\isoto p_1^*(\dot\CG_{S'},\,\BP_\nu\times_{\Spec A_\nu}\Spec(A_\nu\wh\otimes_{\BF_q}R'),\,\psi\beta)
\]
induces a descent datum on $\CG'$ which is effective by \cite[\S\,6.1, Theorem~7]{BLR} because $\CG'$ is affine over $C_{S'}$. Thus $\CG'$ descends to a $\FG$-torsor $\CG\in\scrH^1(C,\FG)(S)$. This defines the inverse morphism and finishes the proof.
\end{proof}

\subsection{The Global-Local Functor}\label{SectGlobalLocalFunctor}

Analogously to the functor which assigns to an abelian variety $A$ over a $\BZ_p$-scheme its $p$-divisible group $A[p^\infty]$ we introduce a global-local functor from global $\FG$-shtukas to local $\BP_{\nu_i}$-shtukas. But whereas abelian varieties only have one characteristic place, our global $\FG$-shtukas have $n$ characteristic places $\ul\nu=(\nu_1,\ldots,\nu_n)$. So the global-local functor will assign to each global $\FG$-shtuka of characteristic $\ul\nu$ an $n$-tuple of local $\BP_{\nu_i}$-shtukas. We begin with a

\begin{remark}\label{G-LFrob}
Let $\nu$ be a place on $C$ and let $\BD_\nu:=\Spec A_\nu$ and $\hat\BD_\nu:=\Spf A_\nu$. Let $\deg\nu:=[\BF_\nu:\BF_q]$ and fix an inclusion $\BF_\nu\subset A_\nu$. Assume that we have a section $\charsect\colon S\rightarrow C$ which factors through $\Spf A_\nu$, that is, the image in $\CO_S$ of a uniformizer of $A_\nu$ is locally nilpotent. In this case we have
\begin{equation}\label{EqDecomp}
\hat\BD_\nu\whtimes_{\BF_q}\, S\cong \coprod_{\ell\in\BZ/(\deg\nu)} \Var (\mathfrak{a}_{\nu,\ell})\cong \coprod_{\ell\in \BZ/(\deg\nu)}  \hat{\BD}_{\nu,S},
\end{equation}
where we write $\hat\BD_{\nu,S}:=\hat\BD_\nu\whtimes_{\BF_\nu}S$ and where we denote by $\Var (\mathfrak{a}_{\nu,\ell})$ the component identified by the ideal \mbox{$\mathfrak{a}_{\nu,\ell}=\langle a \otimes 1 -1\otimes \charsect^\ast (a)^{q^\ell}\colon a \in \BF_\nu \rangle$}. Note that $\sigma$ cyclically permutes these components and thus the $\BF_\nu$-Frobenius $\sigma^{\deg{\nu}}=:\hat\sigma$ leaves each of the components $\Var (\mathfrak{a}_{\nu,\ell})$ stable. Also note that there are canonical isomorphisms $\Var (\mathfrak{a}_{\nu,\ell})\cong\hat\BD_{\nu,S}$ for all $\ell$.
\end{remark}

Although we will not need it in the sequel, we note the following interpretation of the component $\Var (\mathfrak{a}_{\nu,0})$.

\begin{lemma}\label{LemmaCompletionAtGraph}
The section $\charsect\colon S\to C$ induces an isomorphism of the component $\Var(\Fa_{\nu,0})$ with the formal completion $\wh{C_S}^{\Gamma_\charsect}$ of $C_S$ along the graph $\Gamma_\charsect$ of $\charsect$. In particular $\wh{C_S}^{\Gamma_\charsect}$ is canonically isomorphic to $\hat\BD_{\nu,S}$.
\end{lemma}

\begin{proof}
We first consider the formal completion $\wh{C_{A_\nu}}{}^{\Gamma_{\tilde\charsect}}$ of $C_{A_\nu}$ along the graph $\Gamma_{\tilde\charsect}$ of the morphism $\tilde\charsect\colon\Spec A_\nu\to C$. Let $\Spec A\subset C$ be a neighborhood of $\nu$ such that a uniformizing parameter $z$ of $C$ at $\nu$ lies in $A$. Then $I=(a\otimes1-1\otimes a\colon a\in A)\subset A\otimes_{\BF_q}A_\nu$ is the ideal defining $\Gamma_{\tilde\charsect}$ and $\wh{C_{A_\nu}}{}^{\Gamma_{\tilde\charsect}}=\Spf\invlim(A\otimes_{\BF_q}A_\nu)/I^n$. The module $I/I^2$ is free of rank one over $A_\nu=(A\otimes_{\BF_q}A_\nu)/I$ since $C$ is a smooth curve over $\BF_q$. We claim that $I/I^2$ is generated by $z\otimes1-1\otimes z$. Let $a\in A$, let $m\in\BF_q(z)[X]$ be the minimal polynomial of $a$ over $\BF_q(z)$, and multiply it with the least common denominator to obtain the polynomial $F(X,z)\in\BF_q[X,z]$. Note that the least common denominator lies in $\CO_{C,\nu}\mal$ because $a$ is integral over $\BF_q[z]$ near $\nu$. In $A\otimes_{\BF_q}A_\nu[X]$ we use the abbreviations $\zeta:=1\otimes z$ and $\alpha:=1\otimes a$ and we consider the two-variable Taylor expansion of $F$ at $(\alpha,\zeta)$
\[
\TS F(X,z\otimes1)\;\equiv\; F(\alpha,\zeta)+\frac{\partial F}{\partial X}(\alpha,\zeta)\cdot(X-\alpha)+\frac{\partial F}{\partial z}(\alpha,\zeta)\cdot(z\otimes 1-\zeta)\;\mod I^2.
\]
Plugging in $a\otimes 1$ for $X$ yields $F(a\otimes1,z\otimes1)=0$ in addition to $F(\alpha,\zeta)=0$. Since $A$ is unramified over $\BF_q[z]$ at $\nu$ we have $\frac{\partial F}{\partial X}(\alpha,\zeta)\in A_\nu\mal$. This shows that $z\otimes1-1\otimes z$ generates the $A_\nu$-module $I/I^2$.

By Nakayama's Lemma~\cite[Corollary~4.7]{Eisenbud} there is an element $f\in 1+I$ that annihilates the $A\otimes_{\BF_q}A_\nu$-module $I/(z\otimes1-1\otimes z)$. We may replace $\Spec(A\otimes_{\BF_q}A_\nu)$ by the principal open subset $\Spec(A\otimes_{\BF_q}A_\nu)[\tfrac{1}{f}]$ which contains the graph $\Gamma_{\tilde\charsect}$ and on which $I$ is generated by the non-zero divisor $z\otimes1-1\otimes z$. This implies that $I^n/I^{n+1}$ is a free $A_\nu$-module with generator $(z\otimes1-1\otimes z)^n$. Therefore the morphism $A_\nu\dbl t\dbr\to\invlim(A\otimes_{\BF_q}A_\nu)/I^n,t\mapsto z\otimes1-1\otimes z$ induces an isomorphism on the associated graded rings, and hence is itself an isomorphism by \cite[Lemma~10.23]{AM}. 

Now observe that $\Var(\Fa_{\nu,0})$ is the formal scheme on the topological space $S$ with structure sheaf $\CO_S\dbl z\dbr$, and that $\charsect$ identifies the topological spaces $S$ and $\Gamma_\charsect$. Under base change to $S$ this implies that the formal completion $\wh{C_S}^{\Gamma_\charsect}$ has structure sheaf $\CO_S\dbl t\dbr=\CO_S\dbl z-\zeta\dbr$, where we write $t=z-\zeta$. Since $\zeta$ is locally nilpotent in $\CO_S$, the latter is isomorphic to $\CO_S\dbl z\dbr$ proving the lemma.
\end{proof}

\begin{definition}\label{DefGlobalLocalFunctor}
Fix a tuple $\ul\nu:=(\nu_i)_{i=1\ldots n}$ of places on $C$ with $\nu_i\ne\nu_j$ for $i\ne j$. Let $A_{\ul\nu}$ be the completion of the local ring $\CO_{C^n,\ul\nu}$ of $C^n$ at the closed point $\ul\nu$, and let $\BF_{\ul\nu}$ be the residue field of the point $\ul\nu$. Then $\BF_{\ul\nu}$ is the compositum of the fields $\BF_{\nu_i}$ inside $\BF_q^\alg$, and $A_{\ul\nu}\cong\BF_{\ul\nu}\dbl\zeta_1,\ldots,\zeta_n\dbr$ where $\zeta_i$ is a uniformizing parameter of $C$ at $\nu_i$. Let the stack 
\[
\nabla_n\scrH^1(C,\FG)^{\ul\nu}\;:=\;\nabla_n\scrH^1(C,\FG)\whtimes_{C^n}\Spf A_{\ul\nu}
\]
be the formal completion of the stack $\nabla_n\scrH^1(C,\FG)$ along $\ul\nu\in C^n$. Although we will not need it in this article, the reader should note that $\nabla_n\scrH^1(C,\FG)$ is an ind-algebraic stack over $\Spf A_{\ul\nu}$ which is ind-separated and locally of ind-finite type by \cite[\ThmRepNablaH]{AH_Unif}. Set $\BP_{\nu_i}:=\FG\times_C\Spec A_{\nu_i}$ and $\hat\BP_{\nu_i}:=\FG\times_C\Spf A_{\nu_i}$.

Let $(\CG,\charsect_1,\ldots,\charsect_n,\tauGlob)\in \nabla_n\scrH^1(C,\FG)^{\ul\nu}(S)$, that is, $s_i\colon S\to C$ factors through $\Spf A_{\nu_i}$. By Remark~\ref{G-LFrob} we may decompose 
\[
\CG\whtimes_{C_S}(\Spf A_{\nu_i}\whtimes_{\BF_q}S)\cong\coprod_{\ell\in \BZ/(\deg\nu_i)}\CG\whtimes_{C_S}\Var (\mathfrak{a}_{\nu_i,\ell})
\]
into a finite product with components $\CG\whtimes_{C_S}\Var (\mathfrak{a}_{\nu_i,\ell})\in\scrH^1(\hat\BD_{\nu_i},\hat\BP_{\nu_i})$. Using Proposition~\ref{formal torsor prop} we view $\bigl(\CG\whtimes_{C_S}\Var (\mathfrak{a}_{\nu_i,0}),\tauGlob^{\deg\nu_i}\bigr)$ as a local $\BP_{\nu_i}$-shtuka over $S$, where $\tauGlob^{\deg\nu_i}\colon(\sigma^{\deg\nu_i})^*\CL_i\isoto\CL_i$ is the $\BF_{\nu_i}$-Frobenius on the loop group torsor $\CL_i$ associated with $\CG\whtimes_{C_S}\Var (\mathfrak{a}_{\nu_i,0})$. We define the \emph{global-local functor} $\ul{\wh\Gamma}$ by
\begin{eqnarray}\label{EqGlobLocFunctor}
&&\wh\Gamma_{\nu_i}\colon \nabla_n\scrH^1(C,\FG)^{\ul\nu}(S)\longto \Sht_{\BP_{\nu_i}}^{\Spec A_{\nu_i}}(S)\,,\quad(\CG,\tauGlob)\longmapsto \bigl(\CG\whtimes_{C_S}\Var (\mathfrak{a}_{\nu_i,0}),\tauGlob^{\deg\nu_i}\bigr) \quad\text{and}\nonumber\\[2mm]
&&\ul{\wh\Gamma}\;:=\;\prod_i\wh\Gamma_{\nu_i}\colon  \nabla_n\scrH^1(C,\FG)^{\ul\nu}(S)\longto  \prod_i \Sht_{\BP_{\nu_i}}^{\Spec A_{\nu_i}}(S)\,.\label{G-LFunc}
\end{eqnarray}
Note that $\wh\Gamma_{\nu_i}$ and $\ul{\wh\Gamma}$ also transform quasi-isogenies into quasi-isogenies, as can be seen by tracing through the proof of Lemma~\ref{LemmaBL}.
\end{definition}

\begin{remark}\label{RemBH}
Consider the preimages in $\Var (\mathfrak{a}_{\nu_i,\ell})$ of the graphs $\Gamma_{\charsect_j}\subset C_S$ of $\charsect_j$. Since $\nu_i\ne\nu_j$ for $i\ne j$ the preimage of $\Gamma_{\charsect_j}$ is empty for $j\ne i$. Also the preimage of $\Gamma_{\charsect_i}$ equals $\Var (\mathfrak{a}_{\nu_i,0})$ and does not meet $\Var (\mathfrak{a}_{\nu_i,\ell})$ for $\ell\ne 0$. Thus for $\ell\ne0$ the restriction of $\tauGlob$ to $\Var (\mathfrak{a}_{\nu_i,\ell})$ is an isomorphism
\begin{equation}\label{EqIsomBH}
\tauGlob\times\id\colon\s\bigl(\CG\whtimes_{C_S}\Var (\mathfrak{a}_{\nu_i,\ell-1})\bigr)=(\s\CG)\whtimes_{C_S}\Var (\mathfrak{a}_{\nu_i,\ell})\isoto\CG\whtimes_{C_S}\Var (\mathfrak{a}_{\nu_i,\ell})\,.
\end{equation}
This allows to recover $(\CG,\tauGlob)\whtimes_{C_S}(\Spf A_{\nu_i}\whtimes_{\BF_q}S)$ from $\bigl(\CG\whtimes_{C_S}\Var (\mathfrak{a}_{\nu_i,0}),\tauGlob^{\deg\nu_i}\bigr)$ via the isomorphism
\begin{equation}\label{EqIsomBH2}
{\prod_\ell(\tau^\ell\whtimes\id)}\colon \Biggl(\prod_\ell\sigma^{\ell*}\bigl(\CG\whtimes_{C_S}\Var (\mathfrak{a}_{\nu_i,0})\bigr),\left( \raisebox{2.9ex}{$
\xymatrix @C=-0.2pc @R=-0.1pc {
0 \ar@{.}[drdrdr] & & & \!\!\tauGlob^{\deg\nu_i}\!\!\\
1\ar@{.}[drdr]\\
& & & \\
& & 1 & 0\;\;\\
}$}
\right)\Biggr)\isoto (\CG,\tauGlob)\whtimes_{C_S}(\Spf A_{\nu_i}\whtimes_{\BF_q}S).
\end{equation}
Recall from Section~\ref{SectPrelimGTors} that the Weil restriction $\Res_{\BF_{\nu_i}/\BF_q}\wh\CG_{\nu_i}$ of the torsor $\wh\CG_{\nu_i}:=\CG\whtimes_{C_S}(\Spf A_{\nu_i}\whtimes_{\BF_q}S)$ corresponds by Proposition~\ref{formal torsor prop} to an $L^+\wt\BP_{\nu_i}$-torsor $L^+_{\nu_i}(\CG)$. Then $(L^+_{\nu_i}(\CG),\tauGlob\whtimes\id)$ is a local $\wt\BP_{\nu_i}$-shtuka over $S$. We call it the \emph{local $\wt\BP_{\nu_i}$-shtuka associated with $\ul\CG$ at the place $\nu_i$}. By equation~\eqref{EqIsomBH2} there is an equivalence between the category of local $\BP_{\nu_i}$-shtukas over schemes $S\in\Nilp_{A_{\nu_i}}$ and the category of local $\wt\BP_{\nu_i}$-shtukas over $S$ for which the Frobenius $\tauGlob$ is an isomorphism outside $\Var(\Fa_{\nu_i,0})$. (Compare also \cite[Proposition~8.8]{BH1}.)
\end{remark}

\begin{remark}\label{RemTateF}
Note that in a similar way one can associate a local $\wt\BP_\nu$-shtuka $L^+_\nu(\ul\CG)$ with a global $\FG$-shtuka $\ul \CG=(\CG,\tauGlob)$ at a place $\nu$ outside the characteristic places $\nu_i$. Namely $L^+_\nu(\ul\CG)$ is the local $\wt\BP_\nu$-shtuka associated with $\Res_{\BF_\nu/\BF_q}\left(\ul\CG\whtimes_{C_S}(\Spf A_\nu\whtimes_{\BF_q}S)\right)$ by Proposition~\ref{formal torsor prop}. It is \'etale because $\tauGlob$ is an isomorphism at $\nu$. We call $L^+_\nu(\ul\CG)$ the \emph{\'etale local $\wt\BP_\nu$-shtuka associated with $\ul\CG$ at the place $\nu\notin\ul\nu$}. In \cite[\LevelStructure]{AH_Unif} it will become useful for considerations of Tate-modules (Definition~\ref{DefTateFunctor}).

For this purpose we write \mbox{$A_\nu\cong\BF_\nu\dbl z\dbr$}. For every representation $\rho\colon\BP_\nu\to\GL_{r,A_\nu}$ in $\Rep_{A_\nu}\BP_\nu$ we consider the representation $\tilde\rho\in\Rep_{\BF_q\dbl z\dbr}\wt\BP_\nu$ which is the composition of $\Res_{\BF_\nu/\BF_q}(\rho)\colon\wt\BP_\nu\to\Res_{\BF_\nu/\BF_q}\GL_{r,A_\nu}$ followed by the natural inclusion $\Res_{\BF_\nu/\BF_q}\GL_{r,A_\nu}\subset\GL_{\,r\cdot[\BF_\nu:\BF_q]\,,\,\BF_q\dbl z\dbr}$. We set $\ul{\wt\CL}=L^+_\nu(\ul\CG)$ and define $\check{\CT}_{\ul{\wt\CL}}(\rho):=\check{\CT}_{\ul{\wt\CL}}(\tilde\rho):=\check{T}_{\tilde\rho_*\ul{\wt\CL}}$. Then there is a canonical isomorphism $\check{\CT}_{\ul{\wt\CL}}(\rho)\cong\invlim[n]\rho_\ast\bigl(\ul\CG\times_C\Spec A_\nu/(\nu^n)\bigr)^\tauGlob$ of $A_\nu$-modules. This will be used in \cite[Chapter~6]{AH_Unif}.

If $\BF_\nu\subset\CO_S$ there also exists the decomposition \eqref{EqDecomp} and we can associate a local $\BP_\nu$-shtuka $\ul\CL$ with $L^+_\nu(\ul\CG)$. The main difference to Definition~\ref{DefGlobalLocalFunctor} and Remark~\ref{RemBH} is that there is no distinguished component of $\CG\whtimes_{C_S}(\Spf A_\nu\whtimes_{\BF_q}S)$, like the one given by the characteristic section at $\nu_i$. But $\tauGlob$ induces isomorphisms between all components as in \eqref{EqIsomBH}. Therefore we may take any component and the associated local $\BP_\nu$-shtuka $\ul\CL$. Equation~\eqref{EqIsomBH2} shows that over $\BF_\nu$-schemes $S$ we obtain an equivalence between the category of \'etale local $\wt\BP_\nu$-shtukas and the category of \'etale local $\BP_\nu$-shtukas. If $\BP_\nu=\FG\times_C\BD_\nu$ is smooth with connected special fiber, the same is true for $\wt\BP_\nu$ by \cite[Prop.~A.5.9]{CGP}, and then Corollary~\ref{CorEtIsTrivial} applies also to \'etale local $\wt\BP_\nu$-shtukas. There is also a canonical isomorphism of Tate functors $\check{\CT}_{\ul\CL}=\check{\CT}_{L^+_\nu(\ul\CG)}$; compare \cite[Proposition~8.5]{BH1} for more details.
\end{remark}

Like abelian varieties also global $\FG$-shtukas can be pulled back along quasi-isogenies of their associated local $\BP$-shtukas as follows.

\begin{proposition}\label{PropLocalIsogeny}
Let $\ul\CG\in\nabla_n\scrH^1(C,\FG)^{\ul\nu}(\SSS)$ be a global $\FG$-shtuka over $\SSS$ and let $\nu\in C$ be a place. Let $L^+_\nu(\ul\CG)$ be the local $\wt\BP_\nu$-shtuka associated with $\ul\CG$ at $\nu$ in the sense of Remark~\ref{RemBH} (if $\nu\in\ul\nu$), respectively Remark~\ref{RemTateF} (if $\nu\notin\ul\nu$). Let $\tilde f\colon\ul{\wt\CL}{}'_\nu\to L^+_\nu(\ul\CG)$ be a quasi-isogeny of local $\wt\BP_\nu$-shtukas over $\SSS$. If $\nu\in\ul\nu$ we assume that the Frobenius of $\ul{\wt\CL}{}'_\nu$ is an isomorphism outside $\Var(\Fa_{\nu,0})$. If $\nu\notin\ul\nu$ we assume that $\ul{\wt\CL}{}'_\nu$ is \'etale. Then there exists a unique global $\FG$-shtuka $\ul\CG'\in\nabla_n\scrH^1(C,\FG)^{\ul\nu}(\SSS)$ and a unique quasi-isogeny $g\colon\ul\CG'\to\ul\CG$ which is an isomorphism outside $\nu$, such that the local $\wt\BP_\nu$-shtuka associated with $\ul\CG'$ is $\ul{\wt\CL}{}'_\nu$ and the quasi-isogeny of local $\wt\BP_\nu$-shtukas induced by $g$ is $\tilde f$. We denote $\ul\CG'$ by $\tilde f^*\ul\CG$.
\end{proposition}

\begin{remark}\label{RemLocalIsogeny}
Note that if $\nu\in\ul\nu$ then by Remark~\ref{RemBH} there is an equivalence between the category of local $\wt\BP_{\nu}$-shtukas over $S$ for which the Frobenius $\tauGlob$ is an isomorphism outside $\Var(\Fa_{\nu,0})$ and the category of local $\BP_{\nu}$-shtukas over $S$. In particular, if $\wh\Gamma_\nu(\ul\CG)$ is the local $\BP_\nu$-shtuka associated with $\ul\CG$ in Definition~\ref{DefGlobalLocalFunctor}, then every isogeny $f\colon\ul\CL'_\nu\to\wh\Gamma_\nu(\ul\CG)$ corresponds under this equivalence to an isogeny $\tilde f\colon\ul{\wt\CL}{}'_\nu\to L^+_\nu(\ul\CG)$ as in the proposition. We obtain a global $\FG$-shtuka $\tilde f^*\ul\CG$ which we also denote by $f^*\ul\CG$. It satisfies $\wh\Gamma_\nu(f^*\ul\CG)=\ul\CL'_\nu$ and $\wh\Gamma_\nu(g\colon f^*\ul\CG\to\ul\CG)=f$.
\end{remark}

\begin{proof}[Proof of Proposition~\ref{PropLocalIsogeny}]
Let us set $\ul{\CG}:=(\CG,\tauGlob)$. Let $(\dot{\CG},L^+_\nu(\CG), \phi)$ be the triple for the place $\nu$ associated with the $\FG$-torsor $\CG$ by Lemma~\ref{LemmaBL}. Thus $L^+_\nu(\ul\CG)=(L^+_\nu(\CG),\tauGlob)$. We also set $\ul{\wt\CL}{}'_\nu=({\wt\CL}{}'_\nu,\tauGlob')$. Now the triple $(\dot{\CG}_, {\wt\CL}{}'_\nu, \tilde f^{-1}\phi)$ defines a $\FG$-bundle $\CG'$ over $C_\SSS$ which coincides with $\CG$ over $C'_S$ and inherits the Frobenius automorphism $\tauGlob$ from $\ul\CG$ over $C'_\SSS\setminus \bigcup_i\Gamma_{\charsect_i}$. If $\nu\notin\ul\nu$ then this $\tauGlob$ extends to an isomorphism over $\{\nu\}\times_{\BF_q}S$ because $\ul{\wt\CL}{}'_\nu$ is \'etale. If $\nu\in\ul\nu$ then $\tauGlob$ extends to an isomorphism over $C_S\setminus\bigcup_i\Gamma_{\charsect_i}$ because $\tauGlob'$ is an isomorphism outside $\Var(\Fa_{\nu,0})$. This defines the global $\FG$-shtuka $\ul\CG'\in\nabla_n\scrH^1(C,\FG)^{{\ul\nu}}(\SSS)$. The quasi-isogeny $g$ is obtained from the identification $\dot\CG'=\dot\CG$. It has the desired properties.
\end{proof}

\bigskip

Finally we want to prove rigidity for quasi-isogenies of global $\FG$-shtukas. This is the global counterpart of Proposition~\ref{PropRigidityLocal} and fits into the analogy between abelian varieties and global $\FG$-shtukas. It only holds over schemes $S\in\Nilp_{A_{\ul\nu}}$, similarly to rigidity for abelian varieties which only holds over schemes $S\in\Nilp_{\BZ_p}$.

\begin{proposition}\label{PropGlobalRigidity}
\emph{(Rigidity of quasi-isogenies for global $\FG$-shtukas)}
Let $S$ be a scheme in $\Nilp_{A_{\ul\nu}}$ and let $j \colon  \bar{S}\rightarrow S$ be a closed immersion defined by a sheaf of ideals $\CI$ which is locally nilpotent.
Let $\ul{\CG}=(\CG,\tauGlob)$ and $\ul{\CG}'=(\CG',\tauGlob')$ be two global $\FG$-shtukas over $S$. Then
$$
\QIsog_S(\ul{\CG}, \ul{\CG}') \longto \QIsog_{\bar{S}}(j^*\ul{\CG}, j^*\ul{\CG}') ,\quad f \mapsto j^*f
$$
is a bijection of sets. $f$ is an isomorphism at a place $\nu\notin\ul\nu$ if and only if $j^*f$ is an isomorphism at $\nu$.
\end{proposition}

Note that the last assertion need not be true for places $\nu\in\ul\nu$. This is similar to lifts of quasi-isogenies between abelian varieties over schemes $S\in\Nilp_{\BZ_p}$ which can acquire additional ``poles'' at $p$.

\begin{proof}[Proof of Proposition~\ref{PropGlobalRigidity}]
It suffices to treat the case where $\CI^q = (0)$. In this case the morphism $\sigma_S$ factors through $j\colon \bar{S}\to S$
$$
\sigma_S = j \circ \sigma'\colon  S \xrightarrow{\:\,\sigma'\,} \bar{S} \xrightarrow{\;\; j\:\,} S.
$$
Since $C_S\setminus\cup_i\Gamma_{\charsect_i}\supset C_S\setminus\ul\nu\times_{\BF_q}S$, the morphism $\tau$ defines a quasi-isogeny $\tau\colon \sigma'^*j^*\ul\CG=\sigma_S^*\ul\CG \rightarrow \ul\CG$ which is an isomorphism outside $\ul\nu\times_{\BF_q}S$ and similarly for $\ul\CG'$. We fix a finite closed subset $D\subset C$ which contains all $\nu_i$ and consider quasi-isogenies which are isomorphisms outside $D$. Then the following diagram
$$
\CD
\CG|_{C_S\setminus D_S} @>{f}>\cong>\CG'|_{C_S\setminus D_S}\\
@A{\tau}A{\cong}A @A{\tau'}A{\cong}A\\
\sigma_S^*\CG|_{C_S\setminus D_S} @>{\sigma'^*(j^*f)}>\cong>\sigma_S^*\CG'|_{C_S\setminus D_S}
\endCD
$$
allows to recover $f$ from $j^*f$ and this proves the bijectivity.

If $\nu\notin\ul\nu$ we take a subset $D\subset C$ which does not contain $\nu$. Chasing through the diagram again shows that $f$ is an isomorphism at $\nu$ if and only if $j^*f$ is an isomorphism at $\nu$.
\end{proof}

\subsection{The Analog of the Serre--Tate Theorem}\label{SectSerre-Tate}
\bigskip

The Serre--Tate Theorem relates the deformation theory of an abelian variety in characteristic $p$ with the deformation theory of the associated $p$-divisible group. In this section we introduce the analogous situation over function fields and prove the analogous theorem relating the deformation theory of a global $\FG$-shtuka to the deformation theory of the associated $n$-tuple of local $\BP_{\nu_i}$-shtukas via the global-local functor.

\bigskip
Let $S$ be in $\Nilp_{A_{\ul\nu}}$ and let $j\colon  \bar{S}\to S$ be a closed subscheme defined by a locally nilpotent sheaf of ideals $\CI$. Let $\bar{\ul{\CG}}$ be a global $\FG$-shtuka in $\nabla_n\scrH^1(C,\FG)^{\ul \nu}(\bar{S})$. The category $Defo_S(\bar{\ul{\cG}})$ of lifts of $\bar{\ul{\cG}}$ to $S$ consists of all pairs $(\ul{\cG},  \alpha\colon j^*\ul{\cG}\isoto \bar{\ul{\cG}})$ where $\ul{\CG}$ belongs to $\nabla_n\scrH^1(C,\FG)^{\ul \nu}(S)$, where $\alpha$ is an isomorphism of global $\FG$-shtukas over $\bar{S}$, and where morphisms are isomorphisms between the $\ul{\cG}$'s that are compatible with the $\alpha$'s.

Similarly for a local $\BP$-shtuka $\ul{\bar\CL}$ in $\Sht_{\BP}^{\BD}(\bar{S})$ we define the category of lifts $Defo_S(\ul{\bar\CL})$ of $\ul{\bar\CL}$ to $S$. Notice that according to the rigidity of quasi-isogenies (Propositions~\ref{PropGlobalRigidity} and \ref{PropRigidityLocal}) all Hom-sets in these categories contain at most one element.

\bigskip

\begin{theorem}\label{Serre-Tate} 
Let $\bar{\ul{\cG}}:=(\bar{\CG},\bar{\tau})$ be a global $\FG$-shtuka in $\nabla_n\scrH^1(C,\FG)^{\ul \nu}(\bar{S})$. Let $(\ul{\bar\CL}_i)_i=\wh{\ul\Gamma}(\bar{\ul{\cG}})$. Then the functor 
$$
Defo_S(\bar{\ul{\cG}})\longto \prod_i Defo_S(\ul{\bar\CL}_i)\,,\quad \bigl(\ul\CG,\alpha)\longmapsto(\ul{\wh\Gamma}(\ul\CG),\ul{\wh\Gamma}(\alpha)\bigr)
$$ 
induced by the global-local functor, is an equivalence of categories.
\end{theorem}

\begin{proof} 
We proceed by constructing the inverse of the above functor. It suffices to treat the case where $\CI^q = (0)$. In this case the morphism $\sigma_S$ factors through $j\colon \bar{S}\to S$
$$
\sigma_S = j \circ \sigma'\colon  S \xrightarrow{\:\,\sigma'\,} \bar{S} \xrightarrow{\;\; j\:\,} S.
$$
Let $(\ul\CL_i,\hat{\alpha}_i\colon j^*\ul\CL_i \isoto \ul{\bar\CL}_i)_i$ be an object of $\prod_i Defo_S(\ul{\bar\CL}_i)$. Consider the global $\FG$-shtuka $\ul{\wt{\cG}}:=(\wt\CG,\tilde\tau):= \sigma'^*\ul{\bar{\CG}}$ over $S$. Since $C_S\setminus\cup_i\Gamma_{\charsect_i}\supset C_S\setminus\ul\nu\times_{\BF_q}S$, the morphism $\bar{\tau}$ defines a quasi-isogeny $\bar{\tau}\colon j^* \ul{\wt{\cG}} \rightarrow \bar{\ul{\cG}}$ which is an isomorphism outside the graphs of the characteristic sections as one sees from the following diagram
$$
\CD
\sigma_S^*j^*{\wt{\cG}}@={\left(\sigma_S^2\right)^\ast}{\bar{\cG}}@>{\sigma_S^*\bar{\tau}}>>\sigma_S^*\bar{\cG}\\
@Vj^*\tilde\tau VV @V{\sigma_S^*{\bar{\tau}}}VV @V{\bar{\tau}}VV\\
j^*\wt{\cG}@=\sigma_S^*\bar{\cG}@>{\bar{\tau}}>>\;\bar{\cG}\,.
\endCD
$$
We write $\ul{\wh\Gamma}(\ul{\wt\CG})=(\ul{\wt\CL}_i)_i$ and $\ul{\bar\CL}_i=(\bar\CL_i,\hat{\bar\tau}_i)$. We compose ${\hat{\bar{\tau}}_i}^{-1}$ with $\hat{\alpha}_i$ to obtain the quasi-isogenies $\hat{\bar{\gamma}}_i:= {\hat{\bar{\tau}}_i}^{-1} \circ \hat{\alpha}_i \colon j^*\ul\CL_i \rightarrow j^*\ul{\wt\CL}_i$. By rigidity of quasi-isogenies (Proposition~\ref{PropRigidityLocal}) they lift to quasi-isogenies $\hat{\gamma}_i\colon \ul\CL_i \rightarrow \ul{\wt\CL}_i$ with $j^*\hat\gamma_i=\hat{\bar{\gamma}}_i$. We put $\ul{\CG}:=\hat\gamma_n^*\circ\ldots\circ\hat\gamma_1^*\ul{\wt\CG}$ (see Remark~\ref{RemLocalIsogeny}) and recall that there is a quasi-isogeny $\gamma\colon  \ul{\cG} \rightarrow \ul{\wt{\cG}}$ of global $\FG$-shtukas with $\ul{\wh\Gamma}(\gamma)=(\hat\gamma_i)_i$ which is an isomorphism outside $\ul\nu$, see Proposition~\ref{PropLocalIsogeny}. We may now define the functor 
$$
\prod_i Defo_S(\ul{\bar\CL}_i)\rightarrow Defo(\bar{\ul{\cG}})
$$
by sending $(\ul\CL_i,\hat{\alpha}_i\colon j^*\ul\CL_i \rightarrow \ul{\bar\CL}_i)_i$ to $(\ul{\cG},\bar{\tau}\circ j^*\gamma)$. The quasi-isogeny $\alpha:=\bar{\tau}\circ j^\ast \gamma$ is an isomorphism outside the graphs of the $\charsect_i$ by construction, and also at these graphs because $\ul{\wh\Gamma}(\alpha)=(\hat\alpha_i)_i$. It can easily be seen by the above construction that these functors are actually inverse to each other.
\end{proof}

%%%%%%%%%%%%%%%%%%%%%%%%%%%%%%%%%%%%%%%%%%%%%%%%%%%%%%%%%%%%%%%%%%%%%%
%
%    Bibliography
%
%%%%%%%%%%%%%%%%%%%%%%%%%%%%%%%%%%%%%%%%%%%%%%%%%%%%%%%%%%%%%%%%%%%%%%

\vfill

\begin{minipage}[t]{0.5\linewidth}
\noindent
Esmail Arasteh Rad\\
Universit\"at M\"unster\\
Mathematisches Institut \\
Einsteinstr.~62\\
D -- 48149 M\"unster
\\ Germany
\\[1mm]
\end{minipage}
\begin{minipage}[t]{0.45\linewidth}
\noindent
Urs Hartl\\
Universit\"at M\"unster\\
Mathematisches Institut \\
Einsteinstr.~62\\
D -- 48149 M\"unster
\\ Germany
\\[1mm]
\href{http://www.math.uni-muenster.de/u/urs.hartl/index.html.en}{www.math.uni-muenster.de/u/urs.hartl/}
\end{minipage}

\end{document}